\documentclass[reqno,oneside]{amsart}
\usepackage[margin=3cm]{geometry}
\usepackage{amssymb,amsmath}
\usepackage{amsfonts}
\usepackage{rotating}
\usepackage{graphicx}
\usepackage[boxed,oldcommands,norelsize]{algorithm2e}
\usepackage[numbers,sort&compress]{natbib}
\usepackage{multirow}
\usepackage{hhline}
\usepackage{colortbl}
\newcolumntype{L}{>{\centering\arraybackslash}m{0.46\textwidth}}
\usepackage{arydshln}

\usepackage[dvips]{epsfig}
\usepackage{algcompatible}
\usepackage{subfigure}
\usepackage{bbm}

\newtheorem{theorem}{Theorem}[section]
\newtheorem{lemma}[theorem]{Lemma}

\newtheorem{corollary}[theorem]{Corollary}
\newtheorem{definition}[theorem]{Definition}

\newtheorem{remark}[theorem]{Remark}

\def\half{\frac{1}{2}}

\def\Fig#1{\ref{fig-#1}}

\def\jumpl{\lbrack\!\lbrack}
\def\jumpr{\rbrack\!\rbrack}
\def\jump#1{\jumpl #1 \jumpr}

\def\mean#1{\{\! \!\{ #1\}\! \!\}}

\def\jumpl{\lbrack\!\lbrack}
\def\jumpr{\rbrack\!\rbrack}
\def\jump#1{\jumpl #1 \jumpr}
\def\vh1{V_{h}^1}

\def\thn{\mathcal{T}_h}
\def\thm{\mathcal{T}_h}

\def\half{\frac{1}{2}}

\def\Fig#1{Fig. \ref{fig-#1}}

\def\jumpl{\lbrack\!\lbrack}
\def\jumpr{\rbrack\!\rbrack}
\def\jump#1{\jumpl #1 \jumpr}

\def\mean#1{\{\! \!\{ #1\}\! \!\}}

\def\jumpl{\lbrack\!\lbrack}
\def\jumpr{\rbrack\!\rbrack}
\def\jump#1{\jumpl #1 \jumpr}
\def\vh{V_h}

\def\sgn#1{\textrm{sign}(#1)}
\def\ssgn#1{\textrm{sign}_{\tau_h}(#1)}
\def\half{\frac{1}{2}}

\def\Fig#1{Fig. \ref{fig-#1}}

\def\jumpl{\lbrack\!\lbrack}
\def\jumpr{\rbrack\!\rbrack}
\def\jump#1{\jumpl #1 \jumpr}

\def\mean#1{\{\! \!\{ #1\}\! \!\}}

\def\jumpl{\lbrack\!\lbrack}
\def\jumpr{\rbrack\!\rbrack}
\def\jump#1{\jumpl #1 \jumpr}
\def\vh{V_h}

\def\sgn#1{\textrm{sign}(#1)}

\def\half{\frac{1}{2}}

\def\Fig#1{Fig. \ref{fig-#1}}

\def\jumpl{\lbrack\!\lbrack}
\def\jumpr{\rbrack\!\rbrack}
\def\jump#1{\jumpl #1 \jumpr}

\def\mean#1{\{\! \!\{ #1\}\! \!\}}

\def\jumpl{\lbrack\!\lbrack}
\def\jumpr{\rbrack\!\rbrack}
\def\jump#1{\jumpl #1 \jumpr}

\def\vm1{V_{m}^1}
\def\vm{V_m}

\def\sgn#1{\textrm{sign}(#1)}

\def\setfaces{\mathcal{E}_h}
\def\setintfaces{\mathcal{E}_h^0}

\def\Kp{{K'}}
\def\setneighbours{\mathcal{N}_h}
\def\setbcneighbours{\setneighbours^\partial }
\def\spacebilinearform{K_h}
\def\pertbform{\tilde{K}_h}
\def\pertwbc{\tilde{B}_h}

\def\rhsg{G_h}
\def\shap{\varphi}

\def\domain{\Omega}

\def\dofsupport{\mathcal{N}_h}

\def\graphvisc{\nu}

\def\graphviscosityab{\graphvisc_{ab}}
\def\nubc{{\nu}^\partial}
\def\unityfunction{\mathbbm{1}}
\def\lumpedbform#1{M_h^L(#1)}

\def\l2norm#1{\|#1\|}
\def\tstep{n}
\def\ntsteps{N^t}

\def\smoothalpha{\tau_h}
\def\smoothperturb{\sigma_h}
\def\greatabs#1{\left|#1\right|_{1,\smoothalpha}}
\def\lowabs#1{\left|#1\right|_{2,\smoothalpha}}
\def\quotient{\zeta}
\def\smoothquotient{Z}
\def\smax{{\rm max}_{\smoothperturb}}
\def\smoothquot{\gamma_h}
\def\smoothalphac{\tau}
\def\smoothperturbc{\sigma}
\def\smoothquotc{\gamma}

\newcommand{\JB}[1]{{#1}}

\newcommand{\flux}[1][]{
	\ifx\empty#1\empty
	\boldsymbol{f}
	\else
	\boldsymbol{f}(#1)
	\fi
}

\newcommand{\gradient}{\boldsymbol{\nabla}}

\newcommand{\force}{g}
\newcommand{\bc}{\overline{u}}
\newcommand{\bcprojection}{\bc_h}
\newcommand{\rhsbc}[1]{B_h(\bcprojection;#1)}
\newcommand{\cip}{c^{\rm ip}}

\newcommand{\M}{\boldsymbol{\mathsf{M}}}
\newcommand{\G}{\boldsymbol{\mathsf{G}}}
\newcommand{\A}{\boldsymbol{\mathsf{K}}}
\newcommand{\T}{\boldsymbol{\mathsf{T}}}
\newcommand{\B}{\boldsymbol{\mathsf{B}}}

\newcommand{\Mbar}{\boldsymbol{\mathsf{\tilde{M}}}}
\newcommand{\Abar}{\boldsymbol{\mathsf{\tilde{K}}}}
\newcommand{\Bbar}{\boldsymbol{\mathsf{\tilde{B}}}}

\newcommand{\ab}{{ab}}
\newcommand{\Mt}{\boldsymbol{\mathsf{\tilde{M}}}}

\newcommand{\At}{\boldsymbol{\mathsf{\tilde{K}}}}
\newcommand{\Bt}{\boldsymbol{\mathsf{\tilde{B}}}}

\newcommand{\x}{\boldsymbol{x}}
\newcommand{\n}{\boldsymbol{n}}
\newcommand{\rr}{\boldsymbol{r}}
\newcommand{\ru}{\hat{\boldsymbol{r}}}
\newcommand{\der}{{\rm d}}
\newcommand{\unknp}{u_h^{n+1}}

\newcommand{\unkn}{u_h^{n}}
\newcommand{\sym}{{\rm sym}}

\def\Viscop{D}

\begin{document}

\title[Differentiable monotonicity-preserving schemes for dG methods]{Differentiable monotonicity-preserving schemes for discontinuous Galerkin methods on arbitrary meshes}

\author{ Santiago Badia$^{\dag\ddag}$  \and Jes\'{u}s Bonilla$^{\dag\ddag}$ \and Alba Hierro$^{\dag\ddag}$ }
\thanks{
	$^\dag$ Centre Internacional de M\`etodes Num\`erics en Enginyeria (CIMNE)
	\@, Parc
	Mediterrani de la Tecnologia, UPC, Esteve Terradas 5, 08860 Castelldefels,
	Spain (\{sbadia,jbonilla,ahierro\}@cimne.upc.edu). \\ 
	\indent$^\ddag$
	Universitat Polit\`ecnica de Catalunya, Jordi Girona 1-3, Edifici C1, 08034
	Barcelona, Spain.
}

\renewcommand{\thefootnote}{\fnsymbol{footnote}}

%

\renewcommand{\thefootnote}{\arabic{footnote}}

\begin{abstract}

This work is devoted to the design of interior penalty discontinuous
Galerkin (dG) schemes that preserve maximum principles at the
discrete level for the steady transport and convection-diffusion
problems and the respective transient problems with implicit time
integration. Monotonic schemes that combine explicit time stepping
with dG space discretization are very common, but the design of such
schemes for implicit time stepping is rare, and it had only been
attained so far for 1D problems. The proposed scheme is based on \JB{a
piecewise linear dG discretization supplemented with }an
artificial diffusion that linearly depends on a shock detector that
identifies the troublesome areas. In order to define the new shock
detector, we have introduced the concept of \emph{discrete local
	extrema}. The diffusion operator is a graph-Laplacian, instead of
the more common finite element discretization of the Laplacian
operator, which is essential to keep monotonicity on general meshes
and in multi-dimension. The resulting nonlinear stabilization is
non-smooth and nonlinear solvers can fail to converge. As a result, we
propose a smoothed (twice differentiable) version of the nonlinear
stabilization, which allows us to use Newton with line search
nonlinear solvers and dramatically improve nonlinear convergence. A
theoretical numerical analysis of the proposed schemes show that
they satisfy the desired monotonicity properties. Further, the
resulting operator is Lipschitz continuous and there exists at least
one solution of the discrete problem, even in the non-smooth
version. We provide a set of numerical results to support our
findings.
\end{abstract}

\maketitle

\tableofcontents
\section{Introduction}

The transport problem is one of many problems that might satisfy a
maximum principle (MP) or a positivity property. However, its
numerical discretization may violate these properties at the discrete
level. These violations arise in the form of local spurious
oscillations near sharp layers of the solution. Such oscillations
break the MP of the continuous problem. 
\JB{For steady problems with no source term, the MP implies that the extrema of the solution are on the boundary of the domain}; they
are bounded by the boundary and the initial solution extrema in the
transient case.

\JB{
Many authors have focused on developing accurate schemes that inherit the MP at the discrete level, i.e. discrete maximum principle (DMP) preserving schemes. To this end, several approaches have been used. In the case of explicit time integration combined with finite volumes or discontinuous Galerkin (dG) methods, the schemes are usually based on either slope or flux limiters, or special reconstruction algorithms. These methods are widely present in literature and already well understood (see, e.g., \cite{leveque_book_2002}).
}

\JB{
For implicit time integration and continuous Galerkin (cG)
finite element space discretization, methods attaining DMPs are not as well understood as the previous ones. However, several schemes have been developed to date.
In this case, most of the approaches are based on adding an artificial diffusion operator. Then, in order to maintain high-order convergence rates in smooth regions, this operator is scaled such that it vanishes in smooth regions and it is active in the vicinity of sharp layers. Depending on how this activation is controlled, one may distinguish among residual-based, entropy-based, and fluctuation-based schemes. The first DMP-preserving schemes were residual-based (see e.g. \cite{mizukami_petrov-galerkin_1985,burman_nonlinear_2002}). Afterwards, fluctuation based schemes were developed \cite{burman_edge_2004,burman_stabilized_2005,burman_nonlinear_2007,badia_monotonicity-preserving_2014,burman_monotonicity_2015}. These schemes are based on computing \emph{a priori} an artificial diffusion that ensures DMP preservation. The artificial diffusion is activated based on a so-called shock detector, usually based on the unknown gradient jumps across elements. Lately, Guermond and co-workers have proposed a similar approach for hyperbolic problems, but using an alternative detector based on the entropy production  \cite{guermond_maximum-principle_2014,guermond_second-order_2014}. The more recent fluctuation-based schemes compute the amount of diffusion required to preserve the DMP in a way that resembles Algebraic Flux Correction (AFC) techniques \cite{dmitri_kuzmin_new_2015,barrenechea_2016,barrenechea_analysis_2016, kuzmin_gradient-based_2016,badia_monotonicity-preserving_2016}.
}
\JB{ 
The reader might refer to \cite{kuzmin_algebraic_2005} and the references therein for more insights about AFC.
}

\JB{
In the case at hand, the dG space discretization of steady problems and transient problems, the situation is much less understood. An attempt to develop implicit DMP-preserving dG schemes has been proposed in
\cite{badia_discrete_2015}, but even though a DMP enjoying artificial
diffusion method can be constructed for the 1D problem, the extension
to the multi-dimensional case fails to enjoy such property.}
The objective of this work is \emph{to design a multidimensional
DMP-preserving dG method on arbitrary meshes for both implicit time
integration and steady problems}. Furthermore, {we propose a
linearity preserving and differentiable method. This latter
property} is particularly important for improving the convergence of
the nonlinear solver, as shown in \cite{badia_monotonicity-preserving_2016}. 

\JB{
In order to do so, we propose a stabilization method based on the
following four key ingredients:
\begin{enumerate}
	\item  A \emph{shock detector} that only activates the artificial diffusion in regions around shock. 
	As previously said, a shock detector restricts the application of the 
	stabilization to regions where the solution presents shocks 
	or sharp layers, and is the key ingredient to obtain a high-order
	stabilization method;
	\item The \emph{amount of diffusion}
	added to ensure the DMP. We motivate it using similar ideas behind the AFC low-order scheme construction (see  \cite{kuzmin_flux_2002,kuzmin_algebraic_2005});
	\item The \emph{discrete diffusion operator} in order to keep the DMP on arbitrary meshes. Guermond and co-workers \cite{guermond_second-order_2014,guermond_maximum-principle_2014} have proposed to use graph-theoretic artificial diffusion operators, instead of the classical PDE-based ones. This strategy has already been used in  \cite{dmitri_kuzmin_new_2015,kuzmin_gradient-based_2016,badia_monotonicity-preserving_2016};
	\item For transient problems, a perturbation of the mass matrix is required
	to obtain a local extremum diminishing (LED) scheme.
\end{enumerate}}


This work is structured as follows. In Sect. \ref{sec-problem_paper4},
we introduce the problem to solve, the notation, and the
discretization of the problem in space using the interior penalty dG
method. Then, in Sect. \ref{sec-DMP_paper4}, we state a novel
definition of the DMP property for dG methods, by introducing the
concept of discrete local extrema. In Sect. \ref{sec-dmpproof}, we
propose a scheme that fulfills such property. Lipschitz
continuity and existence of solutions are proved in Sect
\ref{sec-ex}.
A discussion about the importance of smoothing the
computation of the shock capturing terms and some tests to choose the
optimal values of the smoothing parameters are developed in
Sect. \ref{app-smoothing}. Finally, numerical experiments show the
performance of the method in Sect. \ref{sec-numexp4}, and some
conclusions are drawn in Sect. \ref{sec-conclusions}.

\section{The Convection-diffusion problem and its discretization}
\label{sec-problem_paper4}

\def\betab{\boldsymbol{\beta}} We consider a transient
convection-diffusion problem with Dirichlet boundary conditions:
\begin{equation}
\label{eq-strongform4}
\left\{
\begin{array}{rcll}
\partial_t u + \gradient\cdot(\betab u) - \gradient\cdot(\mu\gradient u)  &=& \force & \rm{in}\ \Omega\times[0,T], \\
u(x,t)&=&\bc(x,t) & \rm{on}\ \partial\Omega\times[0,T],\\
u(x,0)&=&u_0(x) & x\in\Omega.
\end{array}\right.
\end{equation}
The domain $\Omega$ is an open, bounded, connected subset of
$\mathbb{R}^d$ with a Lipschitz boundary $\partial\Omega$, where $d$
is the space dimension, \JB{$\betab=\betab(\x)$} is the convective velocity, which is assumed to be divergence-free, and \JB{$\mu\geq 0$} is a constant diffusion. 
Even though we have considered Dirichlet boundary conditions in
the statement of problem \eqref{eq-strongform4}, i.e.,
$\partial\domain\equiv\partial\domain_D$, Neumann boundary
conditions can also be considered straightforwardly. In
the case of pure convection ($\mu=0$), boundary conditions are only
imposed on the inflow boundary
$\partial\Omega^-\doteq\{\x\in\partial\Omega :
\betab\cdot\n_{\partial \Omega}<0\}$, where $\n_{\partial \Omega}$ is
the outward-pointing unit normal. Further, we define the outflow
boundary as
$\partial\Omega^+\doteq\partial\Omega\backslash\partial\Omega^-$. Below,
we also consider the steady case, by eliminating the time derivative
term. 

\subsection{Notation}

Let $\thn = \{K\}$ be a partition of $\overline{\Omega}$ formed by elements $K$
of characteristic length $h_K$.  For quasi-uniform meshes, we can
define a global characteristic length of the mesh $h$. \JB{We
	denote by $\x_i$ the coordinates of vertex $i$ and by $\mathcal{V}_h$ 
	the set of vertices of the partition. We also 
	define $\mathcal{V}_h(K) \doteq \{ i \in \mathcal{V}_h \, : \, \x_i
	\in K \}$.}

Given the mesh $\thn$, the non-empty intersection $F=\partial
K\cap\partial \Kp$ of two neighbor elements $K,\Kp\in\thn$ is called
an interior facet of $\thn$ if it is a subdomain of dimension
$d-1$. The set of all the interior facets is denoted by
$\setintfaces$. On the other hand, the non-empty intersection
$F=\partial K \cap \partial\Omega^-$ of an element $K\in\thn$ on the
boundary with the boundary of the domain is called an inflow boundary
facet (analogously for $\partial \Omega^+$ and outflow boundary
facets). The set of inflow boundary facets is denoted by
$\setfaces^-$, and the set of outflow boundary facets is denoted by
$\setfaces^+$. (We assume that the finite element partition is conforming with the
inflow and outflow boundaries.) The set of all the facets is denoted
by $\setfaces \doteq \setintfaces \cup \setfaces^+ \cup
\setfaces^-$. For any facet $F\in\setintfaces$, we represent with
$K_F^+$ and $K_F^-$ the only two neighbor elements such that $\partial
K_F^+\cap \partial K_F^- = F$. In addition, we call $n_F^+$ and
$n_F^-$ the unitary normal to facet $F$ outside $K_F^+$ and $K_F^-$,
respectively. Given a facet $F$, we can also define the characteristic
facet length $h_F$.

On tetrahedral (or triangular) meshes, the discrete space considered
henceforth is the discontinuous space of piecewise linear functions $
\vh=\{v_h \, : \, v_h|_K \in \mathbb{P}_1(K)$ $\forall K \in \thn\}$,
where $\mathbb{P}_1(K)$ is the space of linear polynomials in
$K$. For hexahedral (or quadrilateral) meshes, $ \vh=\{v_h \, :
\, v_h|_K \in \mathbb{Q}_1(K)$ $\forall K \in \thn\}$, where
$\mathbb{Q}_1(K)$ is the tensor product space of piecewise linear
1D polynomials. In addition, we represent the space of traces of $V_h$ 
\JB{on $\partial\Omega$} as $V_h|_{\partial \Omega}$.

In order to define dG spaces, we use the nodal set as the Cartesian
product of element vertices, i.e., $\mathcal{N}_h = \Pi_{K \in
	\mathcal{T}_h} \mathcal{V}_h(K)$. Thus, every node $a \in
\mathcal{N}_h$ can also be represented as a pair $(i,K)$, with $K \in
\mathcal{T}_h$ and $i \in \mathcal{V}_h(K)$. \JB{Therefore,
	more than one interior node might be placed at the same coordinates. Indeed, if $n$ elements have a vertex on $\x_i$, there will be $n$ nodes at $\x_i$, each one  with its own degree of freedom.}
Given the node $a \in
\mathcal{N}_h$, its coordinates are represented with $\x_a$,
{$\Omega_a \doteq \{ K \in \thn \, : \, \x_a\in K \}$}
is its support, and $\dofsupport(a)= \{ b \in \dofsupport \, : \, \x_b
\in \Omega_a\}$ is the set of nodes \emph{connected} to
$a$. \JB{Notice that $a$ itself is included in $\dofsupport(a)$.} We define the set of boundary nodes $\mathcal{N}^\partial_h \doteq \{ a
\in \mathcal{N}_h \, : \, \x_a \in \partial \Omega \}$, and
$\dofsupport^\partial (a) \doteq \dofsupport(a) \cap
\mathcal{N}_h^\partial$.

The functions $v_h\in\vh$ can be expressed as a linear combination of
the basis $\{\varphi_a\}_{a\in\dofsupport}$, where $\varphi_a$
corresponds to the shape function of node $a$. It is defined as
follows. Given $a \in \mathcal{N}_h$ and its corresponding
vertex-element pair $(i,K)$, we define $\varphi_a$ as the elementwise
(bi)linear function such that $\varphi_a(\x_a)\vert_K = 1$ and
$\varphi_a(\x_b)\vert_K=0$ for $b\neq a$, and
$\varphi_a\vert_{K^\prime} \equiv 0$ for $K^\prime\neq K$. Any
function $v_h\in\vh$ is double-valued on $\setintfaces$ and
single-valued on $\partial \domain$. Thus, $v_h\in\vh$ can be
expressed as $v_h = \sum_{a\in\dofsupport} v_a \shap_a$. Moreover we
consider $v_h^K$ as the restriction of $v_h$ into $K$.

Given $v_h\in \vh$,  we can define the common concepts of average $\mean{\cdot}$ and jump $\jump{\cdot}$ on an interior point $\x$ of a facet $F\in \setintfaces$ as follows: 
$$\mean{v_h}(\x)=
\frac{1}{2}\left(v_h^{K_F^+}(\x)+v_h^{K_F^-}(\x)\right),\quad
\jump{v_h}(\x) = v_h^{K_F^+}(\x)\n_F^++v_h^{K^-_F}(\x)\n_F^-, $$ where
$\n_F^+$ (resp. $\n_F^-$) is the outward normal with respect to $K^+$
(resp. $K^-$) on $F$; we use $\n_F$ on boundary facets and in places
where the sign is not relevant. On boundary facet points $\x\in F$,
$F\subset\partial\Omega$, we define $\mean{v_h}(\x)=v_h^{K_F^+}(\x)$,
$\jump{v_h}(\x) = v_h^{K_F^+}(\x) \n_F^+(\x)$.

We will use standard notation for Sobolev spaces (see, e.g., \cite{brezis_functional_2010}). In particular, the $L^2(\omega)$ scalar product will be denoted by $(\cdot,\cdot)_\omega$ for some $\omega \subset \Omega$, but the domain subscript is omitted for $\omega \equiv \Omega$. The $L^2(\Omega)$ norm is denoted by $\l2norm{\cdot}$. We will denote by $\unityfunction$ the function that is equal to $1$ in $\domain$; $\unityfunction(\x) = 1$ $\forall \x\in\domain$.

\subsection{Weak form and interior penalty dG approximation}\label{sec:weak-form-and-the-interior-penalty-dg-approximation}
The stabilized dG bilinear form for the transport problem proposed in \cite{brezzi_discontinuous_2004} combined with the interior penalty (IP) method for the viscosity term reads as:
\begin{align}\label{eq-dscrtpbm}
{\rm Find}\quad u_h\in \vh \quad \hbox{such that }\quad (\partial_t u_h,v_h) + \spacebilinearform(u_h,v_h) = \rhsg(v_h) + \rhsbc{v_h} \quad \forall v_h\in \vh,
\end{align}
with 
\begin{align}\label{eq-bform4}
\begin{split}
\spacebilinearform(u_h,v_h) \doteq &   \displaystyle\sum\limits_{K \in \thm} \int_K \left(\mu \gradient u_{h} \cdot \gradient v_{h} -  u_h  \betab \cdot \gradient v_h\right)\\
& + \sum\limits_{F\in \setfaces}  \int_{F} \mu \left(-\jump{u_h} \cdot \mean{ \gradient v_h} - \mean{\gradient u_h} \cdot \jump{v_h}  + c^{\rm ip} {h}_F^{-1} \jump{  u_h} \cdot \jump{  v_h}\right)\\
& + \sum\limits_{F\in \setfaces^+ \cup\setfaces^0}
\int_{F}   \mean{\betab u_h} \cdot \jump{v_h} + \sum\limits_{F\in \setintfaces}  \int_{F}  \frac{|\betab\cdot \n_F|}{2}  \jump{u_h} \cdot \jump{v_h}, 
\end{split}
\end{align}
where the right hand side (RHS) includes the terms corresponding to the source 
\begin{align}
\begin{split}
\rhsg(v_h) \doteq & \displaystyle\sum\limits_{K \in \thm} (\force,v_h)_K,
\end{split}
\end{align}
and weak boundary conditions $\rhsbc{v_h}$,  
\begin{align}
\begin{split}
B_h(w_h;v_h) \doteq -\sum\limits_{F\in\setfaces^-}\int_{F} \betab \cdot \n_{\partial \Omega} w_h v_h -\sum\limits_{F\in\setfaces^-\cup\setfaces^+}\int_F \mu w_h \mean{ \gradient v_h} \cdot \n_{\partial \Omega}
+\sum\limits_{F\in \setfaces^-\cup\setfaces^+}\int_F \cip\mu h_F^{-1} w_h v_h. 
\end{split}
\end{align}
The parameter $c^{\rm ip}$ is a constant set to $10$\JB{, as suggested in \cite{badia_discrete_2015}}. The projection
$\bcprojection\in V_h|_{\partial \Omega}$ is the facetwise linear 
polynomial function obtained,
e.g., by the nodal interpolation of \JB{the given Dirichlet boundary data} $\bc$ on the nodes of the boundary
$\setbcneighbours$, i.e., $\bar{u}_h = \sum_{a\in\setbcneighbours}
\shap_a \bc(\x_a)$. When the nodal projector is not well-defined,
other projections that preserve the DMP can be also used, e.g., the
Scott-Zhang projection {\cite{scott_finite_1990}}.  Notice that with
this definition $\bcprojection$ is bounded by the maximum and minimum
values of the function $\bc$. Moreover, the semi-discrete problem
\eqref{eq-dscrtpbm} can be rewritten in algebraic form as
\begin{equation}\label{eq-cmprob}
\M\partial_t u_h + \A u_h = \G + \B \bc_h,
\end{equation}
where $\M_\ab\doteq (\shap_b,\shap_a)$ and $\A_\ab\doteq \spacebilinearform(\shap_b,\shap_a)$, for $a, \, b \in \mathcal{N}_h$, $\G_a\doteq\rhsg(\shap_a)$, for $a \in \mathcal{N}_h$,  and $\B_\ab\doteq B_h({\shap_b};\shap_a)$, for $a \in \mathcal{N}_h, \, b \in \mathcal{N}_h^\partial$. 

\subsection{Implicit time integration}
We consider the time discretization of \eqref{eq-dscrtpbm} using the method of lines. In doing so, we are interested in schemes that ensure the DMP as the discrete solution evolves in time. This kind of methods are also known as local extrema diminishing (LED). In particular, we will use the $\theta$-method, even though the generalization of the following results to other schemes that preserve monotonicity properties is straightforward. We consider a partition of $(0,T]$ into $\ntsteps$ time steps with equal time step length $\Delta t = \frac{T}{\ntsteps}$ in such a way that $t_\tstep = \tstep \Delta t$, $\tstep = 0,\cdots,\ntsteps$. The problem will be solved by computing an approximation of $u$ in each of those time steps $u_h^\tstep \approx u(\cdot,t_\tstep)$. The discretization of \eqref{eq-strongform4} by means of the $\theta$-method reads: Find  $u_h^{\tstep+1}\in \vh$ such that
\begin{equation}\label{eq-defdscrtpbm}
\quad \frac{1}{\Delta t} (u_h^{\tstep+1}-u_h^{\tstep},v_h) + \spacebilinearform(\theta u_h^{\tstep+1} + (1-\theta)u_h^{\tstep},v_h) = \rhsg(v_h) + \rhsbc{v_h} \quad \forall v_h\in \vh.
\end{equation}

We use a projection of the actual initial condition $u$ at $t=0$ as the initial discrete solution $u_h^0$, such that it inherits the DMP. The value of $\theta$ is to be chosen in the interval $[0,1]$. Some common values are $\theta=0$, which leads to the explicit forward Euler scheme, $\theta = 0.5$ for the Crank-Nicolson scheme, and $\theta=1$, leading to the Backward-Euler  (BE) scheme. Each of these methods has different features. In particular, in order to obtain an unconditionally LED scheme, it is necessary to use $\theta=1$. Using BE, the discrete problem in compact form reads
\begin{equation}\label{eq-discgal}
\M\delta_t u_h + \A\unknp = \G + \B\bcprojection,
\end{equation}
where $\delta_t u_h=\Delta t^{-1}(\unknp-\unkn)$. Other choices of $\theta$ lead to  LED schemes under a CFL-like condition. This is specially important in the case of the Crank-Nicolson (CN) method, which is second-order accurate and non-dissipative. The requirements to obtain monotonicity-preserving schemes in these cases are commented in Sect. \ref{sec-dmpproof}.

\section{Monotonicity Properties}
\label{sec-DMP_paper4}
In this section we introduce the desired properties that we want our
discrete problem to fulfill. The use of local extrema in dG is too
restrictive for our purposes. \JB{In dG, due to the existence of jumps, local extrema that do not harm the MP may appear, e.g. a possitive jump between two elements with negative gradients}. Thus, we
consider the concept of \emph{discrete local extrema}, which is
defined on nodes, and means that the nodal value is extremum in the
support of the node. A discrete local extremum is a local extremum, but
not the opposite. We will see later on that this weaker definition is
enough for our purposes.

\begin{definition}[Local Discrete Extremum]\label{def-extrema}
	The function $u_h \in\vh$ has a local discrete maximum (resp. minimum) on a node $a \in \mathcal{N}_h$ if $u_a\geq u_h(\x)$ (resp. $u_a\leq u_h(\x)$) $\forall \x\in \domain_a$, and also $u_a\geq \bc(\x)$ (resp. $u_a\leq \bc(\x)$) $\forall \x\in \partial\domain_a\cap\partial\domain$.
\end{definition} 
Therefore, the DMP can be defined as follows.
\begin{definition}[DMP]\label{def-DMP}
	For steady problems a solution $u_h\in\vh$ satisfies the local DMP if for every $a \in \mathcal{N}_h$, we have:
	\begin{align}
	u_a^{\rm min} \leq u_a \leq u_a^{\rm max}, 
	\quad where\;\;& u_a^{\rm max}\doteq \max \left\{ \max_{b\in\dofsupport (a)\backslash \{a\}}  u_b,  \max_{\x\in \partial\domain_a\cap\partial\domain_D}\bcprojection(\x)\right\},\\
	and \;\;& u_a^{\rm min}\doteq \min \left\{ \min_{b\in\dofsupport (a)\backslash \{a\}} u_b,\, \min_{\x\in \partial\domain_a\cap\partial\domain_D} \bcprojection(\x)\right\},
	\end{align}
	where $\bcprojection\in\vh|_{\partial\Omega}$ is the finite element interpolation of the boundary conditions on the Dirichlet boundary $\partial\domain_D$.	
	In the case of transient problems, $u_a^{\max}$ and $u_a^{\min}$ are defined as
	\begin{align}
	u_a^{\rm max}\doteq \max \left\{ \max_{\dofsupport (a)\backslash \{a\}}  u_b,  \max_{\x\in \partial\domain_a\cap\partial\domain_D\times(0,T]}\bc(\x), \max_{\x\in\domain} u_h^0(\x)\right\} \\
	u_a^{\rm min}\doteq \min \left\{ \min_{\dofsupport (a)\backslash \{a\}} u_b,\, \min_{\x\in \partial\domain_a\cap\partial\domain_D\times(0,T]} \bc(\x),\min_{\x\in\domain} u_h^0(\x) \right\},
	\end{align}
	where $u_h^0\in\vh$ is the finite element projection of the initial condition $u_0$.
\end{definition}
A scheme such that its solutions satisfy the DMP is called DMP-preserving. Instead, for transient problems, we define LED schemes.
\begin{definition}[LED]\label{def-LED}
	A method is called LED if for $g=0$ and any time in $t \in (0,T]$, the solution $u_h(t) \in V_h$ satisfies
	\begin{equation}
	\der_t u_a \leq 0\; \text{if } u_a \text{ is a maximum and } 
	\der_t u_a \geq 0\; \text{if } u_a \text{ is a minimum.} 
	\end{equation}
	For time-discrete methods, the same definition applies, replacing $\der_t$ by the time derivative discrete approximation $\delta_t$.
\end{definition}

Let us assume now that we have system \eqref{eq-cmprob} plus a \JB{Lipschitz continuous nonlinear diffusion  term (\emph{nonlinear stabilization})}:
\begin{equation}\label{eq-absdscrtpbm}
\Mt(u_h,\bcprojection) \partial_t u_h + \At(u_h,\bcprojection) u_h = \G + \Bt(u_h,\bcprojection)\bc_h.
\end{equation}
The superscript, e.g., in $\At$, denotes the fact that the
operator $\At$ is equal to $\A$ plus stabilization terms. 
We have written, e.g.,
$\At(u_h,\bcprojection)$, to explicitly denote the fact that the
entries of the matrix $\At$ are potentially nonlinear with respect to
$u_h$ and $\bcprojection$. Problem \eqref{eq-absdscrtpbm} is LED
under the following requirements.
\begin{theorem}[LED]\label{thm-led}
	The semi-discrete problem \eqref{eq-absdscrtpbm} is LED (as defined in Def. \ref{def-LED}) if $g = 0$ and for every $a\in\dofsupport$ such that $u_a$ is a local extremum, it holds:
	\begin{subequations} \label{eq-ledcond}
		\begin{eqnarray}\label{eq-ledcond-a}
		&\Mt_\ab(u_h,\bcprojection) \doteq \delta_\ab m_a,\, \text{with } m_a>0, \\
		\label{eq-ledcond-b}
		&\At_\ab(u_h,\bcprojection) \leq 0,  {\,\forall \, b \in \mathcal{N}_h \, : \, b \neq a,\; \qquad \Bt_\ab(u_h,\bcprojection) \geq 0, \, \:\;\forall \, b \in \mathcal{N}_h^\partial}, \\ 
		\label{eq-ledcond-c}& \sum_{b\in\dofsupport(a)} \At_\ab(u_h,\bcprojection) - \sum_{b\in\setbcneighbours(a)} \Bt_\ab(u_h,\bcprojection) = 0, 
		\end{eqnarray}
	\end{subequations}
	where $m_a \doteq \int_{\Omega} \shap_a d\Omega$, $\delta_\ab$ is the Kronecker delta. Further, for $g \leq 0$ (resp. $g \geq 0$) in $\domain$ solutions of \eqref{eq-absdscrtpbm} satisfy the DMP property in Def. \ref{def-DMP}. Moreover, the discrete problem \eqref{eq-absdscrtpbm} is positivity-preserving for $g\geq 0$ and $u_0>0$.
\end{theorem}
\begin{proof}
	Assume $u_a$ is a discrete maximum. \JB{From the conditions in \eqref{eq-ledcond} and particularizing equation \eqref{eq-absdscrtpbm} to node $a$, we have that} 
	\begin{align}
	\G_a &= m_a \der_t u_a + \sum_{b\in\dofsupport(a)} \At_\ab(u_h,\bcprojection) u_b - \sum_{b\in\setbcneighbours(a)} \Bt_\ab(u_h,\bcprojection) \bc_b \\
	& \geq m_a \der_t u_a + \left(\sum_{b\in\dofsupport(a)} \At_\ab(u_h,\bcprojection) - \sum_{b\in\setbcneighbours(a)}\Bt_\ab(u_h,\bcprojection)\right) u_a 
	= m_a \der_t u_a .
	\end{align}
	Therefore, $\der_t u_a \leq \G_a = 0$. Proceeding analogously for a minimum we can prove that the method is LED. The proof is equivalent for the discrete problem with BE time integration.
	
	Next, we prove positivity. Let us consider that at some time step $m$ the solution becomes negative, and consider the degree of freedom $a$ in which the minimum value is attained. Using the previous result for a minimum at the discrete level, we have that $\delta_t u_a\geq 0$ and thus $u_a^m \geq u_a^{m-1}$. It leads to a contradiction, since $u_a^{m-1}\geq 0$. Hence, the solution must remain positive.
\end{proof}

\begin{corollary}
	\JB{If the problem in \eqref{eq-absdscrtpbm} is discretized in time with BE and it} meets the conditions in Th. \ref{thm-led}, then it leads to solutions that satisfy the local DMP in Def. \ref{def-DMP} at every time $t^n$, for $n=1,...N^t$.
\end{corollary}
\begin{proof}
	By the LED property we know that a discrete maximum (resp. minimum) will be bounded above (resp. below) by the solution at the previous time step. Proceeding by induction, the solution will be bounded by the initial condition $u_h^0$ and the boundary conditions imposed at any previous time step.
\end{proof}

Following \cite[Th. 1]{codina_discontinuity-capturing_1993}, we can prove that the steady counterpart of problem \eqref{eq-absdscrtpbm} is DMP-preserving.

\begin{theorem}[DMP]\label{thm-genDMP}
	A steady solution of the \JB{semi-}discrete problem \eqref{eq-absdscrtpbm} satisfies the DMP in Def. \ref{def-DMP} if $g = 0$ in $\domain$ and, for every degree of freedom $a\in\dofsupport$ such that $u_a$ is a local discrete extremum, conditions \eqref{eq-ledcond-b}-\eqref{eq-ledcond-c} hold.
\end{theorem}
\begin{proof}
	Assume $u_a$ is a discrete maximum, then the steady counterpart of problem \eqref{eq-absdscrtpbm} reads 
	\begin{equation}
	\sum_{b\in\dofsupport(a)} \At_\ab(u_h,\bcprojection)u_b - \sum_{b\in\setbcneighbours(a)} \Bt_\ab(u_h,\bcprojection) \bc_b = 0,
	\end{equation}
	Therefore, $u_a$ can be computed as
	\begin{equation}
	u_a = \frac{\sum_{b\in\setbcneighbours(a)} \Bt_\ab(u_h,\bcprojection) \bc_b - \sum_{b\in\dofsupport(a)\backslash\{a\}}\At_\ab(u_h,\bcprojection)u_b}{\At_{aa}(u_h,\bcprojection)} .
	\end{equation}
	From conditions \eqref{eq-ledcond-b}-\eqref{eq-ledcond-c}, the coefficients 
	that multiply $u_b$ and $\overline{u}_b$ are in $[0,1]$, and the sum of all 
	these coefficients add up to one. Therefore, $u_a$ is a convex
	combination of its neighbors (including boundary conditions
	$\bc_h$). Since $u_a$ is a maximum and a convex combination of its
	neighbors, then $u_b=u_a$ for some $b\in\dofsupport(a)$. Further,
	it can also be proved that $u_a$ is a convex combination of all its
	neighbors \emph{but} $u_b$, and vice versa $u_b$ is a convex combination of
	all its neighbors \emph{but} $u_a$. Hence, by induction, we know that
	extrema at any degree of freedom are bounded by the boundary
	conditions. Thus, the DMP is satisfied.
\end{proof}
\section{The DMP-preserving artificial diffusion scheme} \label{sec-dmpproof}
In the previous section, we have stated the requirements to be
fulfilled by our discrete scheme to be DMP-preserving and LED. In this
section, we build a nonlinear stabilization of the dG formulation
\eqref{eq-dscrtpbm} that satisfies all these conditions. The nonlinear
stabilization will rely on an artificial graph-viscosity term. The
graph-viscosity is supplemented with a shock detector, in order to
obtain higher than linear convergence on smooth regions. Moreover, for
transient methods we make use of the shock detector in order to
perform the mass matrix lumping only where is required, which allows
us to minimize the phase error of the method.

Let us start by defining the graph-viscosity
$\graphviscosityab$. For $a \in \mathcal{N}_h$ and
$b\in\setbcneighbours(a)$ we define
\begin{equation}\label{eq-bgraphviscosity}
\nubc_\ab \doteq \max \{-\alpha_{a} B_h(\shap_b;\shap_a) ,0 \}.
\end{equation}
Clearly, this viscosity is only non-zero when $a \in
\mathcal{N}_h^\partial$. Next, for $a \in \mathcal{N}_h$ and
$b\in\dofsupport(a)$, we define
\begin{equation}\label{eq-graphviscosity}
\graphviscosityab \doteq \left\{
\begin{array}{lc}
\max\{\alpha_{a}\spacebilinearform(\shap_b,\shap_a),0,\alpha_{b}\spacebilinearform(\shap_a,\shap_b)\} &  b\neq a,\\
\sum_{b\in\dofsupport(a)\setminus a} \graphviscosityab +  \sum_{b\in\setbcneighbours(a)} \nubc_\ab & \text{otherwise},
\end{array}
\right.
\end{equation}
where $\alpha_{a}$ is a parameter that enjoys the following property.
\begin{definition}\label{def-shockdetectorproperty}
	Given $a \in \mathcal{N}_h$, we say that $\alpha_a:\vh\longrightarrow \mathbb{R}$ enjoys the \emph{shock detector property} if it is such that $\alpha_a(u_h,\bcprojection)\in[0,1]$ $\forall u_h\in\vh$ and $\alpha_a(u_h,\bcprojection)=1$ if $u_h$ has a local discrete extremum on $\x_a$.
\end{definition}

Next, we design a shock detector that satisfies this property. Given $a \in \mathcal{N}_h$ and $b\in\dofsupport(a)$ with $\x_b\neq \x_a$, we define $\x_\ab^\sym$ as the intersection between $\partial\Omega_a$ and the line that passes through $\x_b$ and $\x_a$, and it is not $\x_b$. Moreover, we define  $\rr_\ab \doteq \x_b - \x_a$, $\rr_\ab^\sym\doteq \x_\ab^\sym - \x_a$, and $u_\ab^\sym \doteq u_h(\x_\ab^\sym)$ (see Fig. \ref{fig-usym}). Further, we denote by $\ru_\ab$ the unit vector of $\rr_\ab$, and by \JB{$h_a$ a characteristic length of $\Omega_a$.} Then, we define the jump and the mean of the unknown gradients as
\begin{equation}
\jump{\gradient u_h}_\ab \doteq \left\{\begin{array}{ll}
\dfrac{u_b - u_a}{h_a} & \text{if} \; \x_a = \x_b, \\
\dfrac{u_b - u_a}{\vert \rr_\ab \vert} + \dfrac{u_\ab^\sym - u_a}{\vert \rr_\ab^\sym \vert} & \text{otherwise},
\end{array}\right.
\end{equation}
\begin{equation}
\mean{\vert \gradient u_h \cdot \ru_\ab \vert}_\ab \doteq \left\{\begin{array}{ll}
\dfrac{|u_b - u_a|}{h_a} & \text{if} \; \x_a = \x_b, \\
\dfrac{1}{2} \left(\dfrac{|u_b-u_a|}{|\rr_\ab|}+\dfrac{|u_\ab^\sym - u_a|}{|\rr_\ab^\sym|}\right) & \text{otherwise}.
\end{array}\right.
\end{equation}

\begin{figure}[ht]
	\centering
	\includegraphics[width=0.25\textwidth]{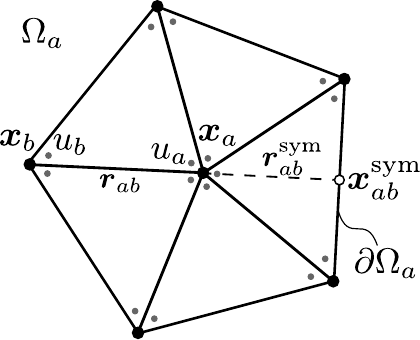}
	\caption{Representation of the symmetric node $\x_\ab^\sym$ of $\x_b$ with respect to $\x_a$.}
	\label{fig-usym}
\end{figure}

\begin{remark}\label{rmk-bcextrapolation}
	These definitions may imply $\x_\ab^\sym=\x_a$ on some boundaries. In these cases, the value at the symmetric point of $\x_b$ with respect to $\x_a$ takes an extrapolation of the boundary condition value such that the method is linearly preserving, i.e., $u_\ab^\sym=\bc_a + (u_b-u_a)\,\sgn{(\bc_a -u_a)(u_b-u_a)} $, and the value $|\rr_\ab^\sym|$ is taken equal to  $|\rr_\ab|$. This extrapolation is not only important for linear preservation, but also for obtaining optimal convergence rates in convection-diffusion problems with boundary layers.
\end{remark}

\begin{remark}
	Notice that when $\x_\ab^\sym$ coincides with a node the value $u_h(\x_\ab^\sym)$ is not unique. In this case, we compute $\jump{\gradient u_h}_\ab$ and $\mean{\vert\gradient u_h\cdot\ru_\ab\vert}_\ab$ for all values of $u_h(\x_\ab^\sym)$ in $\domain_a$.
\end{remark}

Making use of the above definitions, the proposed shock detector reads
\begin{equation}
\label{eq-shockdetector}
\alpha_a(u_h,\bcprojection) \doteq \left\{\begin{array}{cc}
\left(\displaystyle\frac{\left| \sum_{b\in\dofsupport(a)} \jump{\gradient u_h}_\ab\right|}{\sum_{b\in\dofsupport(a)} 2\mean{|\gradient u_h \cdot \ru_\ab|}_\ab}\right)^q & \text{if } \sum_{b\in\dofsupport(a)}\mean{|\gradient u_h \cdot \ru_\ab|}_\ab \neq 0, \\
0 & \text{otherwise}.
\end{array}\right.
\end{equation}
where $q\in\mathbb{R}^+$. Let us prove that the shock detector \eqref{eq-shockdetector} satisfies the shock detector property in Def. \ref{def-shockdetectorproperty}. \JB{We note that the definition of $\alpha_a$ is motivated from \cite{badia_monotonicity-preserving_2014}. Here, instead of using the maximum coefficient obtained, we use the sum all gradient jumps divided by the sum of all gradient means. A similar (more involved) modification can be found in \cite{burman_monotonicity_2015}.}
\begin{lemma}
	The function $\alpha_a(u_h,\bcprojection)$ defined in \eqref{eq-shockdetector} satisfies the shock detector property in Def. \ref{def-shockdetectorproperty}. Furthermore, if $a \in \mathcal{N}_h$ is not an extremum and $q = \infty$, $\alpha_a(u_h,\bcprojection)=0$.
\end{lemma}

\begin{proof}
	Let us assume that $u_h$ has a discrete maximum (resp. minimum) on $\x_a$, then 
	\begin{equation}\label{eq-sdcond}
	\begin{array}{l}
	u_b-u_a \leq 0\quad \forall b\in\dofsupport(a) \qquad\text{and}\qquad
	u^{\sym}_\ab-u_a \leq 0 \quad \forall b\in\dofsupport(a), \\
	(\text{resp. }u_b-u_a \geq 0\quad \forall b\in\dofsupport(a) \qquad\text{and}\qquad
	u^{\sym}_\ab-u_a \geq 0 \quad \forall b\in\dofsupport(a)).
	\end{array}
	\end{equation}
	Therefore,
	\begin{align}
	\left|\sum_{b\in\dofsupport(a)} \jump{\gradient u_h}_\ab\right| &= 
	\left|\sum_{\substack{b\in\dofsupport(a)\\\x_b=\x_a}} \jump{\gradient u_h}_\ab + 
	\sum_{\substack{b\in\dofsupport(a)\\\x_b\neq\x_a}} \jump{\gradient u_h}_\ab\right| \\
	&=\left|\sum_{\substack{b\in\dofsupport(a)\\\x_b=\x_a}} \dfrac{u_b - u_a}{h_a} + 
	\sum_{\substack{b\in\dofsupport(a)\\\x_b\neq\x_a}} \dfrac{u_b - u_a}{\vert \rr_\ab \vert} + \dfrac{u_\ab^\sym - u_a}{\vert \rr_\ab^\sym \vert}\right|\\
	& = 
	\sum_{\substack{b\in\dofsupport(a)\\\x_b=\x_a}} \dfrac{|u_b - u_a|}{h_a} + \sum_{\substack{b\in\dofsupport(a)\\\x_b\neq\x_a}} \dfrac{|u_b - u_a|}{\vert \rr_\ab \vert} + \dfrac{|u_\ab^\sym - u_a|}{\vert \rr_\ab^\sym \vert} 
	= \sum_{b\in\dofsupport(a)}2\mean{\vert \gradient u_h \cdot \ru_\ab \vert}_\ab .
	\end{align}
	Thus, $\alpha_a(u_h,\bcprojection)=1$. Further, if $u_a$ is not an extremum, then \eqref{eq-sdcond} is no longer true. Hence,
	\begin{align}
	\left|\sum_{b\in\dofsupport(a)} \jump{\gradient u_h}_\ab\right| &= 
	\left|\sum_{\substack{b\in\dofsupport(a)\\\x_b=\x_a}} \jump{\gradient u_h}_\ab + 
	\sum_{\substack{b\in\dofsupport(a)\\\x_b\neq\x_a}} \jump{\gradient u_h}_\ab\right| \\
	&=\left|\sum_{\substack{b\in\dofsupport(a)\\\x_b=\x_a}} \dfrac{u_b - u_a}{h_a} + 
	\sum_{\substack{b\in\dofsupport(a)\\\x_b\neq\x_a}} \dfrac{u_b - u_a}{\vert \rr_\ab \vert} + \dfrac{u_\ab^\sym - u_a}{\vert \rr_\ab^\sym \vert}\right|\\
	& < 
	\sum_{\substack{b\in\dofsupport(a)\\\x_b=\x_a}} \dfrac{|u_b - u_a|}{h_a} + \sum_{\substack{b\in\dofsupport(a)\\\x_b\neq\x_a}} \dfrac{|u_b - u_a|}{\vert \rr_\ab \vert} + \dfrac{|u_\ab^\sym - u_a|}{\vert \rr_\ab^\sym \vert} 
	= \sum_{b\in\dofsupport(a)}2\mean{\vert \gradient u_h \cdot \ru_\ab \vert}_\ab .
	\end{align}
	Therefore, $\alpha_a(u_h,\bcprojection)<1$. Moreover, when $q=\infty$, $\alpha_a(u_h,\bcprojection)=0$ if $u_a$ is not an extremum.
\end{proof}

In order to prove the DMP-preservation in the numerical analysis below we need to perturb both the weak boundary conditions and the bilinear form. The perturbed weak boundary conditions read:
\begin{equation}
\label{eq-cccccb}
\pertwbc(w_h,\overline{w}_h,\bcprojection;v_h) \doteq B_h(\bcprojection;v_h) + \sum_{a\in\dofsupport}\sum_{b\in\setbcneighbours(a)} \nubc_\ab(w_h,\overline{w}_h) v_{a}\bc_{b} .
\end{equation}
where $v_h$, $w_h\in\vh$, and $\overline{w}_h,\bcprojection\in\vh|_{\partial\Omega}$. Furthermore, given $u_h, \, v_h, w_h \in V_h$ and $\overline{w}_h,\in\vh|_{\partial\Omega}$,  we can define the perturbed bilinear form $\pertbform$ as:
\begin{align}
\label{eq-ccccc}
\pertbform(w_h,\overline{w}_h;u_h,v_h) \doteq \spacebilinearform(u_h,v_h)  + \sum_{a\in\dofsupport}\sum_{b\in\dofsupport(a)} \graphviscosityab(w_h,\overline{w}_h)  v_{a}u_{b}\ell(a,b),
\end{align}
where $\ell(a,b)\doteq 2\delta_\ab -1$ is the graph-Laplacian operator.
It leads to the following stabilized steady discrete problem: Find $u_h \in V_h$ with $\bcprojection\in V_h|_{\partial\Omega}$ such that
\begin{equation}\label{eq-steadydscrtpbm} \pertbform(u_h,\bcprojection;u_h,v_h) = \rhsg(v_h) + \pertwbc(u_h,\bcprojection,\bcprojection;v_h) \quad \forall v_h\in \vh.
\end{equation}
We are ready to prove the desired DMP property of this method. For this purpose, we define 
$\Abar_{ab}(u_h,\bcprojection)\doteq \pertbform (u_h,\bcprojection;\shap_b,\shap_a)$ for $a, \, b \in \mathcal{N}_h$, and $\Bbar_{ab}(u_h,\bcprojection)\doteq \pertwbc (u_h,\bcprojection,\shap_b;\shap_a)$ for $a \in \mathcal{N}_h, \, b \in \mathcal{N}_h^\partial$.

\begin{theorem}\label{thm-A_DMP}
	The discrete problem \eqref{eq-steadydscrtpbm} with the stabilized semilinear forms defined in \eqref{eq-cccccb} and \eqref{eq-ccccc} is DMP-preserving for $g=0$.
\end{theorem}
\begin{proof}
	\JB{As seen in Th. \ref{thm-genDMP},  the solution is DMP-preserving if conditions \eqref{eq-ledcond-b}-\eqref{eq-ledcond-c} are satisfied. Let us verify these two conditions. }
	Let $\x_a$ be an interior node and assume that $u_h$ has an extremum on $\x_a$. Given the set of all nodes $b\in\dofsupport(a)$ coupled to node $a \in \mathcal{N}_h$, we have:
	\begin{align}
	\JB{	\sum_{ b  \in\dofsupport(a)}  \A_\ab - \B_\ab} =
	&  \sum_{b  \in\dofsupport(a)} \Bigg\{\sum_{K\in\mathcal{T}_h} \int_{K}  (\mu\gradient\shap_b \cdot \gradient\shap_a-\shap_b\betab\cdot\gradient\shap_a)  \\
	& + \sum_{F\in\setfaces}\int_F \mu\left(- \jump{\shap_b} \cdot \mean{\gradient\shap_a}-\mean{\gradient\shap_b} \cdot \jump{\shap_a} + \cip h_F^{-1}\jump{\shap_b}\cdot \jump{\shap_a}\right) \\
	& + \sum_{F\in \setfaces^+\cup\setfaces^0} \int_F  \betab \mean{\shap_b}\cdot \jump{\shap_a} + \sum_{F\in\setfaces^0}  \int_F \frac{|\betab\cdot\n_F|}{2} \cdot \jump{\shap_b}\cdot \jump{\shap_a} \\ 
	& + \sum_{F\in\setfaces^-}\int_F  \betab \cdot \n_F \shap_b\shap_a
	+ \sum_{F\in\setfaces^-\cup\setfaces^+}\int_F \mu \shap_b\mean{\gradient\shap_a}\cdot \n_F \\ 
	& -\sum_{F\in\setfaces^-\cup\setfaces^+}\int_F \cip\mu h_F^{-1}\shap_b\shap_a \Bigg\}. 
	\end{align}
	(We note that $\B_{\ab}$ has only been defined for nodes $b
	\in \mathcal{N}_h^\partial$. Here, we abuse of notation, and
	extend by zero the definition to all nodes, with $\B_{\ab}=0$ when $b
	\not\in \mathcal{N}_h$.)  We use the fact that the shape
	functions are a partition of unity, i.e., $\sum_{b\in
		\mathcal{N}_h(a)} \shap_b$ is equal to one on $\Omega_a$ and
	zero elsewhere. As a result, $\sum_{b\in \mathcal{N}_h(a)}
	\jump{\shap_b} = 0$ on facets $F \subset \Omega_a \setminus
	\partial \Omega_a$, and $\sum_{b\in \mathcal{N}_h(a)} \gradient
	\shap_b = 0$ in any $K \in \mathcal{T}_h$.  On the other hand,
	$\shap_a$ vanishes on any $F \subset \partial \Omega_a
	\setminus \partial \Omega$ by construction. Using these
	properties, we get:
	\begin{align}
	&\sum_{b\in\dofsupport(a)}\sum_{F\in\setfaces^0} \int_F \frac{|\betab\cdot\n_F|}{2} \jump{\shap_b} \cdot \jump{\shap_a} = 0,\\ 
	&\sum_{b\in\dofsupport(a)}\left\{\sum_{F\in\setfaces}\int_F \mu \cip h_F^{-1}\jump{\shap_b}\cdot\jump{\shap_a} -\sum_{F\in\setfaces^-\cup\setfaces^+}\int_F \cip\mu h_F^{-1}\shap_b\shap_a \right\}= 0,\\
	&\sum_{b\in\dofsupport(a)}\left\{\sum_{F\in\setfaces}\int_F  -\mu  \jump{\shap_b}\cdot\mean{\gradient\shap_a} + \sum_{F\in\setfaces^-\cup\setfaces^+}\int_F \mu \shap_b\mean{\gradient\shap_a}\cdot \n_F  \right\}= 0,\\
	&\sum_{b\in\dofsupport(a)}\sum_{F\in\setfaces}\int_F \mu\mean{\gradient\shap_b}\cdot\jump{\shap_a} = 0,
	\quad
	\sum_{b\in\dofsupport(a)}\int_\domain  \mu\gradient\shap_b\cdot\gradient\shap_a = 0,
	\end{align}
	and the following terms can be integrated by parts as
	
	\begin{align}
	& \sum_{b  \in\dofsupport(a)} \left\{ - \sum_{K\in\mathcal{T}_h} \int_{K} \shap_b\betab\cdot\gradient\shap_a
	+ \sum_{F\in \setfaces^+\cup\setfaces^0} \int_F  \betab \mean{\shap_b}\cdot \jump{\shap_a} 
	+ \sum_{F\in\setfaces^-}\int_F  \betab \cdot \n_F \shap_b\shap_a \right\} \\ 
	& = \sum_{b  \in\dofsupport(a)} \left\{ -\sum_{K\in\mathcal{T}_h}\left( 
	\int_{K} \shap_b\betab\cdot\gradient\shap_a
	+ \int_{\partial K}  \betab \cdot \n\, \shap_b \shap_a \right)
	\right\} \\
	& = \sum_{b\in\dofsupport(a)} \sum_{K\in\mathcal{T}_h} \int_K \shap_a\betab\cdot\gradient\shap_b = 0.
	\end{align}
	\JB{Finally, since $\sum_{b\in\dofsupport(a)} \nu_\ab(u_h)v_a v_b \ell(a,b)+\sum_{b\in\setbcneighbours(a)} \nubc_\ab(u_h)v_a v_b=0$} by construction (see \eqref{eq-bgraphviscosity} and \eqref{eq-graphviscosity}), then $\sum_{ b  \in\dofsupport(a)} \Abar_\ab(u_h,\bcprojection) - \Bbar_\ab(u_h,\bcprojection)=0$. Moreover, it is clear that $\Abar_\ab(u_h,\bcprojection)\leq 0$ for any $b\neq a$ and $\Bbar_\ab(u_h,\bcprojection)\geq 0$ in all cases, based on the definition of these operators in \eqref{eq-cccccb}-\eqref{eq-ccccc} and their respective graph-viscosities in \eqref{eq-bgraphviscosity}-\eqref{eq-graphviscosity}. It finishes the proof.
\end{proof}
Thus, by Th. \ref{thm-A_DMP}, we can ensure that the extrema of the solution of \eqref{eq-cmprob}, will be on the boundary of the domain when $g = 0$. 
Let us now define the mass matrix perturbation used in order to obtain a LED scheme. As its name reveals, this property ensures that the value of the discrete maximum and the minimum of a transient problem can be bounded by those in the initial solution $u_0=u(\cdot,t_0)$ and boundary conditions. It has been proved in \cite[Lemma 3.2]{burman_nonlinear_2007} that, if the steady problem, e.g., \eqref{eq-steadydscrtpbm}, enjoys the DMP property, its transient version enjoys the LED property if we replace the mass matrix $(\partial_t u_h,v_h)$ by its lumped version $(\partial_t u_h,v_h)_h$ corresponding to the Gauss-Lobatto sub-integration. The form $(\cdot,\cdot)_h$ is such that $(\partial_t u_h,\shap_a)_h =  \partial_t u_a (\unityfunction,\shap_a)$ for any $a \in \dofsupport$. In fact, as Kuzmin and co-workers have proved in \cite{dmitri_kuzmin_new_2015,kuzmin_gradient-based_2016}, it is enough to lump only the terms associated to the degrees of freedom where $u_h$ has an extremum. Following the same strategy as in \cite{badia_monotonicity-preserving_2016}, we can perform selective lumping using the shock detector. We define:
\begin{align}\label{eq-lumpedmass}
\lumpedbform{u_h,\bcprojection;\partial_t u_h,v_h}\doteq\sum_{a\in\dofsupport} v_a(1-\alpha_a^Q(u_h,\bcprojection))(\partial_t u_h,\shap_a) + \alpha_a^Q(u_h,\bcprojection) \partial_t u_a v_a (\unityfunction,\shap_a).
\end{align}
The exponent $Q>0$ is added in order to minimize the lumping perturbation, which leads to phase error in the discrete solution. In addition, we define $\Mbar_\ab(u_h,\bcprojection)\doteq \lumpedbform{u_h,\bcprojection;\shap_b,\shap_a}$, for $a, \, b \in \mathcal{N}_h$. If one considers the semi-discrete problem in space only, we have: Find $u_h\in \vh $ such that
\begin{equation}\label{eq-transientdscrtpbm}
\lumpedbform{u_h,\bcprojection;\partial_t u_h,v_h} + \pertbform(u_h,\bcprojection;u_h,v_h) = \rhsg(v_h) + \pertwbc(u_h,\bcprojection,\bcprojection;v_h) \quad \forall v_h\in \vh .
\end{equation}
\begin{lemma}\label{lem-LED}
	The scheme \eqref{eq-transientdscrtpbm} with the semilinear forms defined in \eqref{eq-lumpedmass}, \eqref{eq-cccccb}, and \eqref{eq-ccccc} is LED for $g=0$.
\end{lemma}
\begin{proof}
	The conditions required on $\Abar(u_h,\bcprojection)$ and
	$\Bbar(u_h,\bcprojection)$ in Th. \ref{thm-led} to obtain a LED
	scheme have already been proved in Th. \ref{thm-A_DMP}. Further, if
	we assume that $u_a$ is an extremum, then
	$\alpha_a(u_h,\bcprojection)=1$ and $\Mbar_\ab(u_h)$ becomes
	$(\unityfunction,\shap_a)=\delta_\ab m_a$ with $m_a=\int_{\Omega}
	\shap_a$. Hence, the definition of the mass matrix in \eqref{eq-lumpedmass} satisfies
	\eqref {eq-ledcond-a}. As a result, we fulfill
	all conditions stated in Th. \ref{thm-led} and thus the scheme is LED.
\end{proof}
Furthermore, the stabilized problem \eqref{eq-transientdscrtpbm} is linearity preserving, i.e. linear solutions are solution of the original IP dG method \eqref{eq-dscrtpbm}.
\begin{lemma}\label{lem-lp}
	The stabilization terms in \eqref{eq-cccccb}, \eqref{eq-ccccc}, and
	\eqref{eq-lumpedmass}, vanish for functions $u \in P_1(\domain)$, i.e.,
	\begin{align}
	& \pertwbc(u,u_{\partial \Omega},\bar{u}_h;v_h) = B_h(\bar{u}_h;v_h), \qquad
	\pertbform(u,u_{\partial \Omega};u_h,v_h) = \spacebilinearform(u_h,v_h), \\
	&\JB{\lumpedbform{u,u_{\partial \Omega}; u_h,v_h} = ( u_h,v_h)}, \qquad \hbox{for any } \, u_h \in V_h.
	\end{align}
\end{lemma}
\begin{proof}
	If $u$ is linear and continuous, then by definition $\jump{\gradient u}_\ab=0$ $\forall a$ and $b\in\setneighbours(a)$. Hence, for any $a$ we have that $\alpha_a\equiv 0$. Thus, both $\nubc_\ab$ and $\graphviscosityab$ are equal to zero. Thus, all the stabilization terms vanish and we recover the original formulation.
\end{proof}

The results in this section for the BE time discretization can be extended to any $\theta$-method. We refer the reader to the work by Kuzmin and co-workers \cite{kuzmin_flux_2002,kuzmin_algebraic_2005} 
for the proofs of such properties. In particular, $\theta$-methods are positivity-preserving under the CFL-like condition (see \cite[Th. 1]{kuzmin_flux_2002}) 
\begin{align}
\Delta t \leq \min_{a\in \dofsupport} \frac{(\unityfunction,\shap_a)}{(1-\theta)\pertbform(\unknp,\bcprojection^{n+1};\shap_a,\shap_a)}.	
\end{align}
Furthermore, under certain conditions of the matrix and the RHS, it has been proved in \cite[Th. 4]{kuzmin_algebraic_2005} 
that the scheme is not only positivity-preserving but satisfies the DMP. This means that the discrete maximum and the minimum of the solution are bounded by the values of the initial solution and the boundary conditions for any $\theta$-method.

The authors in \cite[Th. 1]{kuzmin_flux_2002} take advantage of the mass lumping properties for all of this proofs, but the lumping only needs to be activated for the degrees of freedom where the discrete solution has an extrema. Thus, the scheme defined in \eqref{eq-defdscrtpbm} together with the definition of $\lumpedbform{\cdot,\cdot;\cdot,\cdot}$ given by \eqref{eq-lumpedmass} leads to a DMP-preserving method under the above CFL-like condition.\\

\section{Lipschitz continuity and existence of solutions}\label{sec-ex}

Let us define the Cartesian product space $\tilde V_h \doteq V_h
\times V_h|_{\partial \Omega}$.  Thus, any function
$\tilde{v} \in \tilde V_h$ can be expressed as $\tilde{v} =
(v,\bar{v})$, where the first component includes the values of the dG
function $v \in V_h$ and the second component the projection of the
Dirichlet values $\bar{v} \in V_h|_{\partial \Omega}$. Analogously, we
can define the set of nodes for $\tilde
V_h$ as $\mathcal{M}_h \doteq \mathcal{N}_h \times
\mathcal{N}_h^\partial \equiv \{ (a_1,0), \, (0,a_2) \, : \, a_1 \in
\mathcal{N}_h, \, a_2 \in \mathcal{N}_h^\partial \}.$ We consider an
extended graph-Laplacian operator over $\tilde V_h \times \tilde V_h$
as follows:
\begin{align}
\tilde{\Viscop}(\tilde u, \tilde v)  & =  \tilde{\Viscop}((u,\bar{u}),(v,\bar{v}))  \doteq \sum_{a \in \mathcal{M}_h} \sum_{b \in \mathcal{M}_h} \tilde \nu_{ab} \ell(a,b) \tilde
u_b \tilde
v_a\\
&   \doteq \sum_{a \in \mathcal{N}_h} \sum_{b \in \mathcal{N}_h} \nu_{ab} \ell(a,b)
u_b v_a
- \sum_{a \in \mathcal{N}_h^\partial} \sum_{b \in \mathcal{N}_h^\partial} \nu_{ab}^\partial
\bar{u}_b {v}_a
- \sum_{a \in \mathcal{N}_h^\partial} \sum_{b \in \mathcal{N}_h^\partial} \nu_{ba}^\partial {u}_b
\bar{v}_a.
\end{align}
\JB{Note that the boundary degrees of freedom are replicated. }
Based on this definition, we implicitly have:
\begin{align}
\begin{array}{lll}
& \tilde \nu_{ab} = \nu_{ab}, \quad &\hbox{if } \, a,\, b \in (\mathcal{N}_h,0), \\
& \tilde \nu_{ab} = \nu_{ab}^\partial, \quad &\hbox{if } \, a \in (\mathcal{N}_h^\partial,0),\, b \in (0,\mathcal{N}_h^\partial), \\
& \tilde \nu_{ab} = \nu_{ba}^\partial, \quad &\hbox{if } \, a \in (0,\mathcal{N}_h^\partial),\, b \in (\mathcal{N}_h^\partial,0), \\
& \tilde \nu_{ab} = 0, \qquad &\hbox{if } \, a,\, b \in (0,\mathcal{N}_h^\partial).
\end{array}
\end{align}
It is easy to check that this operator is symmetric and
positive-semidefinite. In order to show the second property, we use the expression for $\nu_{ab}$ and
$\nu_{ab}^{\partial}$ in \eqref{eq-bgraphviscosity} and \eqref{eq-graphviscosity}, respectively, in order to get $\nu_{ab}, \nu_{ab}^\partial \geq 0$, and
$$\JB{\tilde \nu_{aa} = \nu_{aa} =  \sum_{b \in \mathcal{N}_h(a)\backslash a} \nu_{ab} + \sum_{b \in
		\mathcal{N}_h^\partial (a)} \nu^\partial_{ab} = \sum_{b \in \mathcal{M}_h(a)\backslash a} \tilde \nu_{ab}}.$$
Using the last property \JB{and the definition of $\ell(a,b)$}, we get:
\begin{equation}\label{proplap}
\begin{array}{rl}
2 \tilde{\Viscop}(\tilde u, \tilde v) &= \displaystyle{\JB{ \sum_{a \in \mathcal{M}_h} \sum_{b \in \mathcal{M}_h} \tilde \nu_{ab}\ell(a,b) \tilde v_a ( \tilde u_b - \tilde u_a) 
		+\sum_{a \in \mathcal{M}_h} \sum_{b \in \mathcal{M}_h} \tilde \nu_{ab}\ell(a,b) \tilde u_b ( \tilde v_a - \tilde v_b)}}\\
&=\displaystyle{
	\sum_{a \in \mathcal{M}_h} \sum_{b \in \mathcal{M}_h} \tilde \nu_{ab}\ell(a,b)\JB{ ( \tilde u_b - \tilde u_a)} (\tilde v_a - \tilde v_b)}. 
\end{array}
\end{equation}
Thus, we have $| \tilde u |^2_{\tilde{\Viscop}} \doteq
\tilde{\Viscop}(\tilde u,\tilde u) \geq 0$. Further, we define the
restriction operators $\Viscop(u,v) = \tilde{\Viscop}((u,0),(v,0))$ and
$\Viscop^\partial(\bar{u},\bar{v}) = \tilde{\Viscop}((0,\bar{u}),(0,\bar{v}))$, and their
corresponding semi-norms $| u |_{\Viscop} \doteq \Viscop(u,u)$ and $|
\bar{u} |_{\Viscop^\partial} \doteq \Viscop^\partial(\bar{u},\bar{u})$.

Given the source $g \in V_h'$ and $\bar{u} \in V_h|_{\partial
	\Omega}$, we define the operator $\T: V_h \rightarrow V_h'$ for the
steady problem as:
\begin{equation}\label{eq-proof1}
\arraycolsep=1.4pt\def\arraystretch{1.4}
\begin{array}{rl}
\langle \T(z), v \rangle \doteq \,  & \JB{K_h(z,v) }
- B_h(\bar{u},v) - G_h(v)
+ \JB{\tilde{\Viscop}((z, \bar u),(v,0))}
\\ = & \JB{ K_h(z,v) }
- B_h(\bar{u},v) - G_h(v)
+ \displaystyle\sum_{a \in \mathcal{N}_h} \sum_{b \in \mathcal{N}_h(a)} \nu_{ab} \ell(a,b)
v_a
z_b \\ & \JB{- \displaystyle\sum_{a \in \mathcal{N}_h}  \sum_{b \in \mathcal{N}^\partial_h(a)} \nu_{ab}^\partial
	v_a \bar{u}_b.}
\end{array}
\end{equation}
Clearly, to find $u_h\in V_h$ such that $\T(u_h) = 0$ is equivalent to the stabilized problem \eqref{eq-steadydscrtpbm}.  
For transient problems, given also the previous time step solution $u_h^n \in V_h$, we define the operator $\T^{n+1}: V_h \rightarrow V_h'$ at every time step as
\begin{align}\label{eq-proof1bis}
\langle \T^{n+1}(z), v \rangle \doteq M_h^L(z;z,v) - M_h^L(z;u_h^n,v) + \langle \T(z), v \rangle.
\end{align} System \eqref{eq-transientdscrtpbm} can be stated in compact form as: find $u_h \in V_h$ such that $\T^{n+1}(u_h) = 0$. In the next theorem, we prove that both operators are Lipschitz continuous. We provide a sketch of the proof, since it follows the same lines as in \cite[Th. 6.1]{badia_monotonicity-preserving_2016}.
\begin{theorem}\label{th:lipschitz}
	The nonlinear operators $\T$ and $\T^{n+1}$ are Lipschitz continuous
	in $V_h$ for $q \in \mathbb{N}^+$.
\end{theorem}
\begin{proof}
	In order to prove the Lipschitz continuity, we proceed as in
	\cite{badia_monotonicity-preserving_2016}. After some manipulation, we get:
	\begin{align}
	| \langle \T(u), w \rangle -
	\langle \T(v), w \rangle | \leq
	& | \JB{K_h}(u-v,v) | + \left| \sum_{a \in \mathcal{N}_h} \sum_{b \in \mathcal{N}_h(a)} \nu_{ab}(v) \ell(a,b)
	w_a (u_b-v_b) \right| \\ & + \left|
	\sum_{a \in \mathcal{N}_h} \sum_{b \in \mathcal{N}_h(a)} (\nu_{ab}(u) - \nu_{ab}(v)) \ell(a,b)
	w_a u_b  \right| \\ & + \left|
	\sum_{a \in \mathcal{N}_h} \sum_{b \in \mathcal{N}^\partial_h(a)} (\nu_{ab}^\partial(u) - \nu_{ab}^\partial (v)) 
	w_a \bar{u}_b  \right|.
	\end{align}
	The first term is linear and continuous (see, e.g.,
	\cite{arnold_unified_2002}). We have to prove Lipschitz continuity for
	the rest of terms. We use the inverse inequalities $\| \nabla
	\varphi_a \|_K \leq C h^{-1} \| \varphi_a\|_K$ and $\| \nabla
	\varphi_a \|_F \leq C h^{-1} \| \varphi_a\|_F$ (see
	\cite{brenner_mathematical_2010}) and the fact that shape functions
	are a partition of unity ($\| \varphi_a\|_K \leq C h^{d/2}$ and $\|
	\varphi_a\|_F \leq C h^{(d-1)/2}$), to get:
	\begin{align}\label{eq-boundvisc}
	\JB{\bar{K}_h}(u_h, \bar{u}_h; \varphi_b, \varphi_a) -
	\bar{B}_h(u_h, \bar{u}_h; \varphi_b, \varphi_a) \leq C q ( h^{d-1}\| \betab
	\|_{L^\infty(\Omega)} + \mu h^{d-2}).
	\end{align}
	The rest of the proof follows the same lines as in
	\cite[Th. 6.1]{badia_monotonicity-preserving_2016} and is not included for
	the sake of conciseness. The graph-Laplacian edges for pairs $(a,b)$
	such that $\x_a \neq \x_b$ are as
	in \cite{badia_monotonicity-preserving_2016},
	using \eqref{eq-boundvisc}. The case $\x_a = \x_b$ is simpler.
	
	Lipschitz continuity for the transient problem is a consequence of the
	Lipschitz continuity of $\T$ and of the mass matrix with the
	selective mass lumping. The last property can be proved using again
	the analysis in \cite[Th. 6.1]{badia_monotonicity-preserving_2016}.
\end{proof}

Next, we show that the proposed schemes have at least one
solution. Uniqueness results could also be obtained for the
diffusion-dominated regime following the ideas in
\cite{barrenechea_2016}. \JB{In the following, we will use $C$ as a general constant that can take different values at different appearances.}

\begin{theorem}
	There is at least one solution $u_h \in V_h$ of the steady problem $\T(u_h) = 0$, and one  solution of every time step of the transient problem, i.e., $\T^{n+1}(u_h) = 0$.
\end{theorem}
\begin{proof}
	In order to prove existence of solutions, we rely on the approach in
	\cite{barrenechea_2016}, based on fixed point arguments. First, we combine the stability analysis in \cite{brezzi_discontinuous_2004}
	(for first-order hyperbolic problems) with the stability analysis for
	the interior penalty discretization of the Laplacian operator (see,
	e.g., \cite{arnold_unified_2002} for details), getting:
	\begin{align}\label{eq-proof2} \JB{K_h(z,z)} \geq C \| z\|^2_h, \quad \hbox{with } \, \| z\|^2_h \doteq \sum_{K \in \mathcal{T}_h} \mu | z
	|^2_{H^1(K)} + \sum_{F\in \JB{\mathcal{E}_h}}\left( \mu c^{\rm ip} h_F^{-1} \|
	\jump{z} \|_{L^2(F)}^2 + \| c_{\beta,F}^{\frac{1}{2}} \jump{z} \|^2_{L^2(F)}\right),
	\end{align} with $c^{\rm ip}$ big enough, and $c_{\beta,F}(\x) \doteq | \beta(\x) \cdot
	\boldsymbol{n}_F(\x)|$. On the other hand, using standard dG arguments (see \cite{arnold_unified_2002} and \cite[Prop. 3.55]{ern_theory_2004}), we
	have:
	\begin{align}\label{eq-proof3}  B_h(\bar{u},z) \leq C c^{-1} \| c_e^{\frac{1}{2}} \bar{u}
	\|^2_{L^2(\partial \Omega^-)} + C c^{-1} h^{-1} \mu c^{\rm ip} \|
	\bar{u} \|^2_{L^2(\partial \Omega)} + c \| z \|_h^2,
	\end{align}
	for $c$ arbitrarily small.
	
	We note that the nonlinear stabilization terms can be written in
	terms of the extended graph-Laplacian operator as:
	$$ \sum_{a \in \mathcal{N}_h} \sum_{b \in \mathcal{N}_h(a)} \nu_{ab} \ell(a,b) v_a z_b -
	\sum_{a \in \mathcal{N}_h} \sum_{b \in \mathcal{N}^\partial_h(a)} \nu_{ab}^\partial v_a
	\bar{u}_b = \tilde \Viscop ((z,\bar{u}),(v,0)).
	$$
	Taking $v = z$ and using \eqref{proplap} and the Cauchy-Schwarz inequality, we get:
	$$
	\tilde{\Viscop} ( (z,\bar{u}) , (z,0) ) \leq \JB{\frac{3}{2}} | z |_{\Viscop}^2  + \half | \bar{u}|^{\JB{2}}_{\Viscop^\partial}.
	$$
	Combining \eqref{eq-proof1}, \eqref{eq-proof2}, and \eqref{eq-proof3} (with $c$ small enough), we finally obtain:
	\begin{align}\label{eq-proof4}\JB{ C \langle \T (z), z \rangle} \geq \| z\|^2_h + | z |_{\Viscop}^2 \JB{+ |
		\bar{u}|^2_{\Viscop^\partial}} - \| c_e^{\frac{1}{2}} \bar{u}
	\|^2_{L^2(\partial \Omega^-)} - h^{-1} \mu c^{\rm ip} \| \bar{u}
	\|^2_{L^2(\partial \Omega)}.
	\end{align} 
	We can readily pick a $z \in V_h$ such that $\langle \T(z), z
	\rangle > 0$. Using the Brower's fixed point theorem, there exists $u_h \in V_h$ such that $\T(u_h) = 0$, and thus, solves the steady version of \eqref{eq-transientdscrtpbm} (see \cite{barrenechea_2016} for details). Existence is straightforward for the transient problem, combining the previous results with the coercivity of the mass matrix operator.
\end{proof}

\JB{
	\begin{remark}
		As a result of the previous theorem and Lemma \ref{lem-lp}, the method is linearly preserving and Lipschitz continuous. Using the ideas in \cite[Theorem 4]{barrenechea_2016}, one could prove optimal convergence in diffusion-dominated regimes.
	\end{remark}}


\section{Smoothing the shock detector}
\label{app-smoothing}

In Sect. \ref{sec-dmpproof}, we have defined $\nubc_\ab(u_h,\bcprojection)$, $\graphviscosityab(u_h,\bcprojection)$, and $\alpha_a(u_h,\bcprojection)$ in \eqref{eq-bgraphviscosity}, \eqref{eq-graphviscosity}, and \eqref{eq-shockdetector}, respectively, using non-smooth functions. The problem of using this raw definitions is that, since they are not smooth, it is difficult for the nonlinear solvers to converge. Thus, following the ideas in \cite{badia_monotonicity-preserving_2016}, we add some parameters ($\smoothalpha$,$\smoothquot$,$\smoothperturb$) and regularize the definition of non-smooth functions such as the absolute value and the maximum. In this section we will proceed to unfold all the smooth definitions to facilitate the reproducibility of the method. The resulting formulation is not only Lipschitz continuous but twice differentiable by construction. Furthermore, the smoothing involves slightly more diffusion, and it is easy to check that we keep the DMP and LED properties above. Linearity-preservation is only satisfied weakly (see \cite[Remark 7.3]{badia_monotonicity-preserving_2016}). We do not prove these results for the sake of conciseness, since the proofs are similar to the ones in \cite[Lemma 7.1]{badia_monotonicity-preserving_2016}.

We will start by introducing a couple of smoothed versions of the absolute value:
\begin{equation}
\greatabs{x} =  \sqrt{x^2 + \smoothalpha},\qquad
\lowabs{x} =  \frac{x^2}{\sqrt{x^2 + \smoothalpha}}.
\end{equation}
The value of $\smoothalpha$ is assumed to be small and is going to be specified in Sect. \ref{sec-numexp4}. For values of $x\gg \smoothalpha$, we have $\greatabs{x} \approx |x| \approx \lowabs{x}$ but always $\lowabs{x} \leq |x| \leq \greatabs{x}$. Now, we can redefine $\mean{|\gradient u_h\cdot\ru_\ab|}_\ab$ as:
\begin{equation}
\mean{\lowabs{\gradient u_h \cdot \ru_\ab} }_\ab \doteq \left\{\begin{array}{ll}
\dfrac{\lowabs{u_b - u_a}}{h_a} & \text{if} \; \x_a = \x_b, \\
\dfrac{1}{2} \left(\dfrac{\lowabs{u_b-u_a}}{|\rr_\ab|}+\dfrac{\lowabs{u_\ab^\sym - u_a}}{|\rr_\ab^\sym|}\right) & \text{otherwise}.
\end{array}\right.
\end{equation}
The quotient associated to $\alpha_a$ would read:
\begin{align}
\quotient_a= \frac{\greatabs{ \sum_{b\in\dofsupport(a)} \jump{\gradient u_h}_\ab} + \smoothquot}{\sum_{b\in\dofsupport(a)} 2\mean{\lowabs{\gradient u_h \cdot \ru_\ab}}_\ab + \smoothquot}.
\end{align}
Here $\smoothquot$ is another extra stability parameter added to ensure differentiability of $\quotient_a$ for values of $u_h$ such that the denominator is nullified. By the definition and the properties of $\greatabs{\cdot}$ and $\lowabs{\cdot}$, it is easy to prove that in the case that $u_h$ has a local discrete extremum on $a$, $\quotient_a>1$. So, since we want $\alpha_a$ to enjoy the shock detector property stated in Def. \ref{def-shockdetectorproperty}, we need to construct a twice differentiable function $\smoothquotient$ such that $\smoothquotient(x)=1$ when $x\geq 1$. To this end, we define
\begin{equation}
\smoothquotient (x) = \left\{\begin{array}{ll}
2x^4-5x^3+3x^2+x & x<1,\\
1 & x\geq 1.
\end{array}\right.
\end{equation}
Now we are able to define the smooth value of $\alpha_a$ as 
$
\tilde{\alpha}_a\doteq\left(\smoothquotient(\quotient_a)\right)^q.
$
Moreover, we have also modified the computation of the maximum in the following way:
\begin{align}
\smax (x,y) = \half\sqrt{(x-y)^2+\smoothperturb} +\half(x+y).
\end{align}
Furthermore, at boundaries $u_\ab^\sym$ is computed using the sign function which needs to be regularized too. In particular we use $\ssgn{x}\doteq x/\greatabs{x}$. Then the smooth definition of $\graphviscosityab$ in \eqref{eq-graphviscosity} for $a\in\setneighbours$ and $b\in\setneighbours(a)\setminus \{a\}$ will read
\begin{align}
&\tilde{\graphvisc}_{ab} = \smax\left(0,\smax\left(\tilde{\alpha}_{a}\spacebilinearform\left(\shap_b,\shap_a\right),\tilde{\alpha}_{b}\spacebilinearform\left(\shap_a,\shap_b\right)\right)\right) , 
\end{align}
and for $b\in\setbcneighbours(a)$
\begin{equation}
\nubc_\ab \doteq \smax (-\tilde{\alpha}_{a} B_h(\shap_b;\shap_a) ,0 ).
\end{equation}
The objective of these modifications is twofold. On the one hand, they smooth the function improving the convergence of the nonlinear iterations. On the other hand, they make the method differentiable with respect to $u_h$, and the Jacobian matrix is defined everywhere; some nonlinear iteration methods, such as Newton's method, which need to compute the Jacobian matrix of the problem, can be used. Further, the method is twice differentiable which is required to get quadratic nonlinear convergence rates with Newton's method.

In order to keep a dimensionally correct method and, at the same time, do not affect the convergence of the non-stabilized method, the parameters should scale as follows:
\begin{align}
\smoothperturb = \sigma |\betab|^2 L^{2(d-3)} h^4,
& & \smoothalpha = \tau h^2 L^{-4}, & & \smoothquot = \gamma L^{-1},
\end{align}
where $d$ is the space dimension of the problem, $L$ a characteristic length of the problem, $\tau$ and $\gamma$ have the same dimension as the unknown, and $\sigma$ is dimensionless.

\subsection{Parameters fine-tuning}

In order to find the appropriate values for all the parameters introduced before, we will check how these values affect the performance of the method. To this end, we will consider the steady ($\partial_t u = 0$) transport ($\mu= 0$) problem with no force ($\force=0$) and rotational convection $\betab=(y,-x)$:
\begin{equation}\label{eq-tuningprb}
\gradient\cdot(\betab u) = 0\qquad\qquad \rm{in}\ [0,1]\times[0,1].
\end{equation}
In the transport case, the Dirichlet boundary conditions are only imposed on the inflow boundaries, which, for this convection field, are the sides of the square $[0,1]\times [0,1]$ corresponding to $x=0$ and $y=1$. We will impose $0$ all along the side $y=1$ and the following function on the side $x=0$:
\begin{align}
\bc(0,y) = \left\{\begin{array}{ll}
1 & y\in[0.15,0.45],\\
\cos^2\left(\frac{10}{3}\pi(y-0.4)\right) & y\in[0.55,0.85],\\
0 & {\rm elsewhere}.
\end{array} \right.
\end{align}
We know that the exact solution of this problem consists of a translation of this function in the direction of the convection in such a way that on the outflow boundary corresponding to $y=0$ the solution is 
$
u(x,0) = \bc(0,x).
$
We solve this problem in a $100\times 100\;Q_1$ mesh and check the effect of the constants $\smoothperturbc$, $\smoothalphac$ and $\smoothquotc$ on the resulting outflow profile with respect to the value in the inflow boundary $x=0$, plotted in \Fig{nodallyexactkuzmin}.
\begin{figure}
	\centering
	\subfigure[Solution on inflow boundary $x=0$.]{\includegraphics[clip=true, trim = 0cm 0cm 0cm 0cm, width=7cm]{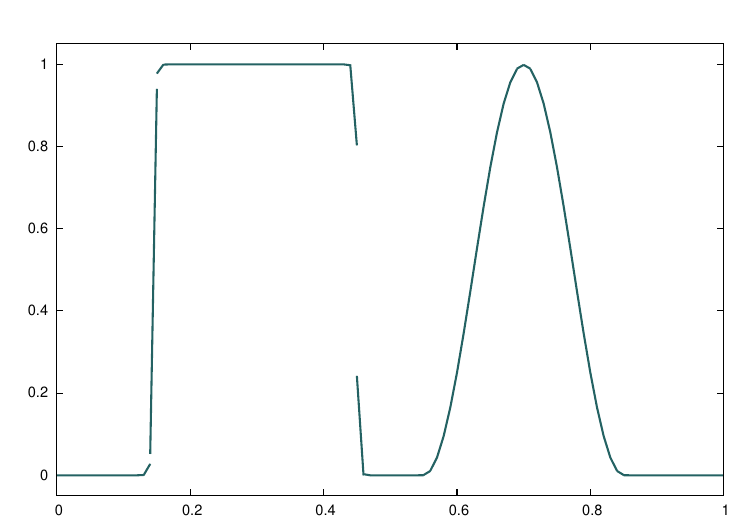}\label{fig-nodallyexactkuzmin}}%
	\subfigure[Solution on the whole domain]{\includegraphics[clip=true, trim = 10cm 4cm 2cm 0cm, width=7cm]{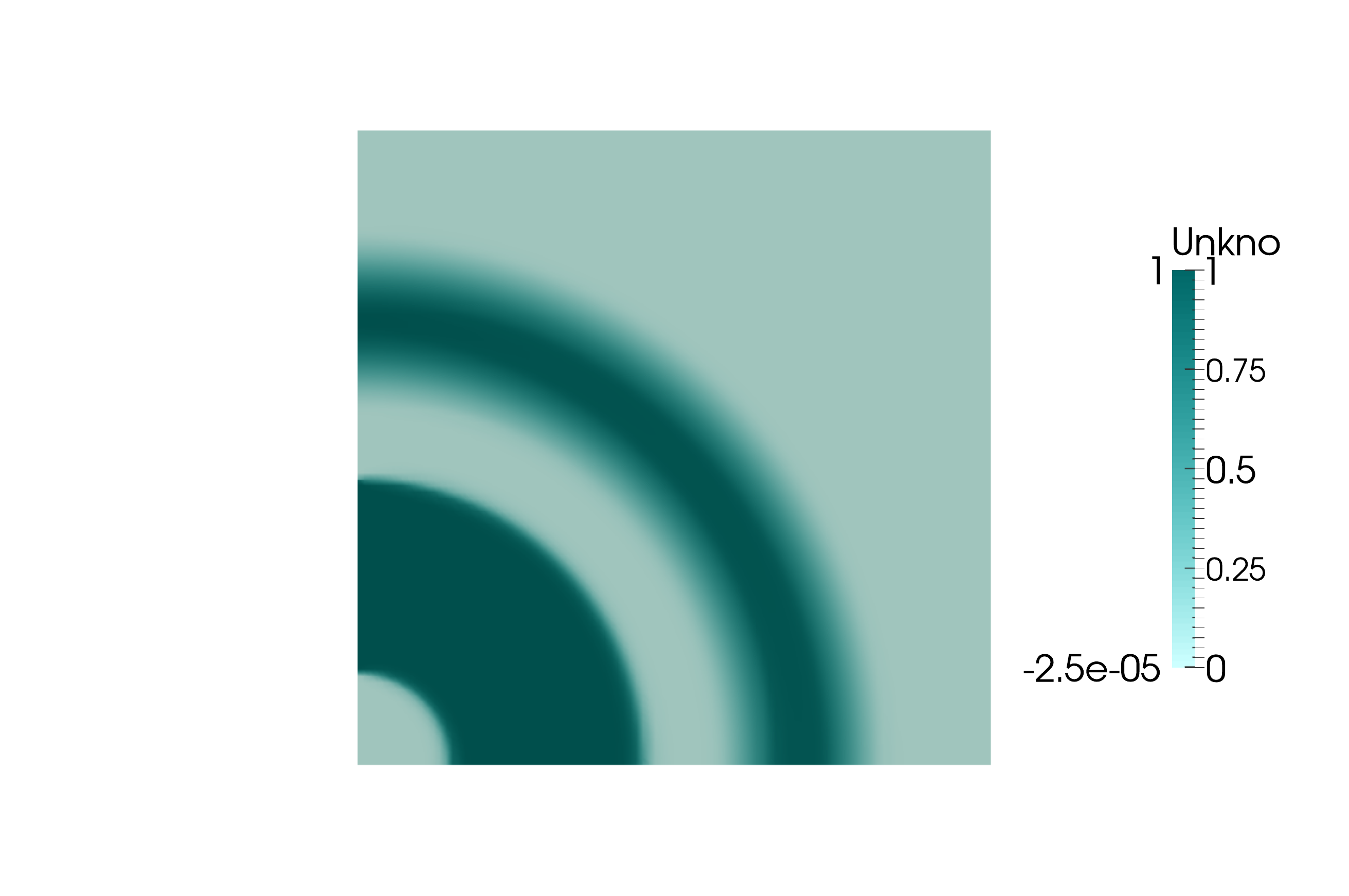}}
	\caption{Solution of the problem \eqref{eq-tuningprb} used for parameter-tuning after $100$ iterations with $q=10$, and $\smoothperturbc=\smoothalphac=\smoothquotc=0$. A $100\times 100\,Q_1$ mesh has been used. }\label{fig-kuxmin_outflow}
\end{figure} 

First of all, we can observe the dissipative effect of the parameters on the final solution. We set values of $q=\{1,2,4,10\}$, $\sigma=\{10^{-1},10^{-2},10^{-3},10^{-4}\}$, and $\tau=\sigma^2$, and fix the value of $\gamma$ to $10^{-2}$. We use Picard linearization and the nonlinear iterative scheme with the relaxation parameters proposed in \cite{john_spurious_2008}, using the same parameter values therein. In addition, we also solve all tests using a hybrid Newton-Picard method; first we use Picard to get a better starting point for Newton, particularly when the nonlinear error is lower than $10^{-2}$ we change to Newton method with line search. Note that for the hybrid scheme the total number of iterations used for comparison also include the first iterations performed with Picard method. For both nonlinear solvers the tolerance is set to be $10^{-4}$ and we allow a maximum of $500$ iterations. Whenever the solver exceeds 500 iterations we define the scheme to be not converged (NC). For both schemes the linearized system of equations is solved with a direct solver. The results are shown in \Fig{setq-newton}. It can be observed that, in order to obtain sharp solutions, it is important to use both high values of $q$ and low values of $\sigma$ and $\tau$. Nevertheless fixing $q=10$, and tuning only $\sigma$ and $\tau$, we can either obtain a method that is easy to converge, but quite dissipative, or a method that is harder to converge, but much more accurate.

\begin{figure}
	\centering
	\subfigure{\includegraphics[clip=true, trim = 0cm 0cm 0cm 0cm, width=3.5cm]{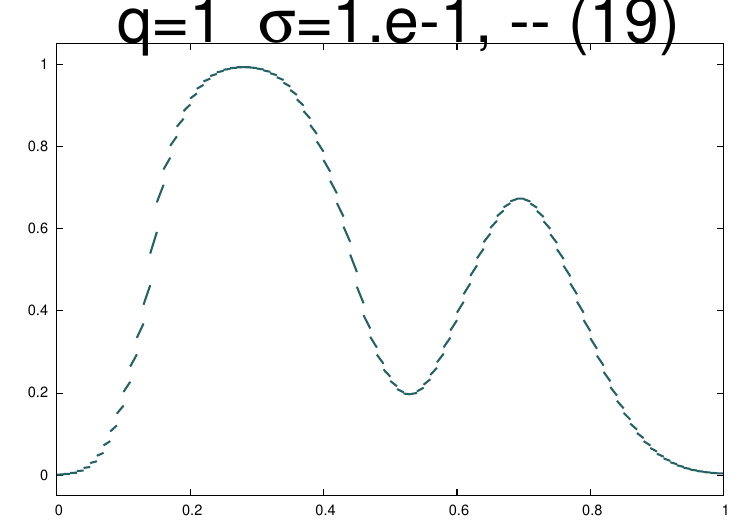}}
	\subfigure{\includegraphics[clip=true, trim = 0cm 0cm 0cm 0cm, width=3.5cm]{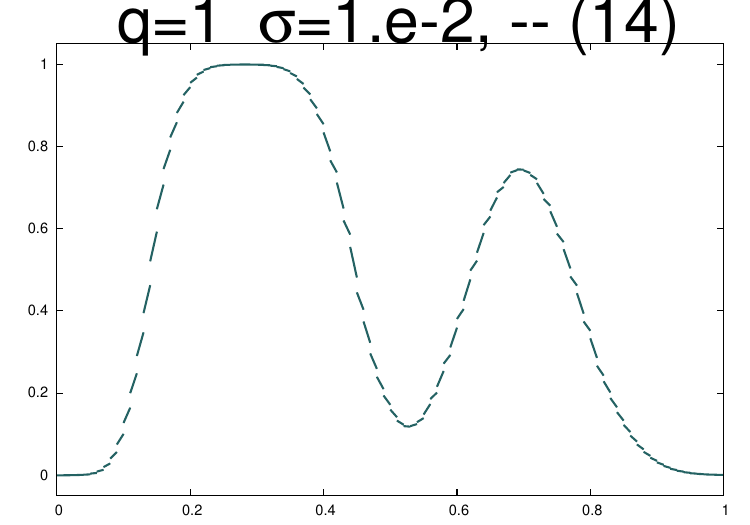}}
	\subfigure{\includegraphics[clip=true, trim = 0cm 0cm 0cm 0cm, width=3.5cm]{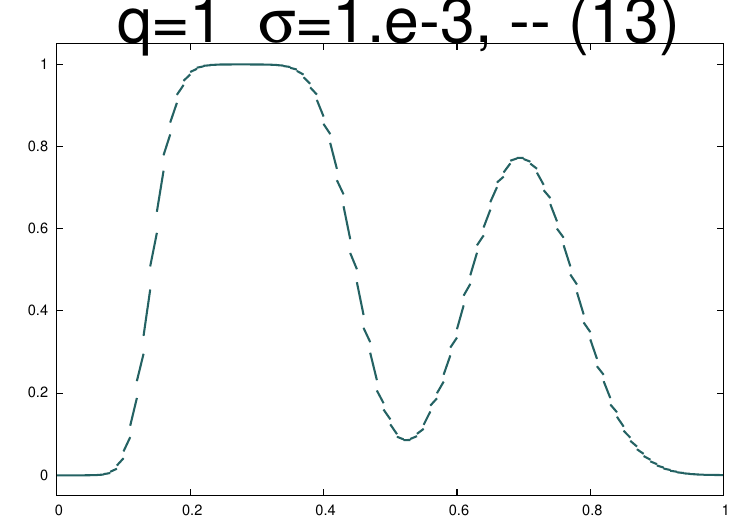}}
	\subfigure{\includegraphics[clip=true, trim = 0cm 0cm 0cm 0cm, width=3.5cm]{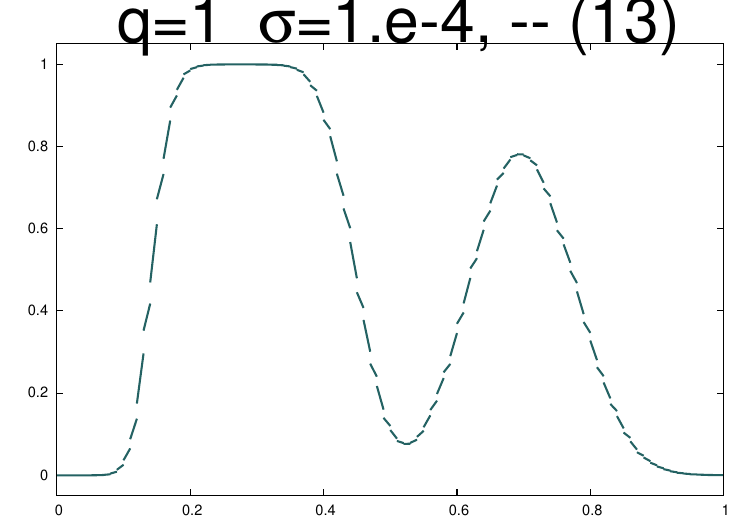}} \\
	\subfigure{\includegraphics[clip=true, trim = 0cm 0cm 0cm 0cm, width=3.5cm]{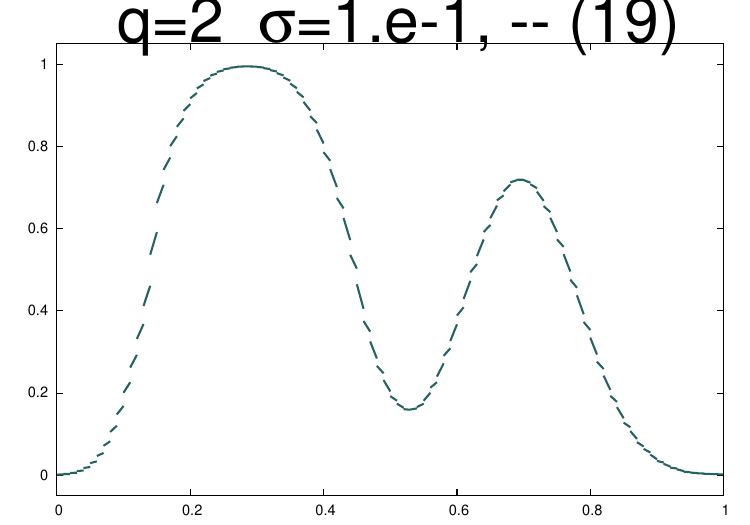}}
	\subfigure{\includegraphics[clip=true, trim = 0cm 0cm 0cm 0cm, width=3.5cm]{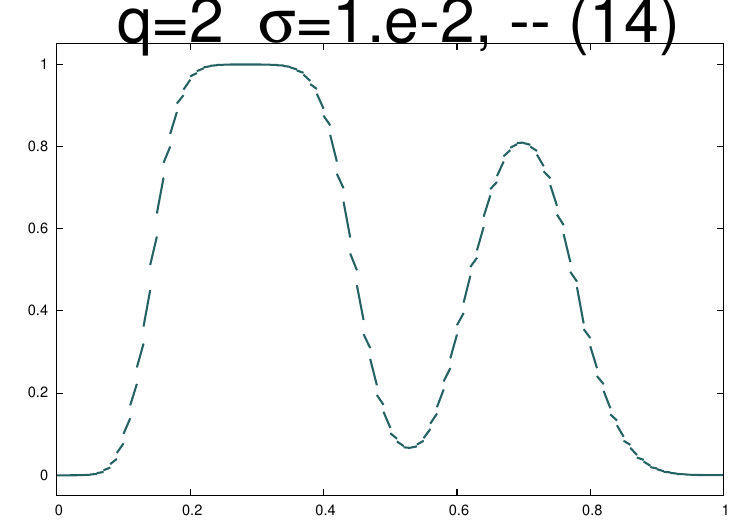}}
	\subfigure{\includegraphics[clip=true, trim = 0cm 0cm 0cm 0cm, width=3.5cm]{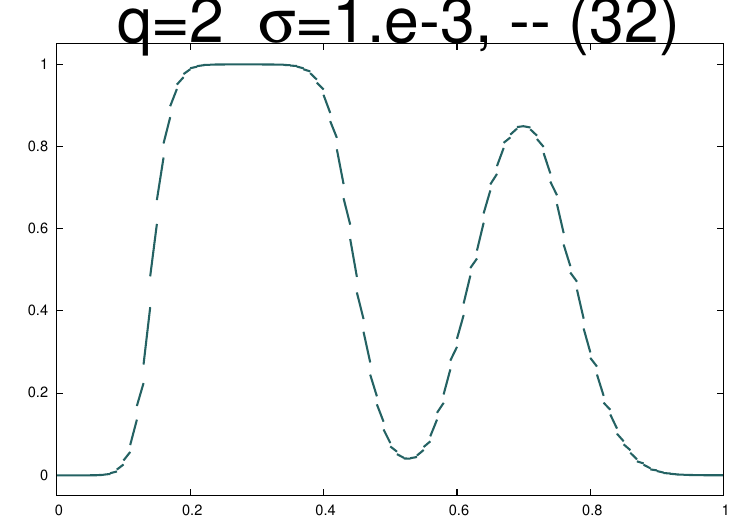}}
	\subfigure{\includegraphics[clip=true, trim = 0cm 0cm 0cm 0cm, width=3.5cm]{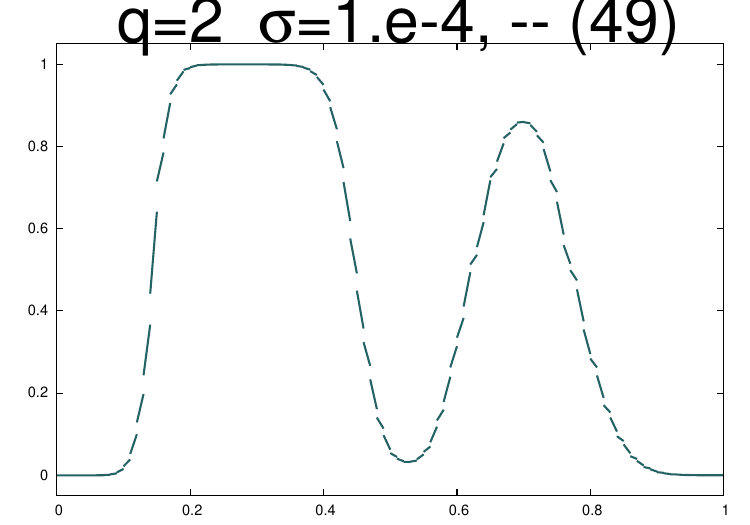}} \\
	\subfigure{\includegraphics[clip=true, trim = 0cm 0cm 0cm 0cm, width=3.5cm]{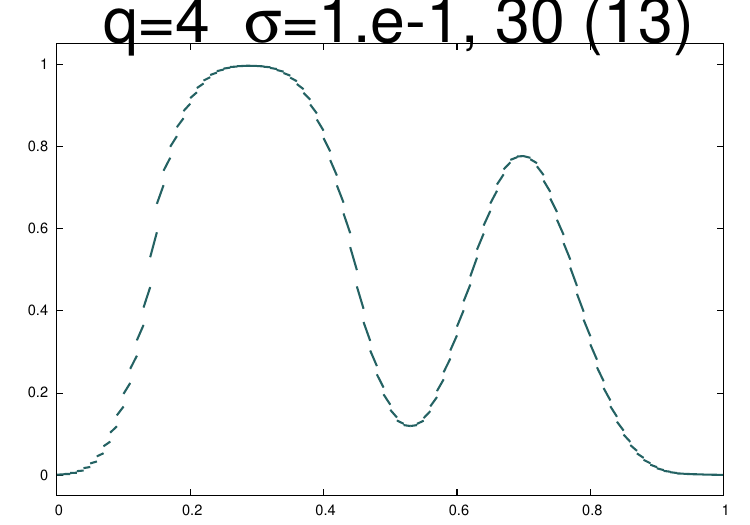}}
	\subfigure{\includegraphics[clip=true, trim = 0cm 0cm 0cm 0cm, width=3.5cm]{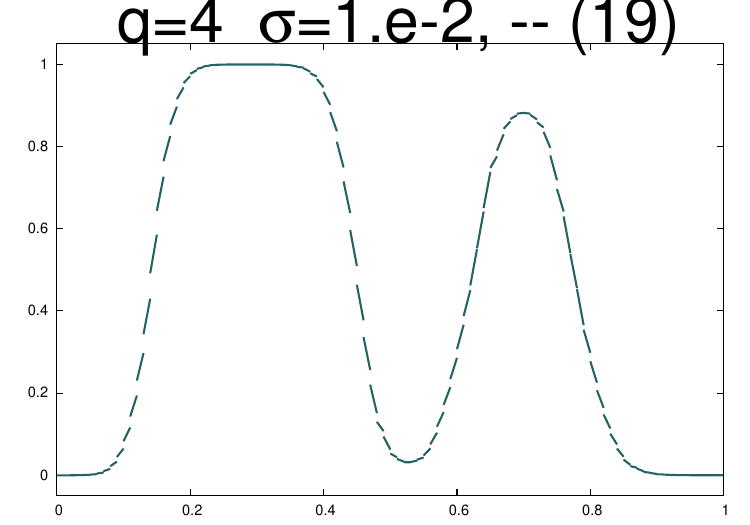}}
	\subfigure{\includegraphics[clip=true, trim = 0cm 0cm 0cm 0cm, width=3.5cm]{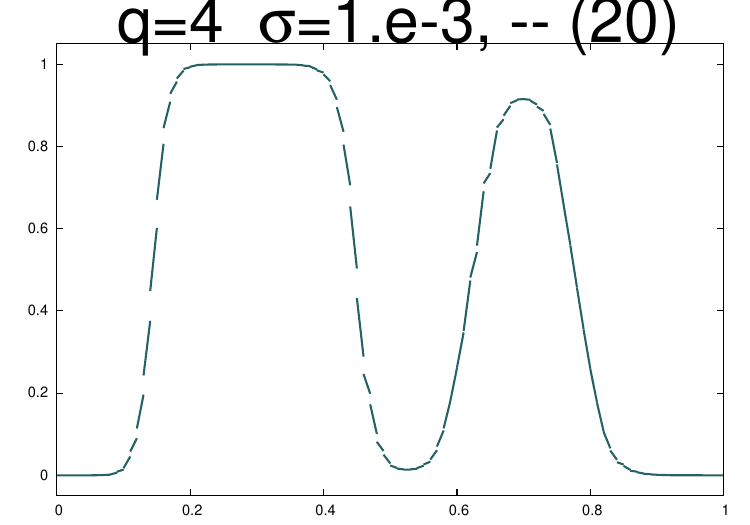}}
	\subfigure{\includegraphics[clip=true, trim = 0cm 0cm 0cm 0cm, width=3.5cm]{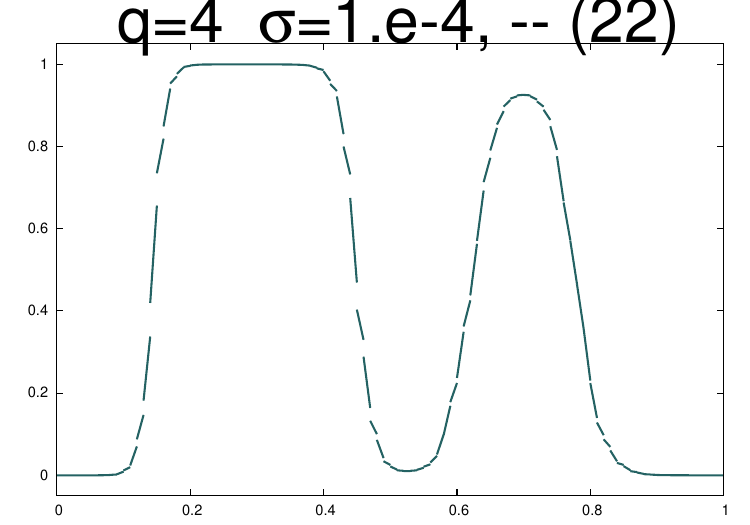}} \\
	\subfigure{\includegraphics[clip=true, trim = 0cm 0cm 0cm 0cm, width=3.5cm]{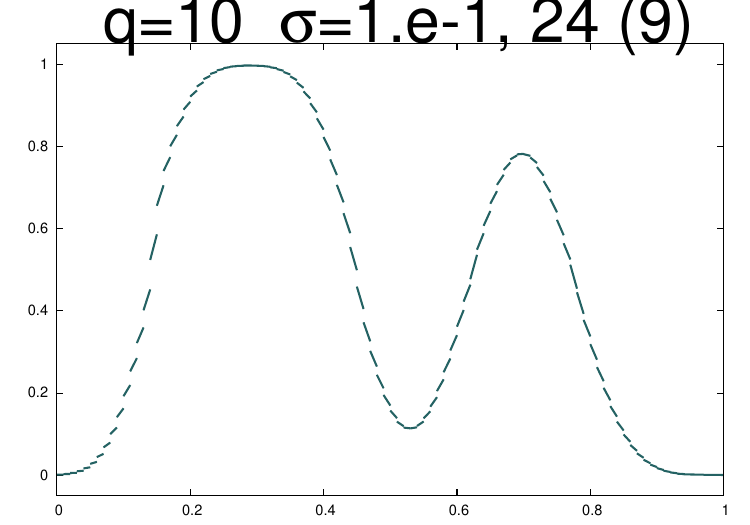}}
	\subfigure{\includegraphics[clip=true, trim = 0cm 0cm 0cm 0cm, width=3.5cm]{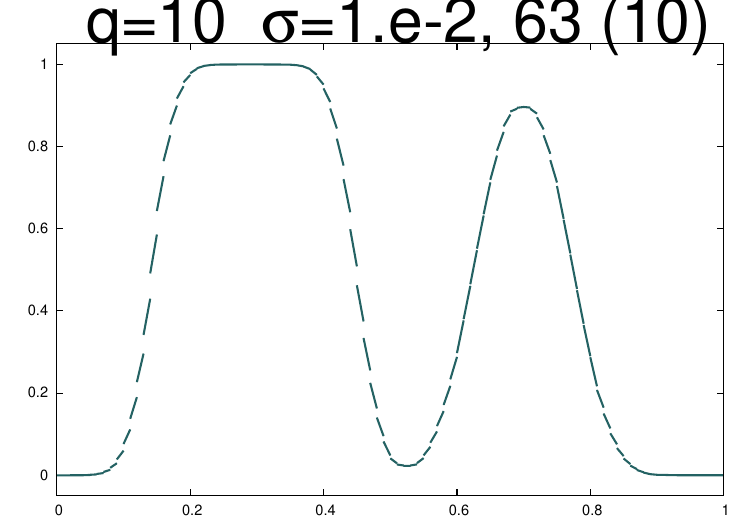}}
	\subfigure{\includegraphics[clip=true, trim = 0cm 0cm 0cm 0cm, width=3.5cm]{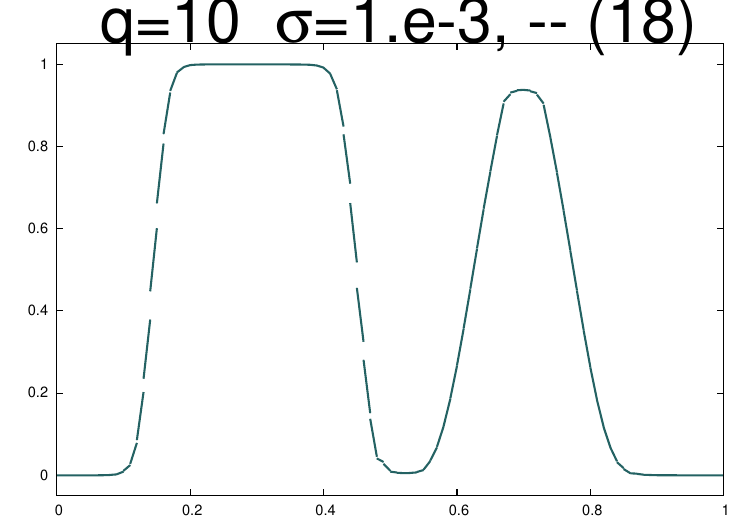}}
	\subfigure{\includegraphics[clip=true, trim = 0cm 0cm 0cm 0cm, width=3.5cm]{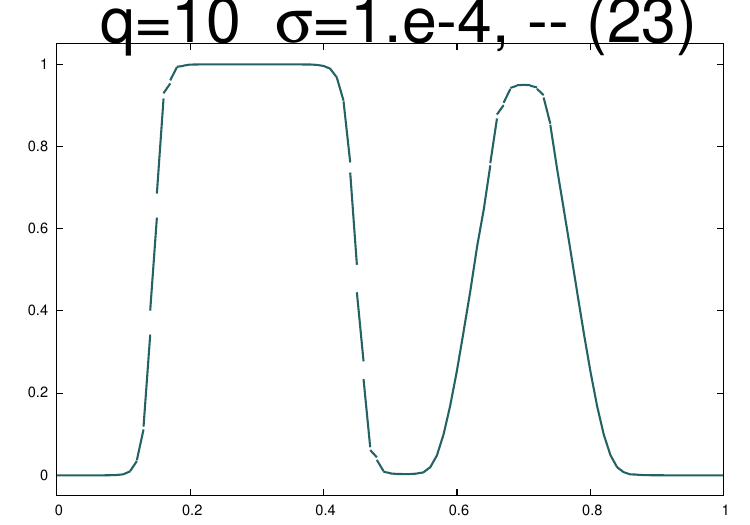}}\\
	\caption{Profile of the solution of \eqref{eq-tuningprb} on the outflow side ($y=0$) for different values of $q$, $\smoothperturbc$, $\smoothalphac=\smoothperturbc^2$, and  $\smoothquotc=10^{-2}$. Each figure title indicates the value of $q$, $\sigma$, the number of nonlinear iterations for Picard and the hybrid scheme (in brackets), (-\,-) means ``not converged''. A $100\times 100$ mesh has been used. }\label{fig-setq-newton}
\end{figure}

For the moment, we have fixed the relation between $\smoothperturbc$
and $\smoothalphac$. In the next test we fix $q=10$, $\gamma=10^{-2}$,
and different values for $\smoothalphac$ and $\smoothperturbc$. In
particular we will use
$\smoothalphac=\{10^{-1},10^{-2}10^{-4},10^{-8}\}$ and
$\smoothperturbc=\{10^{-1},10^{-2},10^{-3},10^{-4}\}$. We will use the
same nonlinear solvers as before. The obtained results are shown in
\Fig{setsigmatau2}. Values of $\smoothperturbc$ around $10^{-2}$ and
$\smoothalphac$ above $10^{-4}$ are needed to ensure Picard
convergence. In the case of the hybrid scheme, we do not observe much
difference in terms of the number of iterations required to converge
for different values of $\smoothalphac$. Nevertheless, as
$\smoothperturbc$ is reduced the method requires a slightly larger
number of iterations.
\begin{figure}
	\centering
	\subfigure{\includegraphics[clip=true, trim = 0cm 0cm 0cm 0cm, width=3.5cm]{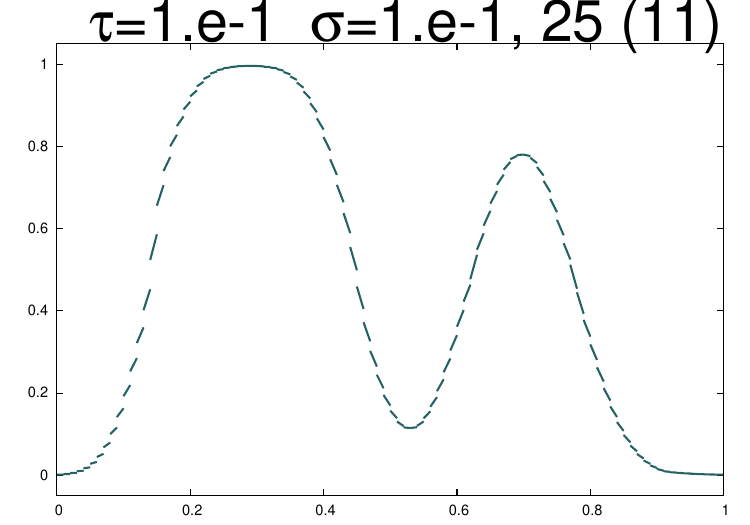}}%
	\subfigure{\includegraphics[clip=true, trim = 0cm 0cm 0cm 0cm, width=3.5cm]{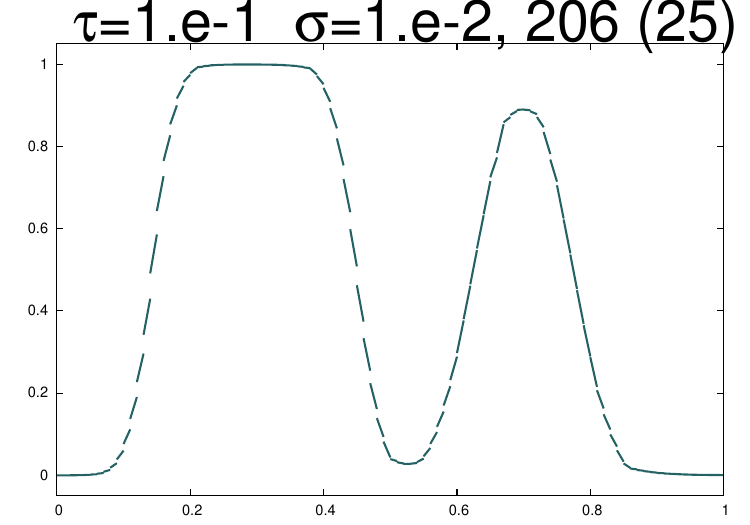}}%
	\subfigure{\includegraphics[clip=true, trim = 0cm 0cm 0cm 0cm, width=3.5cm]{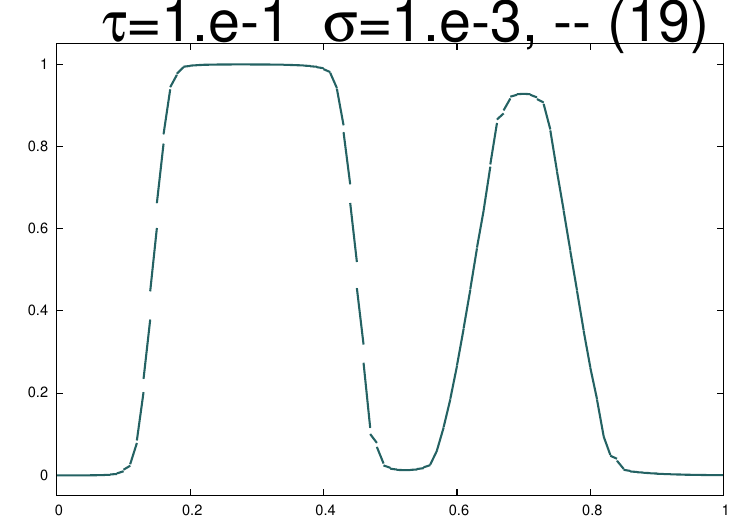}}%
	\subfigure{\includegraphics[clip=true, trim = 0cm 0cm 0cm 0cm, width=3.5cm]{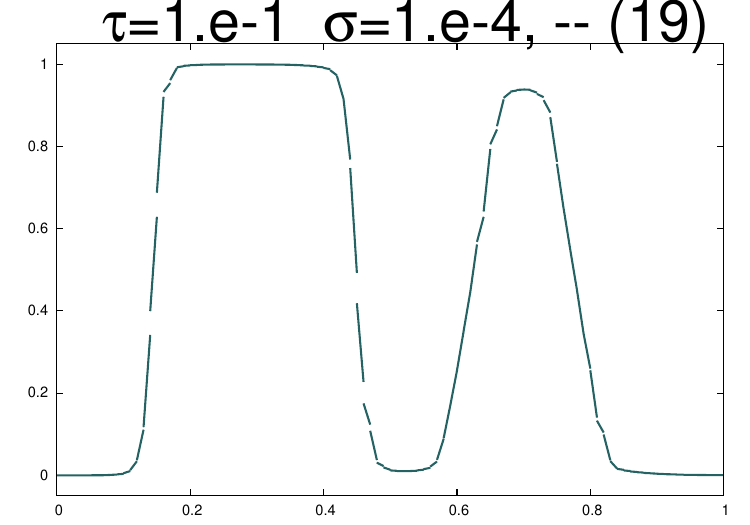}}
	\subfigure{\includegraphics[clip=true, trim = 0cm 0cm 0cm 0cm, width=3.5cm]{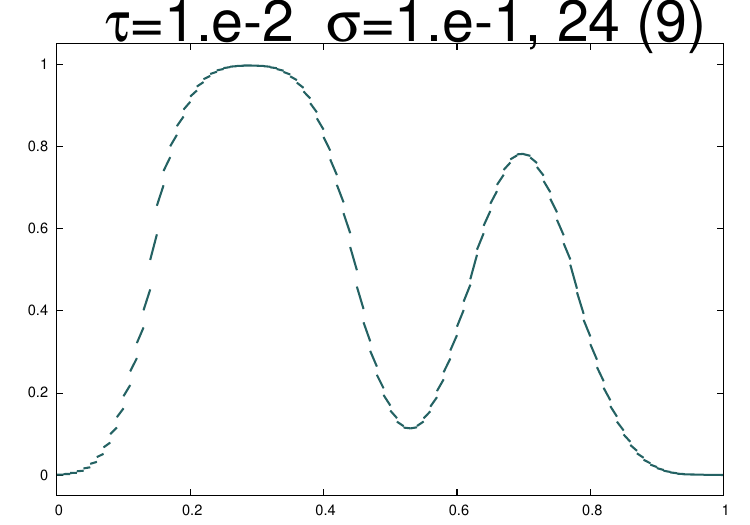}}%
	\subfigure{\includegraphics[clip=true, trim = 0cm 0cm 0cm 0cm, width=3.5cm]{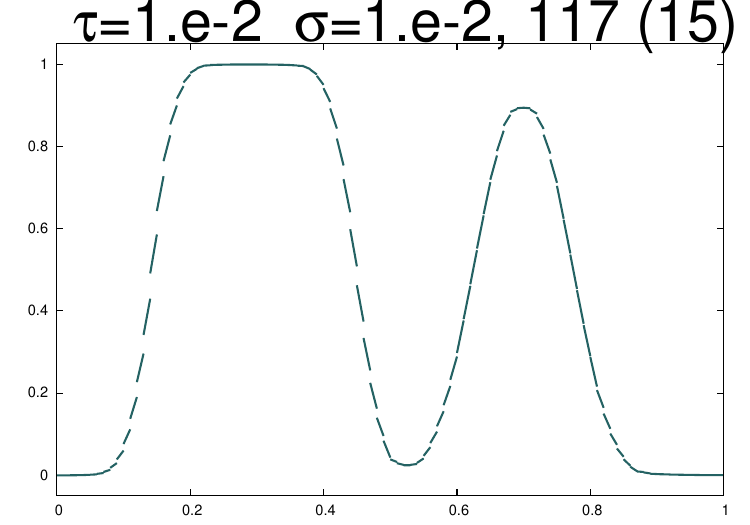}}%
	\subfigure{\includegraphics[clip=true, trim = 0cm 0cm 0cm 0cm, width=3.5cm]{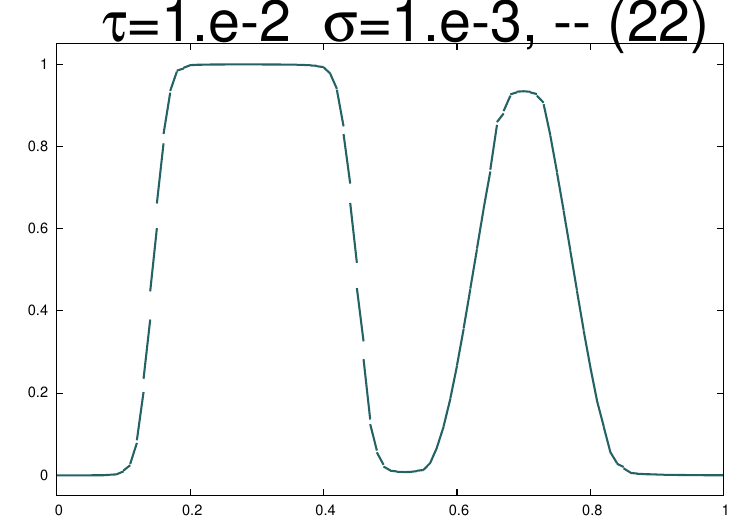}}%
	\subfigure{\includegraphics[clip=true, trim = 0cm 0cm 0cm 0cm, width=3.5cm]{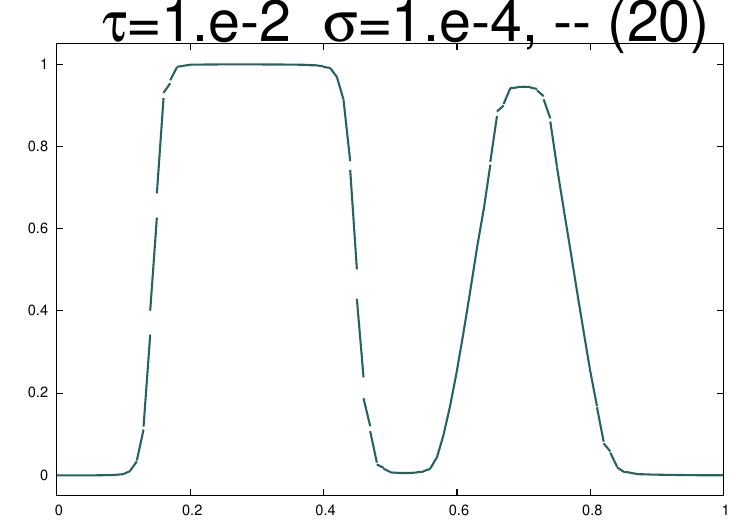}}
	\subfigure{\includegraphics[clip=true, trim = 0cm 0cm 0cm 0cm, width=3.5cm]{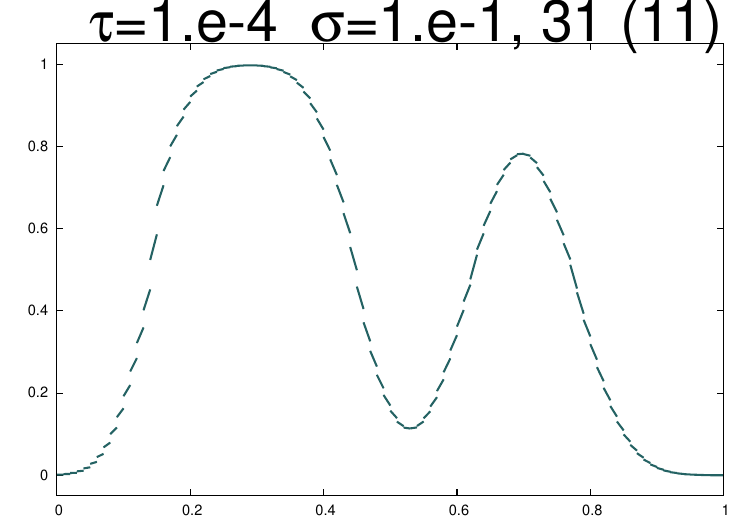}}%
	\subfigure{\includegraphics[clip=true, trim = 0cm 0cm 0cm 0cm, width=3.5cm]{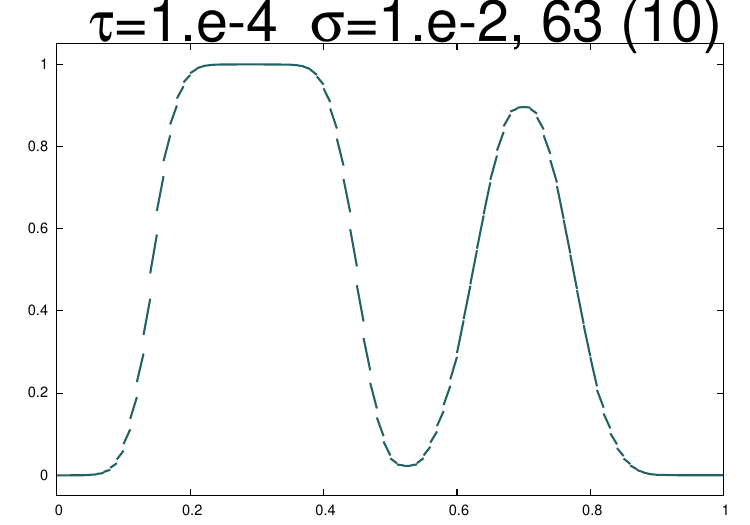}}%
	\subfigure{\includegraphics[clip=true, trim = 0cm 0cm 0cm 0cm, width=3.5cm]{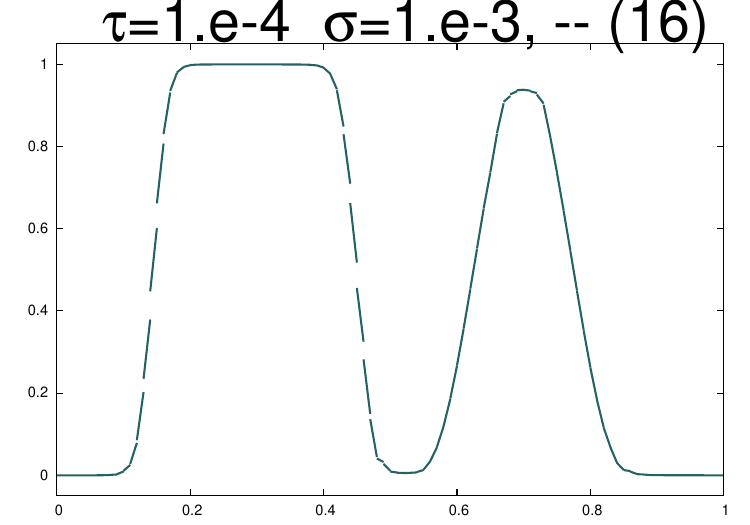}}%
	\subfigure{\includegraphics[clip=true, trim = 0cm 0cm 0cm 0cm, width=3.5cm]{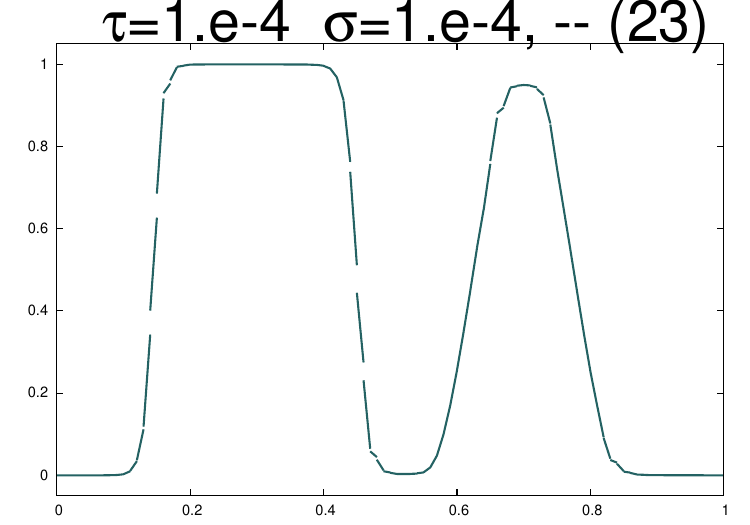}}
	\subfigure{\includegraphics[clip=true, trim = 0cm 0cm 0cm 0cm, width=3.5cm]{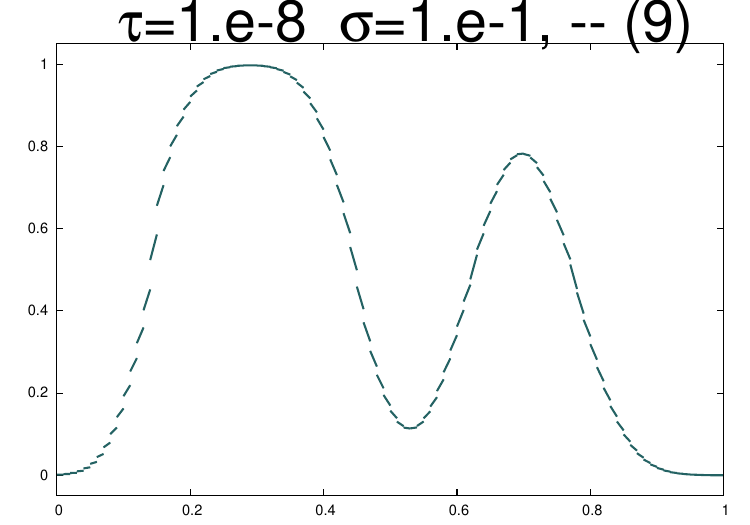}}%
	\subfigure{\includegraphics[clip=true, trim = 0cm 0cm 0cm 0cm, width=3.5cm]{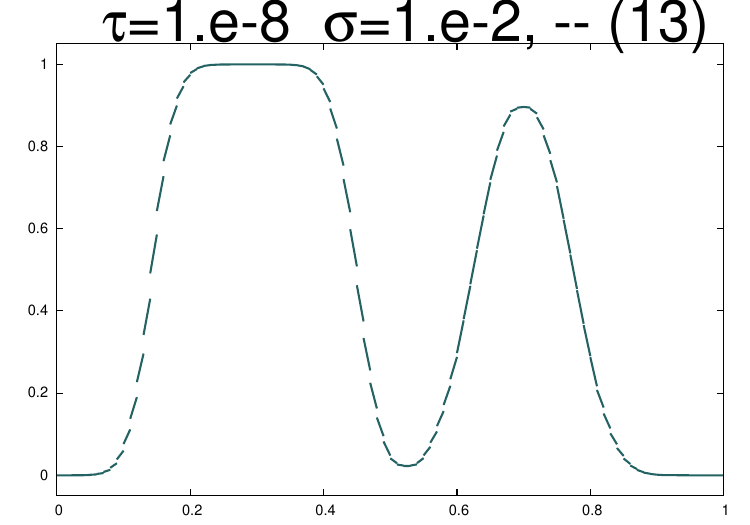}}%
	\subfigure{\includegraphics[clip=true, trim = 0cm 0cm 0cm 0cm, width=3.5cm]{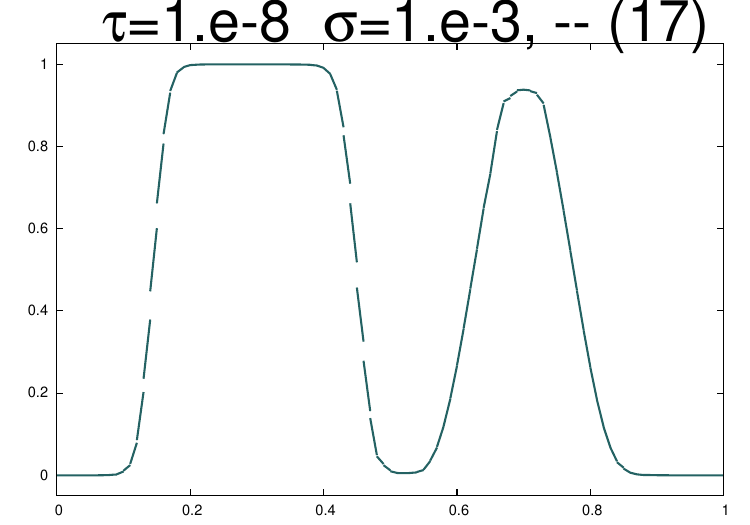}}%
	\subfigure{\includegraphics[clip=true, trim = 0cm 0cm 0cm 0cm, width=3.5cm]{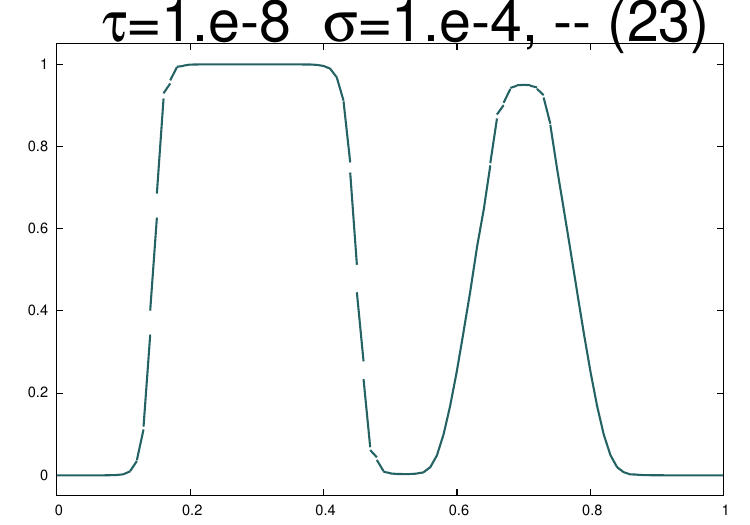}}
	\caption{Profile of the solution of \eqref{eq-tuningprb} on the outflow side ($y=0$) for different values of $\smoothperturbc$, $\smoothalphac$, $q=10$, and  $\smoothquotc=10^{-2}$. Each figure title indicates the value of $\smoothalphac$, $\smoothperturbc$, the number of nonlinear iterations for Picard and the hybrid scheme (in brackets), (-\,-) means ``not converged''. A $100\times 100$ mesh has been used. }\label{fig-setsigmatau2}
\end{figure}

Finally, we want to fine-tune $\smoothquotc$. To do so, we will fix the values of $\smoothperturbc=10^{-2}$ and $\smoothalphac=10^{-4}$ and reduce the value of $\smoothquotc$ from $10^{-4}$ to $0$. We can check in \Fig{setgamma} how, even with $\smoothquotc=0$, the solution is able to converge, but the number of iteration is larger than between $\gamma=10^{-4}$ and $\gamma=10^{-12}$ whereas the solution is practically the same. In our numerical experiments we will work with the smooth and non-smooth version and with both nonlinear solvers. Thus, when comparing results we are interested in solutions that converge in a reasonable amount of time steps (we take $\gamma=10^{-2}$).

In the transient examples, we will only use Picard linearization, allowing us take $\gamma=0$ to see whether we can obtain a sharper solution. Additionally, we will let $Q$, the exponent of $\alpha_a$ for the perturbation of the mass matrix in \eqref{eq-lumpedmass}, be $Q=+\infty$; meaning that the matrix is only perturbed when $\alpha_a=1$. Nevertheless, we recall that we can do so because the nonlinear iterative method that we use does not need the stabilization to be differentiable; if that is not the case, $\gamma$ must be greater than $0$ and $Q<+\infty$. Summarizing, in view of the results obtained, we will use $q=10$, $Q=\infty$, $\smoothperturbc=10^{-2}$, $\smoothalphac\leq 10^{-6}$ and either $\smoothquotc=10^{-2}$ or $\smoothquotc=0$ in the oncoming transient numerical experiments.
\begin{figure}
	\centering
	\subfigure{\includegraphics[clip=true, trim = 0cm 0cm 0cm 0cm, width=3.5cm]{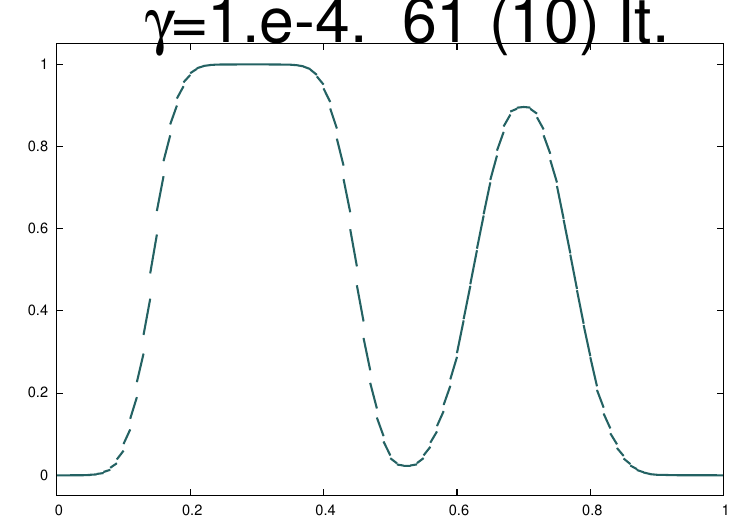}}%
	\subfigure{\includegraphics[clip=true, trim = 0cm 0cm 0cm 0cm, width=3.5cm]{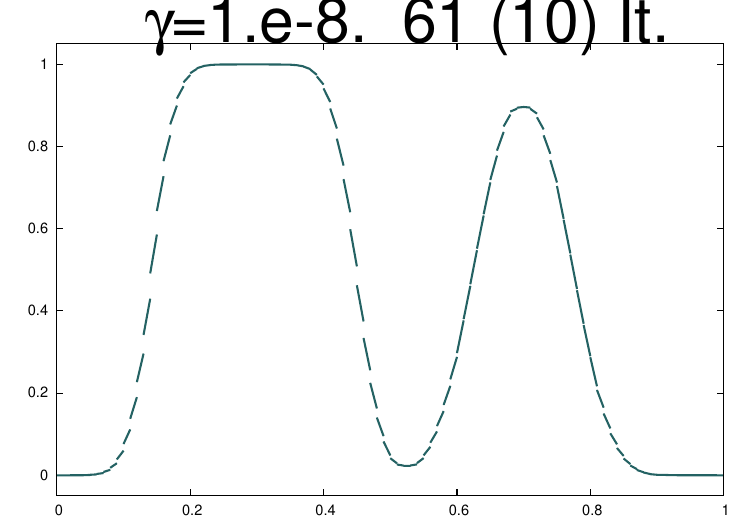}}%
	\subfigure{\includegraphics[clip=true, trim = 0cm 0cm 0cm 0cm, width=3.5cm]{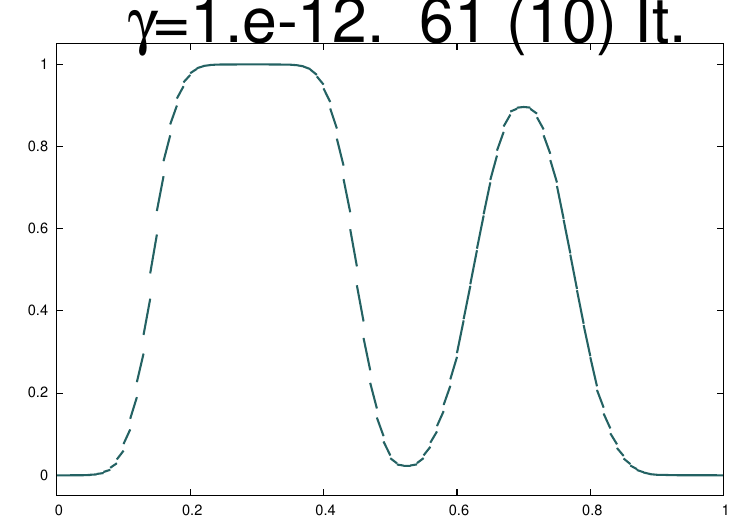}}%
	\subfigure{\includegraphics[clip=true, trim = 0cm 0cm 0cm 0cm, width=3.5cm]{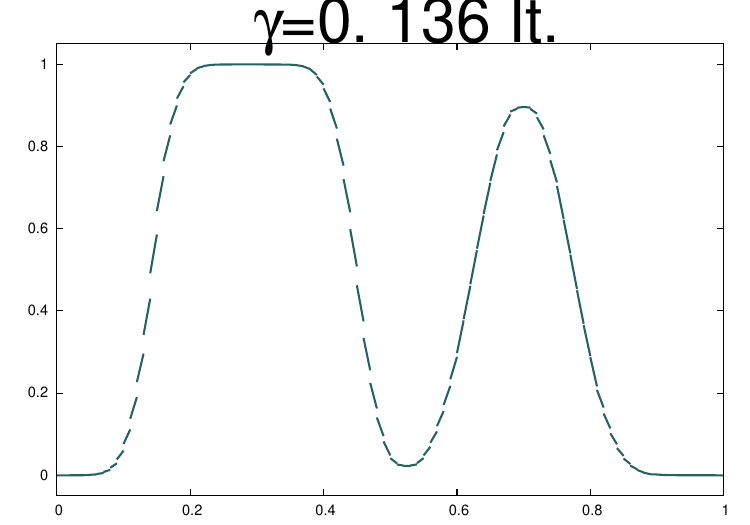}}
	\caption{Profile of the solution of \eqref{eq-tuningprb} on the outflow side ($y=0$) for different values of $\smoothquotc$, $q=10$, $\smoothperturbc=10^{-2}$, and $\smoothalphac=10^{-4}$. Each figure title indicates the value of $\smoothquotc$, and the number of nonlinear iterations for Picard scheme \JB{and the hybrid scheme (in brakets). For $\gamma=0$ the number of iterations is not given for the hybrid scheme since the stabilization is not smooth and the Jacobian might be undefined at some points}. A $100\times 100$ mesh has been used.}\label{fig-setgamma}
\end{figure}

\section{Numerical experiments}\label{sec-numexp4}

In this section, we are interested in showing how the method previously introduced deals with a set of numerical experiments. We recall that, since we are using isotropic uniform meshes, the value of the characteristic length $h_K$ of each element $K$ can be computed as the length of the edges of the squares.

\subsection{Convergence to a smooth solution}
In this first experiment we want to determine the convergence of the method towards a smooth solution that has maxima and minima inside the domain ($\alpha_a$ reaches the value $1$ in some regions in the domain). We compare the performance of the original interior penalty dG method against our dG method with artificial diffusion, both with and without smoothing (allowing a maximum of $100$ iterations). The steady problems we solve to this end are the following:
\begin{equation}
\left\{
\begin{array}{rcll}
- \mu\Delta u + \gradient\cdot(\betab u)  &=& -4\pi^2\sin\left(2\pi\left(x-\frac{y}{\tan \theta}\right)\right)\left(1+\frac{1}{\tan^2\theta}\right) & \rm{in}\ \Omega=[0,1]\times[0,1], \\
u(x,y)&=&\sin\left(2\pi\left(x-\frac{y}{\tan \theta}\right)\right) & \rm{on}\ \partial\domain,
\end{array}\right.
\end{equation}
and
\begin{equation}
\left\{
\begin{array}{rcll}
\gradient\cdot(\betab u)  &=& 0 & \rm{in}\ \Omega=[0,1]\times[0,1], \\
u(x,y)&=&\sin\left(2\pi\left(x-\frac{y}{\tan \theta}\right)\right) & \rm{on}\ \partial\domain,
\end{array}\right.
\end{equation}
with $\betab=(\cos(\theta),\sin(\theta))$ and $\theta=\pi/3$. In both cases the exact solution is $u(x,y) = \sin(2\pi(x-\frac{y}{\tan \theta}))$. We have second-order convergence of the non-stabilized method with (bi)linear finite elements and we would like to preserve it for the stabilized one when applied to smooth solutions. This is indeed what happens. When we consider the problem with diffusion $\mu=1$ in which the contributions to the matrix are dominated by the diffusion term for the finest meshes, the error is almost the same for both the smoothed and the non-smoothed versions. \JB{
	It is important to note that the nonlinear solver is not able to converge for the finest meshes when using the non-smooth stabilization. This introduces an additional error responsible of the convergence degeneration observed at \Fig{smooth_solution4_mu1}.
	In addition, it is worth noting that} the requirement of the boundary condition extrapolation to achieve optimal convergence. Tests without the extrapolation in Remark \ref{rmk-bcextrapolation}, i.e., $u_\ab^\sym=\bc_a$, show degenerated convergence rates. This correction is more important as the gradients on the boundary become larger and the jumps smaller. As opposed, when working with pure transport (see \Fig{smooth_solution4_mu0}), the effect of this extrapolation becomes negligible. In any case, the order of convergence is maintained. For the pure convection test, the smoothing of the parameters add an extra error to the computed solution, as expected, since it implies more artificial diffusion.
\begin{figure}
	\centering
	\includegraphics[clip=true, trim = 0cm 0cm 0cm 0cm]{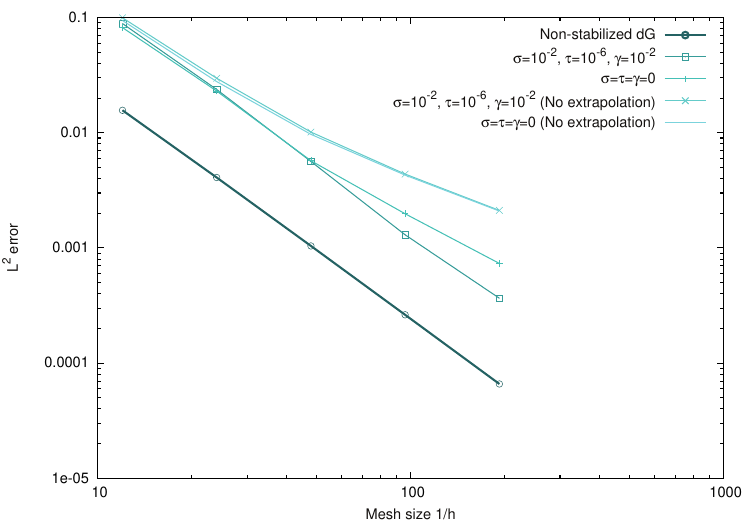}
	\caption{$L^2$ error convergence test for a convection-diffusion problem ($\mu=1$) with a smooth solution. Different choices of stabilization parameters have been tested.}\label{fig-smooth_solution4_mu1}
\end{figure}
\begin{figure}
	\centering
	\includegraphics[clip=true, trim = 0cm 0cm 0cm 0cm]{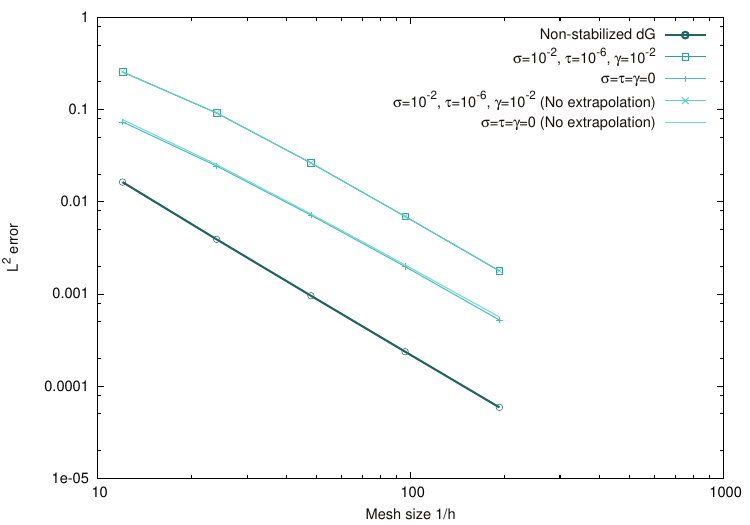}
	\caption{$L^2$ error convergence test for a pure convection problem ($\mu=0$) with a smooth solution. Different choices of stabilization parameters have been tested.}\label{fig-smooth_solution4_mu0}
\end{figure}

\subsection{DMP-preservation}
In order to test the performance of the method in terms of DMP-preservation, we solve the steady problem
\begin{equation} \label{eq-sharpl}
\left\{
\begin{array}{rcll}
- 10^{-4}\Delta u + \gradient\cdot(\betab u)  &=&0 & \rm{in}\ \Omega=[0,1]\times[0,1], \\
u(x,0)&=&0 & x\in[0,1],\\
u(x,1)&=&1 & x\in[0,1],\\
u(0,y)&=&\frac{1}{2}+\frac{1}{\pi}\arctan(10^4(y-0.7)) & y\in(0,1),\\
u(1,y)&=&0 & y\in(0,1),\\
\end{array}\right.
\end{equation}
with $\betab=(\cos(\pi/3),-\sin(\pi/3))$. As the problem is convection-dominated, the expected result is a propagation in the direction defined by $\betab$ of the profile imposed in the inflow boundary $x=0$. It is well known that when the method is not stabilized, it leads to a solution that has strong oscillations around the internal and boundary layers. We expect our method to control such spurious oscillations, as already proved in the numerical analysis.

We use a $100\times100\,Q_1$ \JB{Cartesian} mesh, set the tolerance of the nonlinear iterations to $10^{-4}$, and allow a maximum of $500$ iterations. We plot the maximum oscillation, defined as 
\begin{equation}\label{eq-osc}
{\rm OSC}=\max\{0,-\min_{x\in\domain} u_h(x),\max_{x\in\domain} (u_h(x)-1),\},
\end{equation}
at each iteration. Both the smooth and the non-smooth methods are tested and both nonlinear solvers above are used for the smooth version. The results \JB{in \Fig{maxosc}} show that only the hybrid method is able to converge. \JB{Let us remark that Picard's method has not reeached convergence (having a maximum of 500 iterations) in any case, even though the non-converged solution with smooth stabilization satisfies the DMP up to machine precision.} On the contrary, the hybrid method satisfies the DMP up to the tolerance, but it only needs 15 iterations to converge. The results obtained are plotted in \Fig{layer} and are as sharp as the non-smooth ones.
\begin{figure}
	\centering
	\includegraphics[clip=true, trim = 0cm 0cm 0cm 0cm]{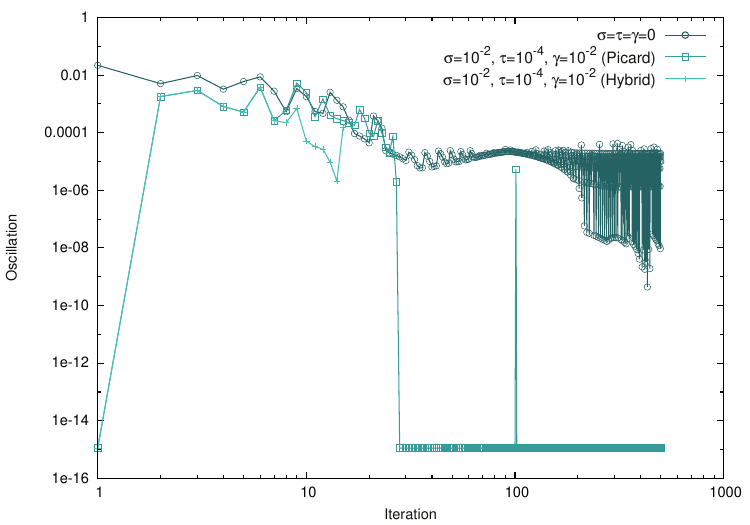}
	\caption{Maximum oscillation as defined in \eqref{eq-osc} for each nonlinear iteration performed for solving \eqref{eq-sharpl}. The results for both a smooth and a non-smooth stabilization, as well as for different nonlinear schemes are depicted. A $100\times 100\,Q_1$ mesh has been used for the tests. }\label{fig-maxosc}
\end{figure} 
\begin{figure}
	\centering
	\subfigure[$\smoothperturb=10^{-2}$; $\smoothalpha=10^{-4}$; $\smoothquot=10^{-2}$. 13 Iterations.]{\includegraphics[width=7cm]{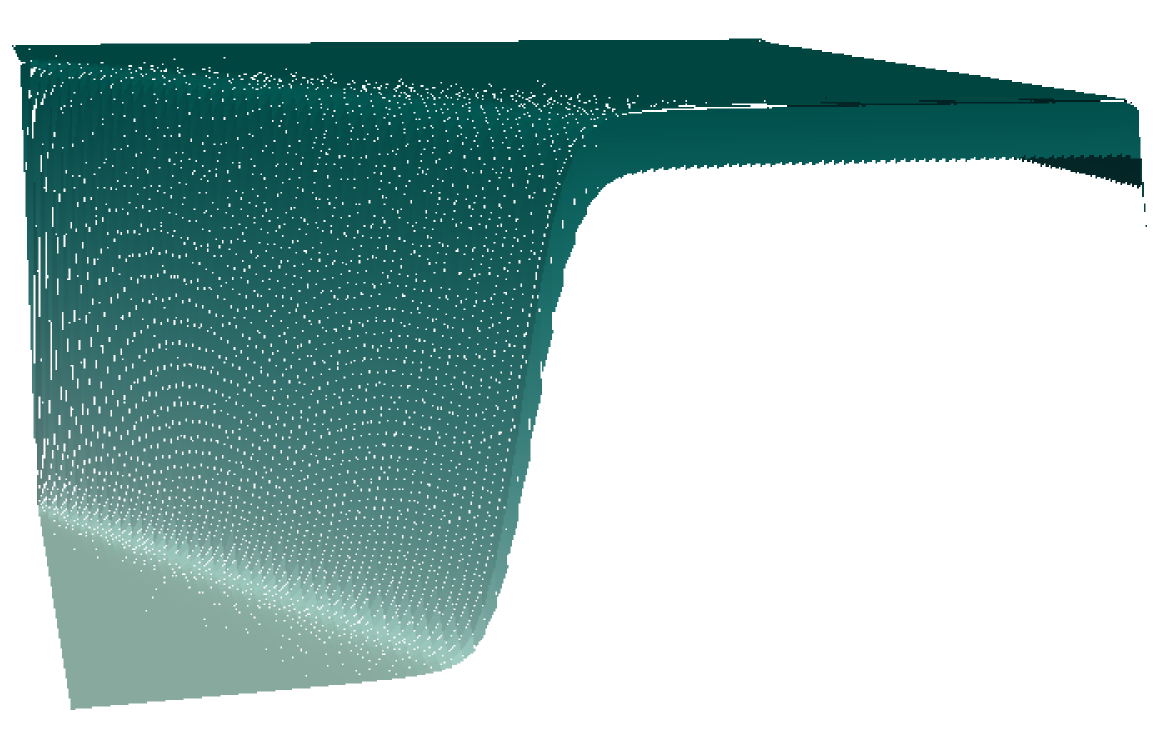}}%
	\subfigure[$\smoothperturbc=\smoothalphac=\smoothquotc=0$. NC.]{\includegraphics[width=7cm]{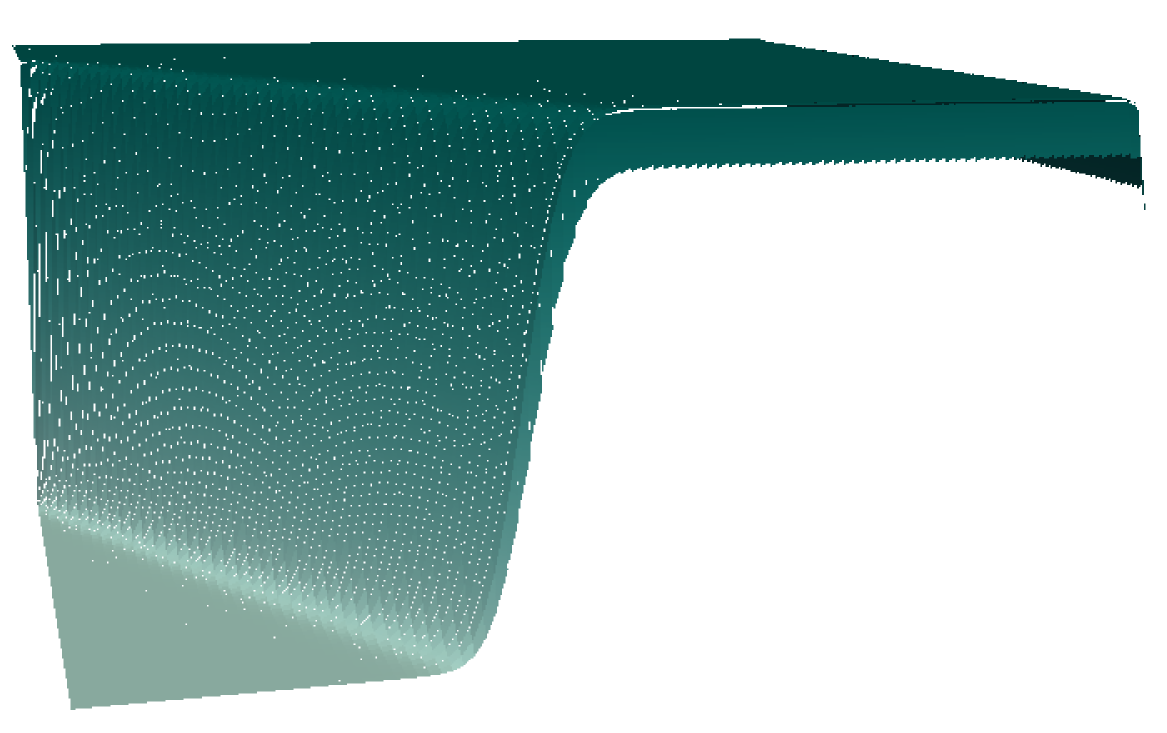}}%
	\caption{Solution of problem \eqref{eq-sharpl} using both the smooth and the non-smooth version of the stabilization. A $100\times 100$ mesh have been used. }\label{fig-layer}
\end{figure}

\JB{We know from the numerical analysis that the method preserves the DMP on unstructured meshes. In the following test we are checking this result and analyzing if it affects to the accuracy of the method. To this end, we solve again the previous problem using an unstructured mesh with an element size $h\approx 10^{-2}$ (see \Fig{layer-uns} (a)). The solution is depicted in \Fig{layer-uns}(b). It can be seen that the accuracy is not affected by reasonable mesh perturbations. Moreover, we compare the effect of the smoothing in the case of unstructured meshes. It can be observed in \Fig{maxosc-uns} that the stabilization method is minimally affected by the mesh topology. Nevertheless, it is worth noting that the number of iterations slightly increases for the hybrid method, from 15 to 17. In any case, we achieve the same conclusions as when using the structured mesh. The hybrid method is able to converge and it satisfies de DMP up to the tolerance. The relaxed Picard it is not able to converge, even though it leads to non-converged solutions that satisfy the DMP for the smooth case at the lasts iterations.}

\begin{figure}
	\centering
	\subfigure[Unstructured mesh used]{\includegraphics[width=55mm]{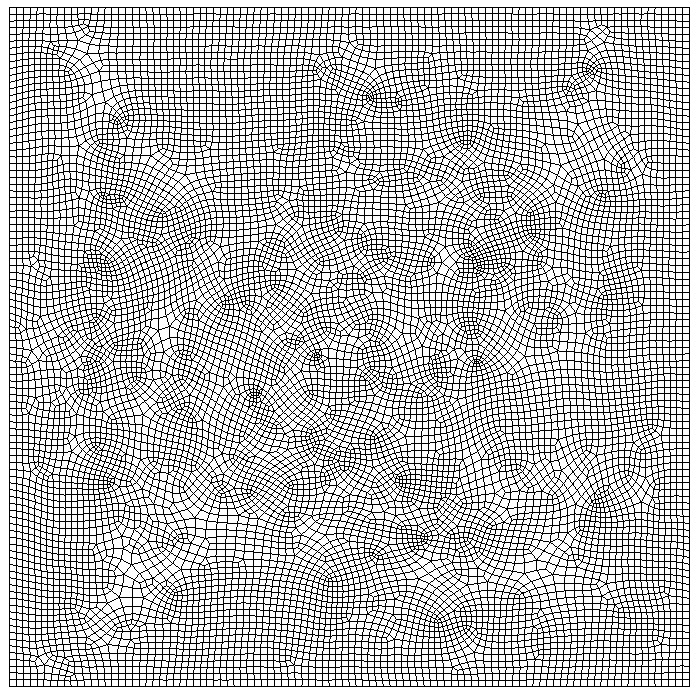}}%
	\hspace{10mm}
	\subfigure[$\smoothperturb=10^{-2}$; $\smoothalpha=10^{-4}$; $\smoothquot=10^{-2}$. 17 Iterations.]{\includegraphics[width=7cm]{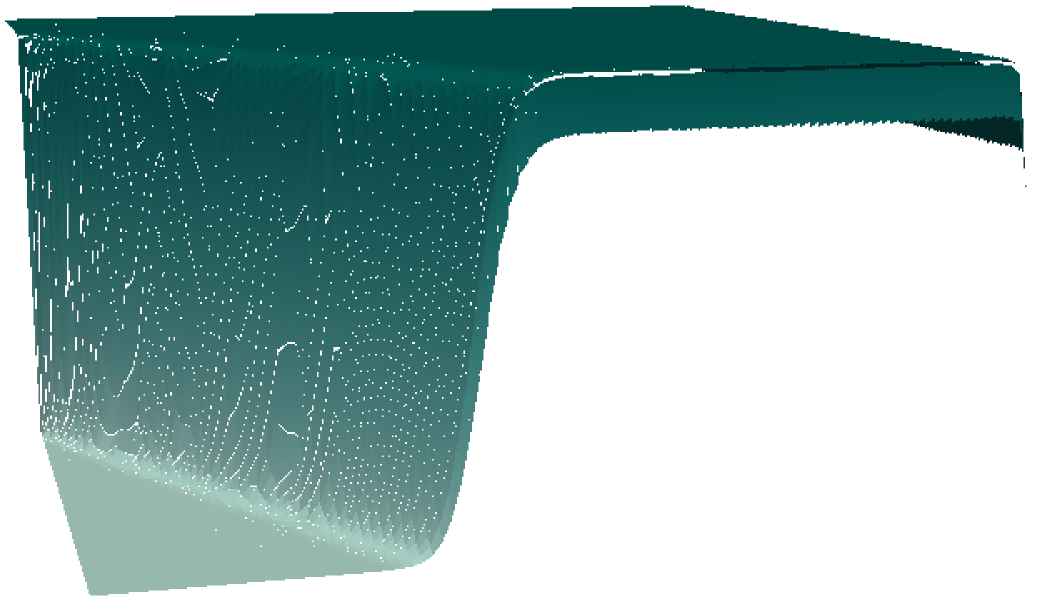}}%
	\caption{Solution of problem \eqref{eq-sharpl} using the smooth version of the stabilization. An unstructured mesh of $h\approx 10^{-2}$ have been used. }\label{fig-layer-uns}
\end{figure}

\begin{figure}
	\centering
	\includegraphics[clip=true, trim = 0cm 0cm 0cm 0cm]{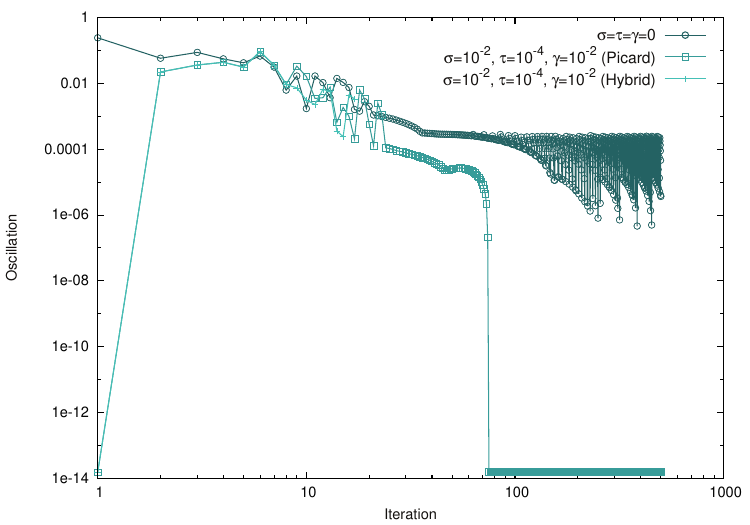}
	\caption{Maximum oscillation as defined in \eqref{eq-osc} for each nonlinear iteration performed for solving \eqref{eq-sharpl}. The results for both a smooth and a non-smooth stabilization, as well as for different nonlinear schemes are depicted. An unstructured mesh of $h\approx 10^{-2}$ has been used for the tests. }\label{fig-maxosc-uns}
\end{figure}

\subsection{Three body rotation}
\label{sec-transientpbm4}
Finally, we want to test the DMP-preservation and LED property of the method for transient problems. To this end, we use the classical three body rotation test. We solve the $2$D transport equation \eqref{eq-strongform4} in $\Omega = [0,1]\times[0,1]$ with $\mu=0$, $\betab=(-2\pi (y-0.5), 2\pi (x-0.5))$. The initial solution is given in \cite{dmitri_kuzmin_guide_2010} and its interpolation in a mesh of $200\times 200$ bilinear elements is displayed in \Fig{threebody0}. 

\begin{figure}[h]
	\centering
	\subfigure[2D]{\includegraphics[clip=true, trim = 10cm 5cm 10cm 5cm, width=6cm]{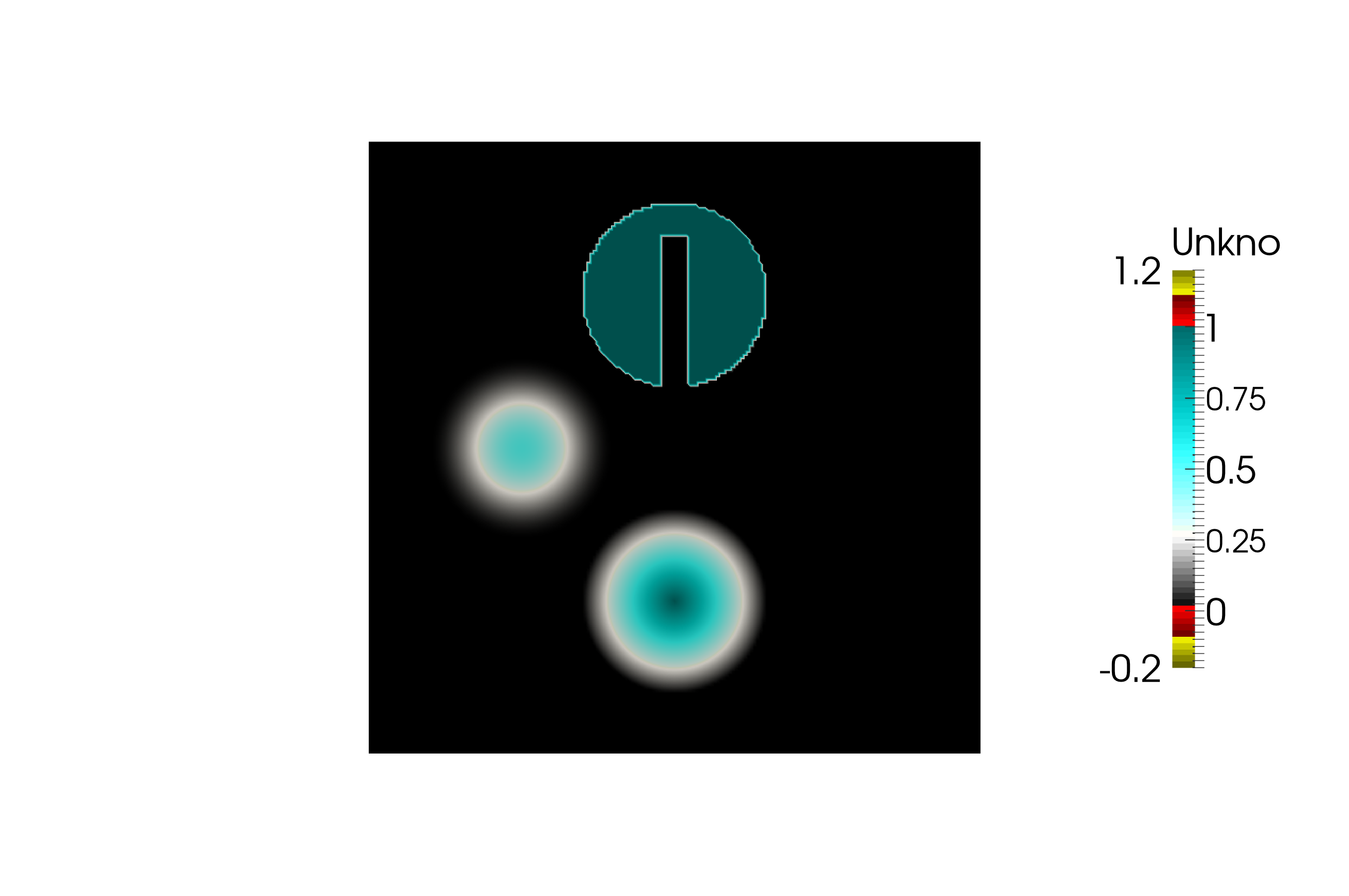}}%
	\subfigure[3D]{\includegraphics[clip=true, trim = 9cm 5cm 7.5cm 0cm, width=6cm]{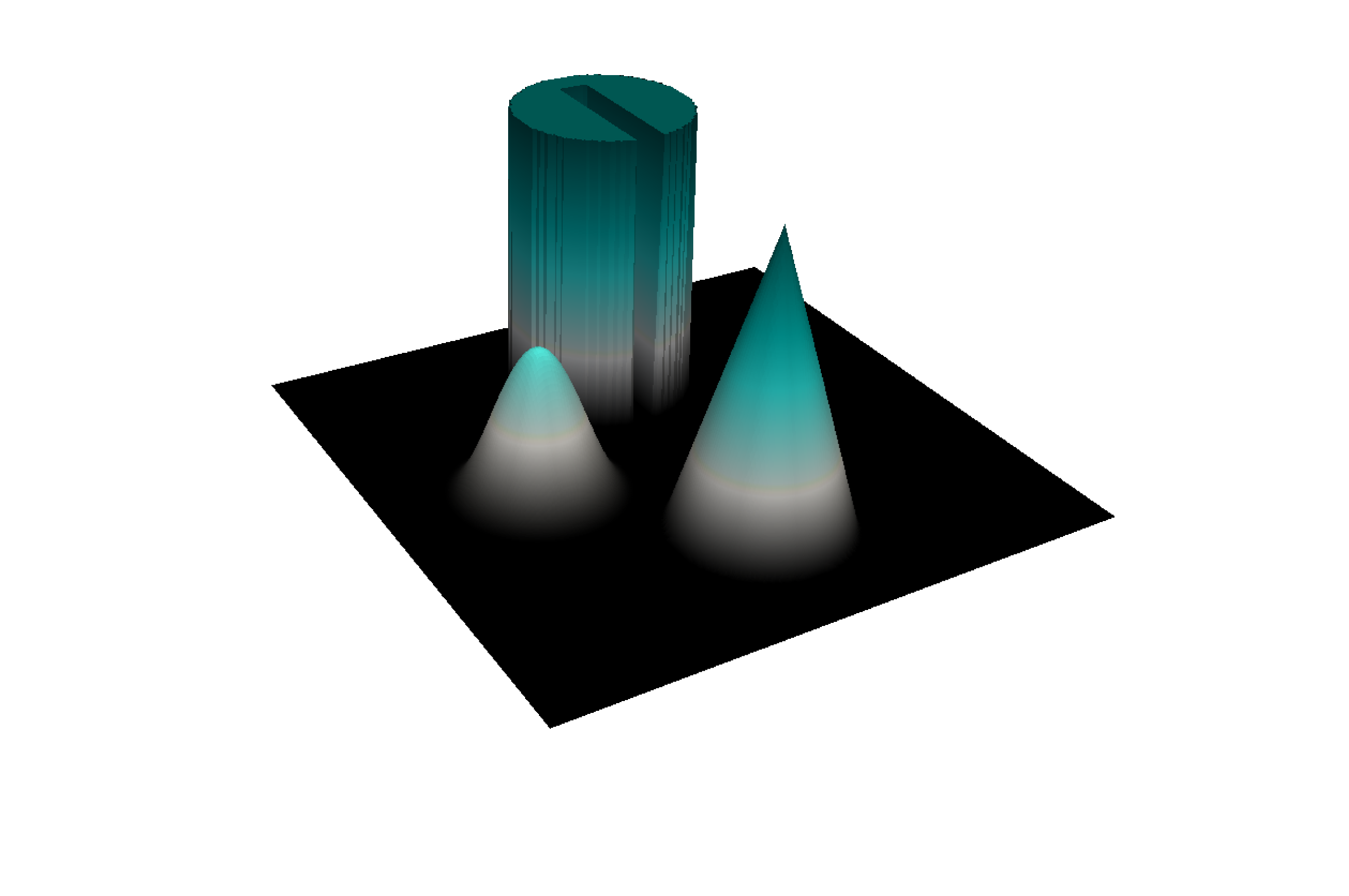}}%
	\subfigure[Legend]{\includegraphics[clip=true, trim = 35cm 5cm 2cm 8.5cm, width=2cm]{3body_initial_2d.png}\label{fig-threebodylegend}}%
	\caption{Initial solution of three body rotation test problem discretized in a $200\times 200\,Q_1$ mesh. }
	\label{fig-threebody0}
\end{figure}

The solution to the transport problem is simply a translation in the direction of the convection. In this case, the initial solution rotates counterclockwise and the final solution is computed at $T=1$ after one complete round. The idea is to compare the initial and the final values to see how dissipative the method is and, at the same time, check the values of the maximum oscillation in each time step to see if there is any violation of the DMP. In order to do so, we have used the color map plotted in \Fig{threebodylegend}, which takes colors from black to green in the interval $[0,1]$ but it uses shades of red in $[-0.1,0)\cup(1,1.1]$ and shades of yellow in $[-0.2,0.1)\cup(1.1,1.2]$. This way, it is  easy to identify the violations of the global DMP of the problem.

The solution is computed with the dG method without any stabilization,
the smoothed stabilized method (in which we consider both $\gamma=0$
and $\gamma=10^{-2}$), and the stabilized method without smoothing
($\sigma=\tau=\gamma=0$). The smoothed case with $\gamma=0$ has been
considered to reduce the mass lumping activation without spoiling the
convergence of the nonlinear iterations. As we will see in the
numerical results, although the method is less dissipative, the final
results do not differ much (see \Fig{threebody-smooth2d} and
\Fig{threebody-smooth02d}).  We recall that, for the integration in
time, a weighted mass lumping is used and the parameter $Q$ is set to
$Q=10$ in the smoothed case and $Q=+\infty$ in the non-smoothed, to
minimize the phase error induced by the lumping of the matrix. We have
run the test in a mesh of $200\times200$ bilinear elements and used
$\ntsteps=2000$ time steps, with nonlinear tolerance of $5\cdot10^{-4}$ and a
maximum of $50$ iterations. The results are plotted in
\Fig{threebody4_sol}. We have also plotted the maximum oscillation in
time in \Fig{threebody4_osc}. Actually, instead of printing the maximum
oscillation at each time step we depict the mean value of the maximum 
in each bunch of 10 time steps, for improving the result visualization
and analysis.
\begin{figure}
	\centering
	\subfigure[\scriptsize $2$D. No stabilization]{\includegraphics[clip=true,trim = 10cm 5cm 10cm 5.5cm, width=6cm]{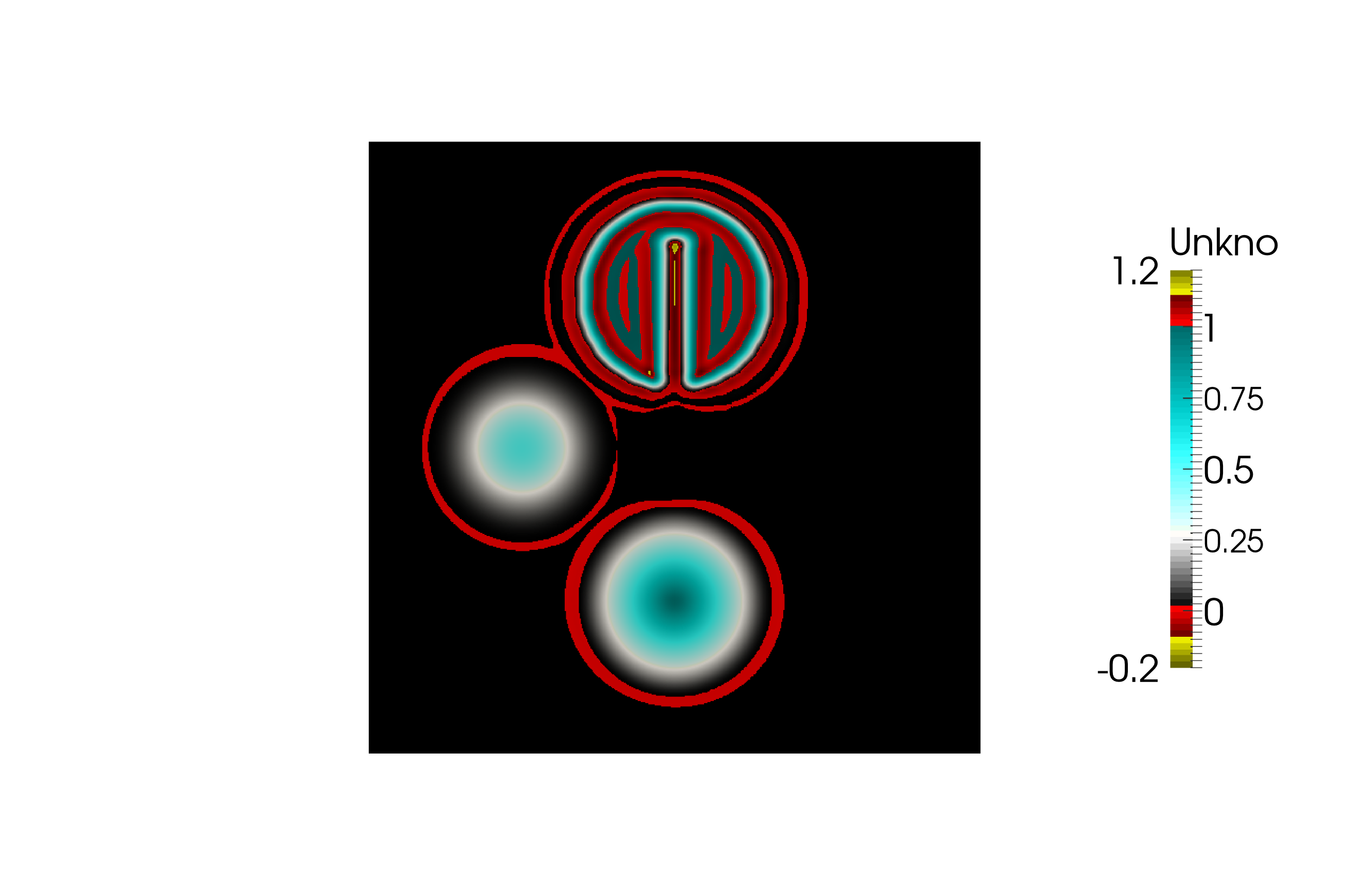}}%
	\subfigure[\scriptsize $3$D.  No stabilization]{\includegraphics[clip=true,trim = 9cm 5cm 7.5cm 0cm, width=6cm]{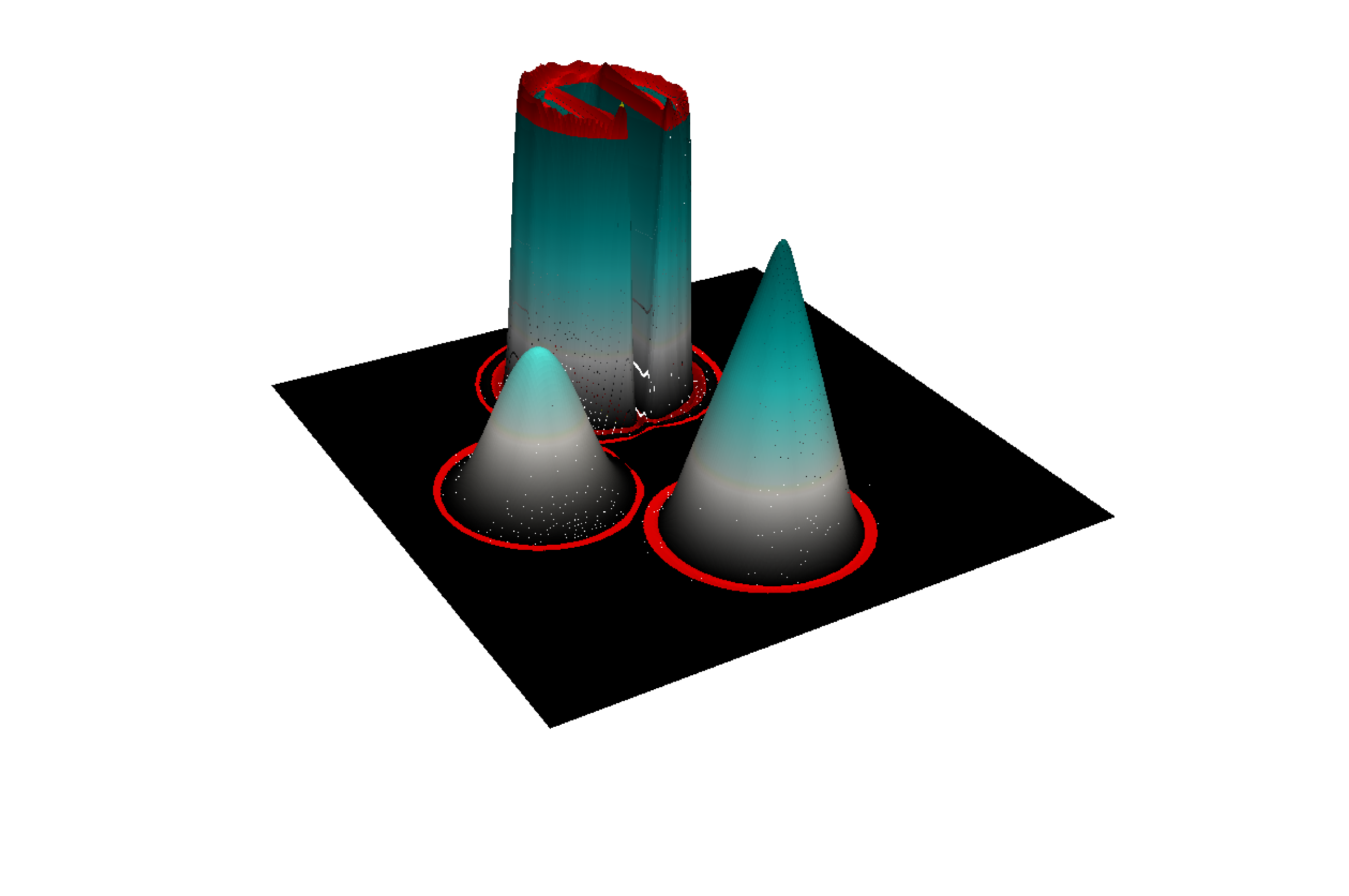}}\\
	
	\subfigure[\scriptsize $2$D. $\smoothperturbc=10^{-2}$; $\smoothalphac=10^{-4}$; $\smoothquotc=10^{-2}$; $Q=10$.]{\includegraphics[clip=true, trim = 10cm 5cm 10cm 5.5cm, width=6cm] {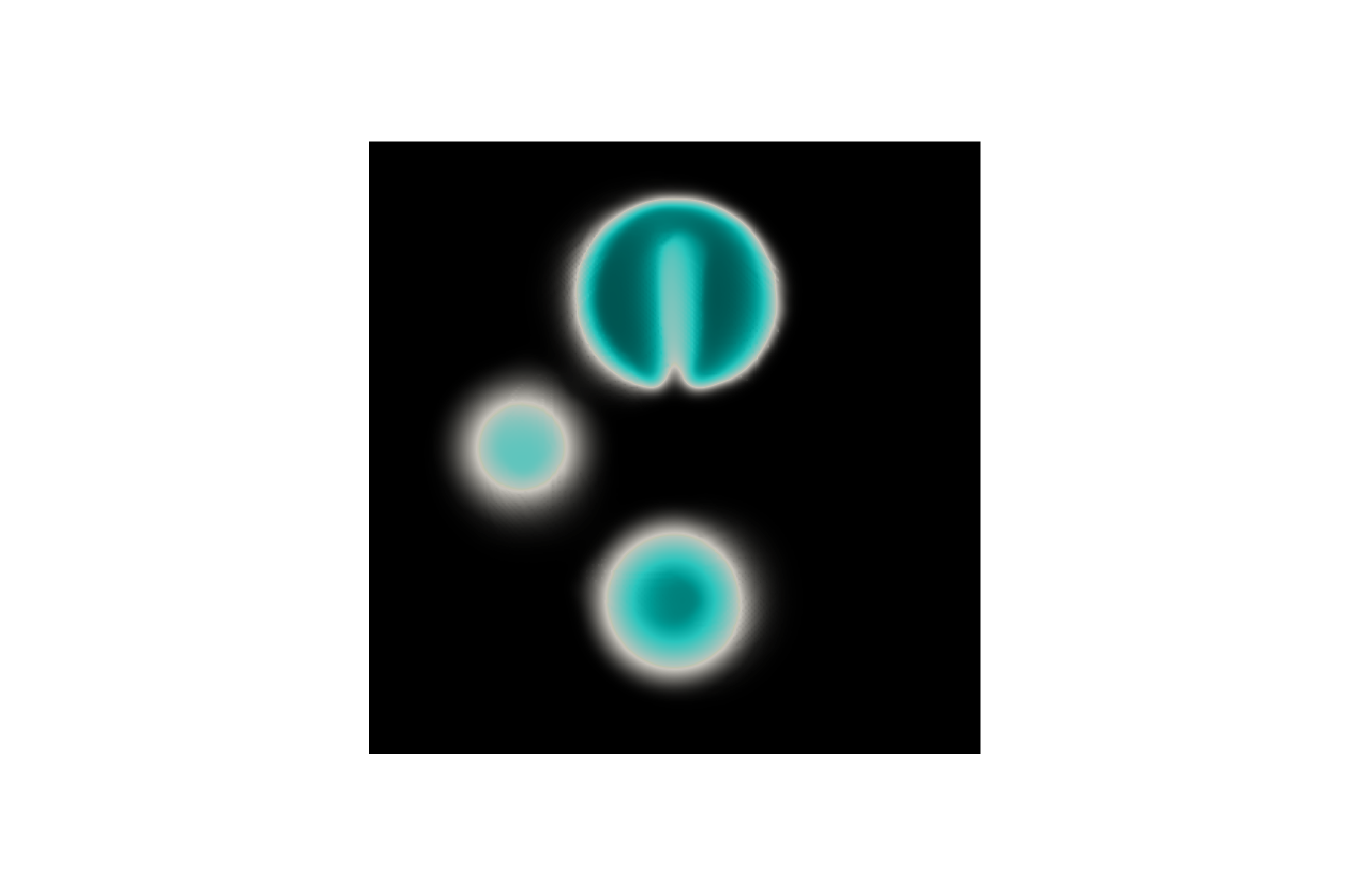}}%
	\subfigure[\scriptsize $3$D. $\smoothperturbc=10^{-2}$; $\smoothalphac=10^{-4}$; $\smoothquotc=10^{-2}$; $Q=10$.]{\includegraphics[clip=true, trim = 9cm 5cm 7.5cm 3cm, width=6cm] {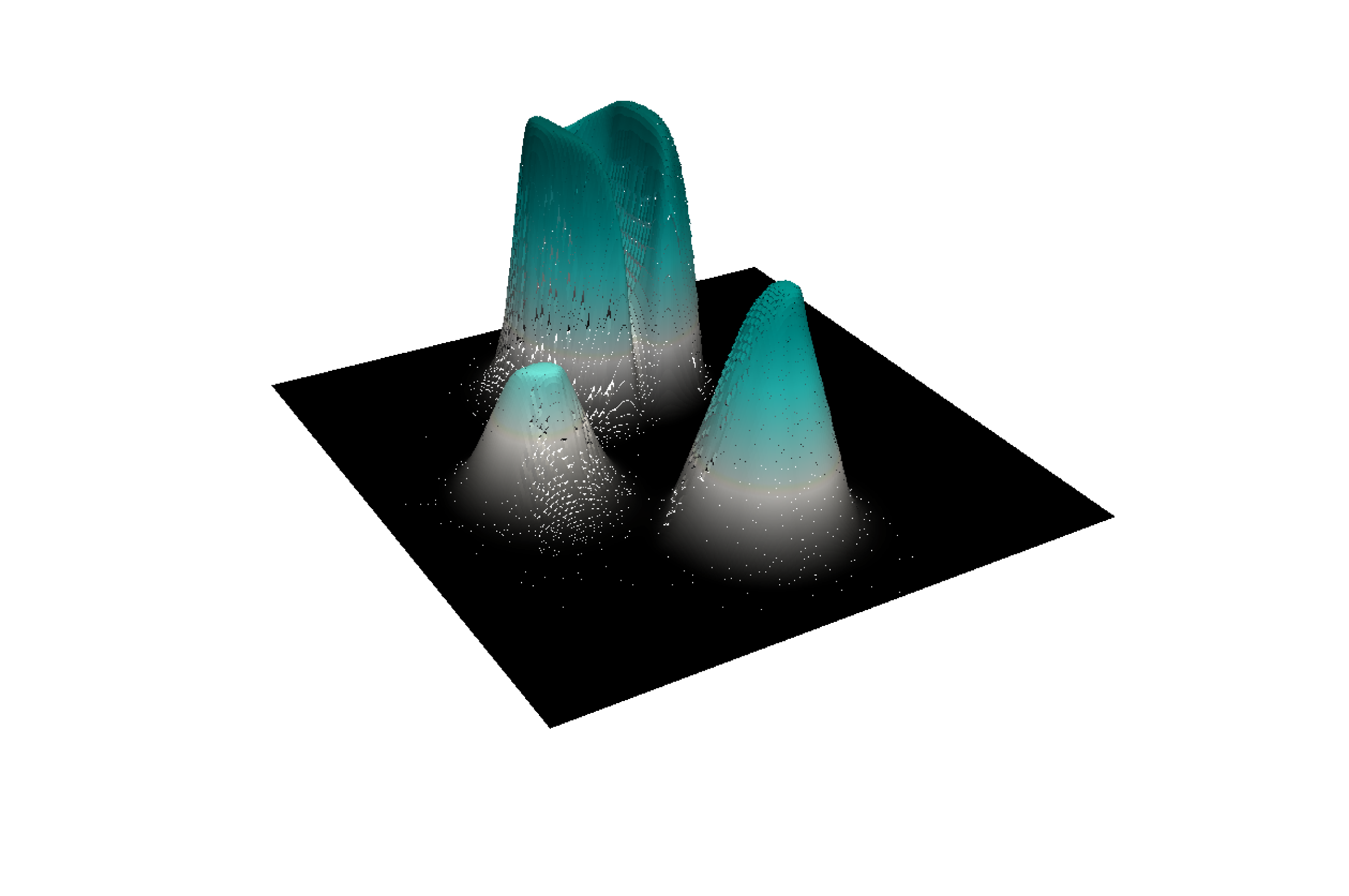}\label{fig-threebody-smooth2d}}\\
	
	\subfigure[\scriptsize $2$D. $\smoothperturbc=10^{-2}$; $\smoothalphac=10^{-4}$; $\smoothquotc=0$; $Q=10$.]{\includegraphics[clip=true, trim = 10cm 5cm 10cm 5.5cm, width=6cm] {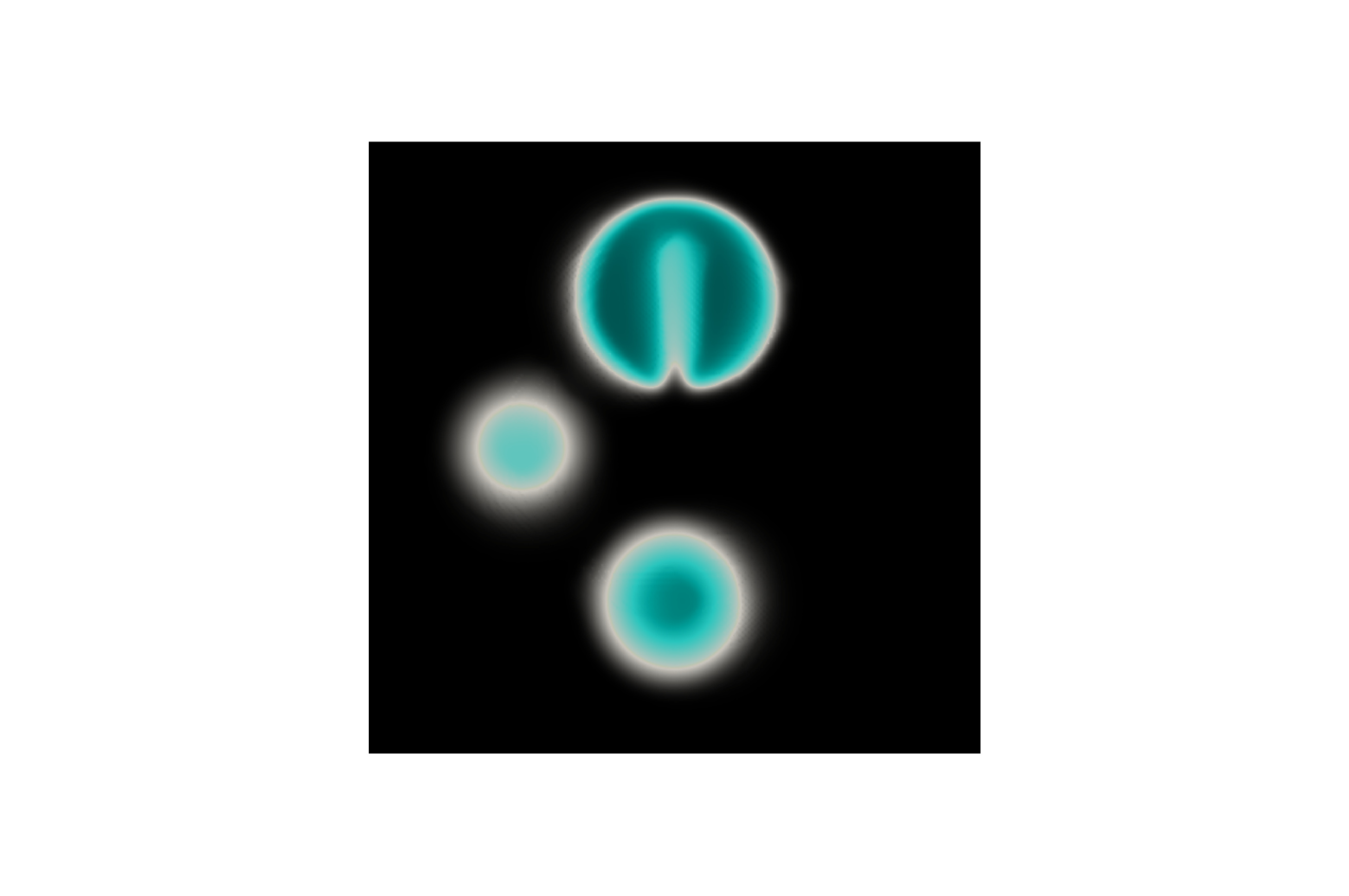}}%
	\subfigure[\scriptsize $3$D. $\smoothperturb=10^{-2}$; $\smoothalphac=10^{-4}$; $\smoothquotc=0$; $Q=10$.]{\includegraphics[clip=true, trim = 9cm 5cm 7.5cm 3cm, width=6cm] {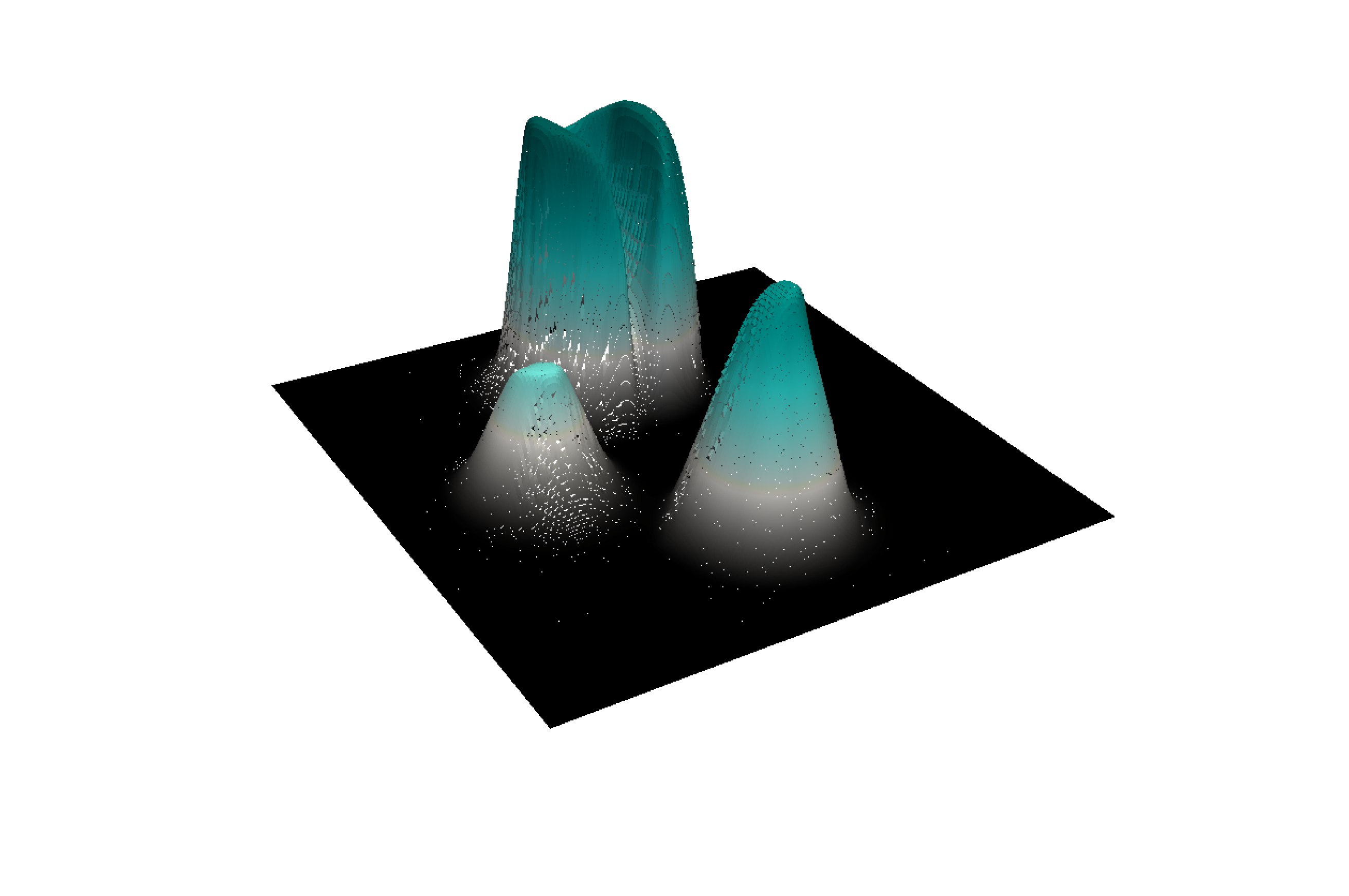}\label{fig-threebody-smooth02d}}\\
	
	\subfigure[\scriptsize $2$D. $\smoothperturbc=\smoothalphac=\smoothquotc=0$. NC.]{\includegraphics[clip=true, trim = 10cm 5cm 10cm 5.5cm, width=6cm]{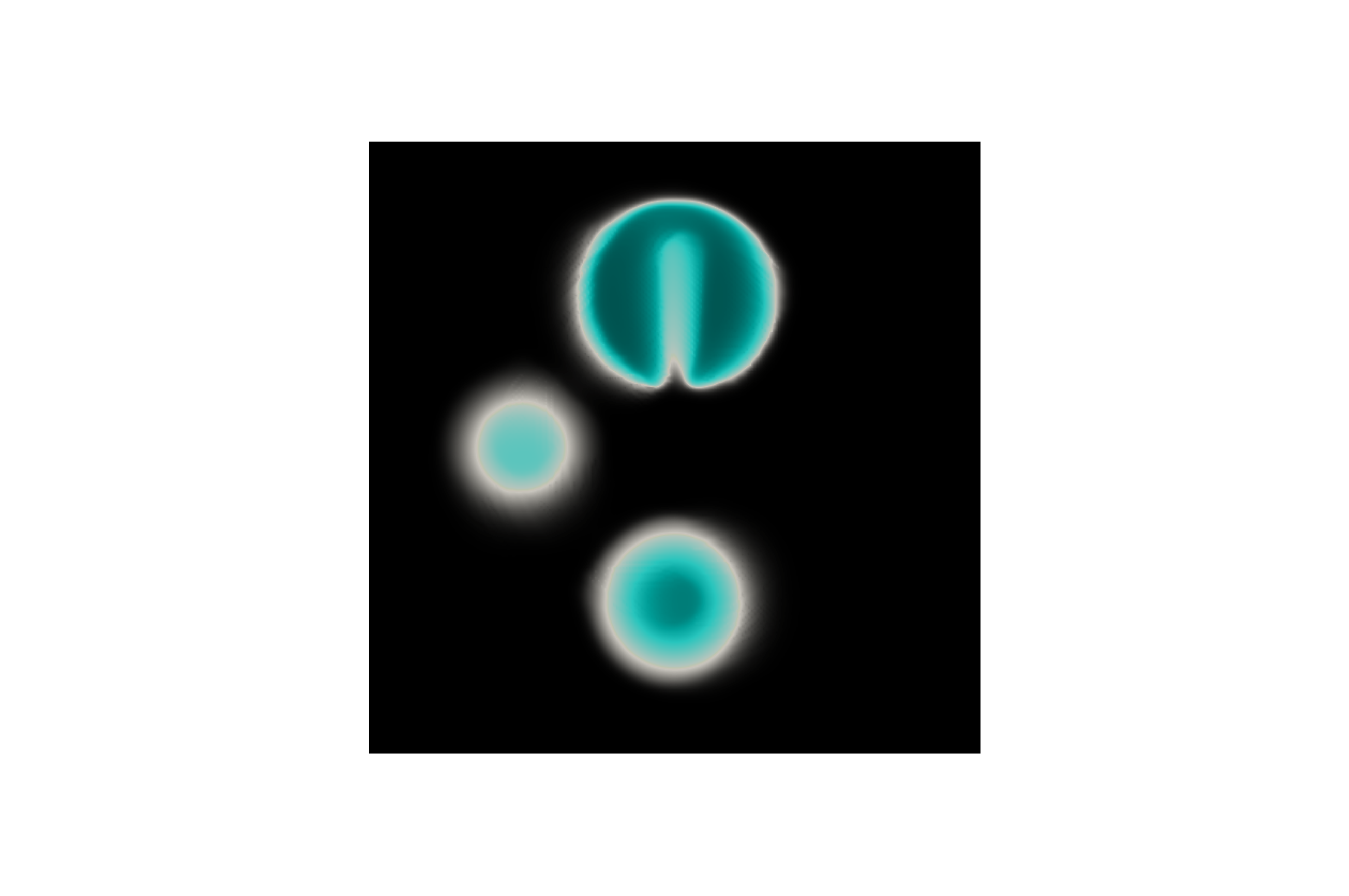}	}
	\subfigure[\scriptsize  $3$D. $\smoothperturbc=\smoothalphac=\smoothquotc=0$. NC.]{\includegraphics[clip=true, trim = 9cm 5cm 7.5cm 3cm, width=6cm]{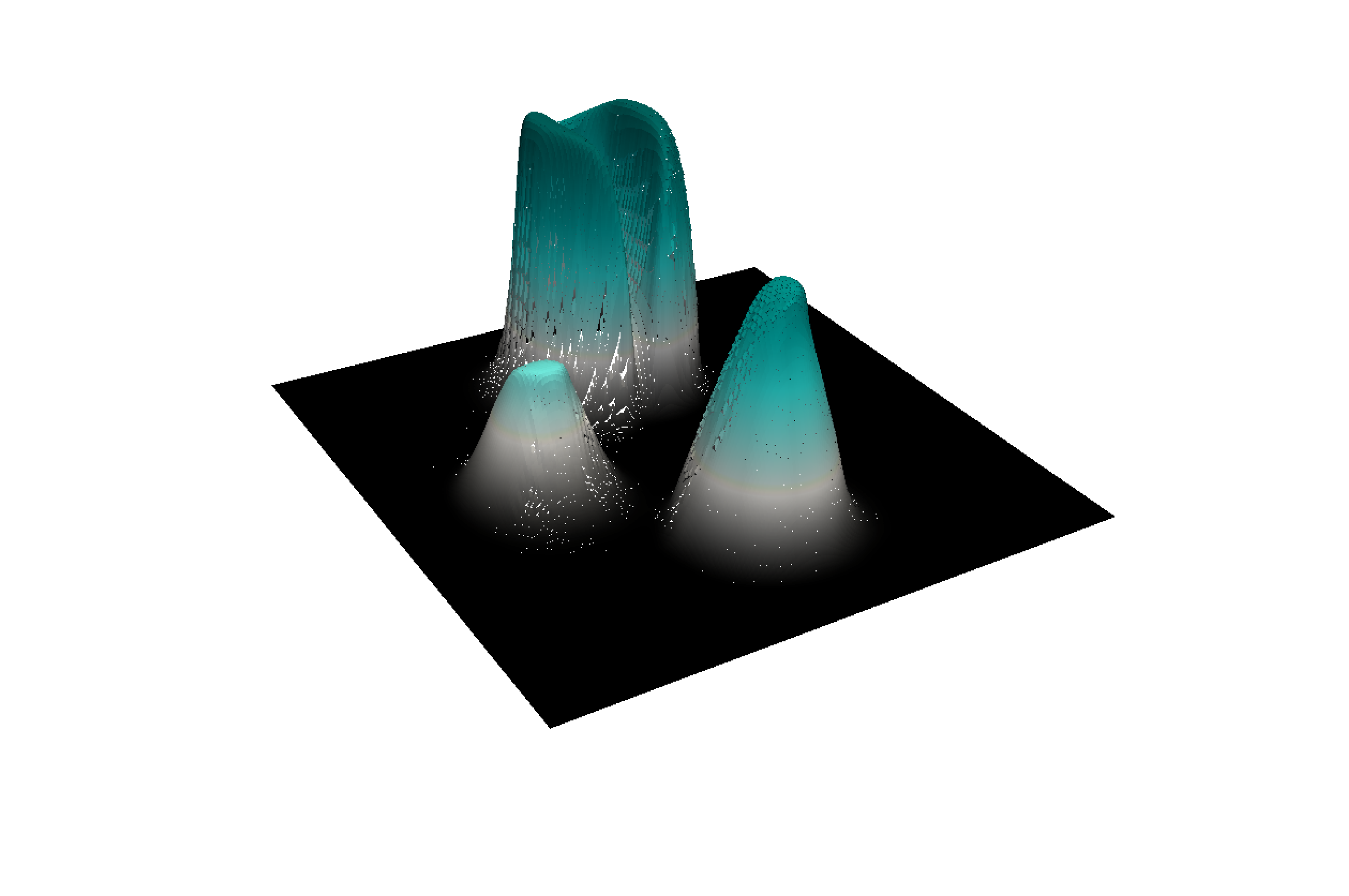}	}
	\caption{Solution of three body rotation test for different combinations of the stabilization parameters. A $200\times 200\,Q_1$ mesh and Crank-Nicolson time integration with $\ntsteps=2000$ have been used. }\label{fig-threebody4_sol}
\end{figure}
\begin{figure}
	\centering
	\includegraphics[clip=true]{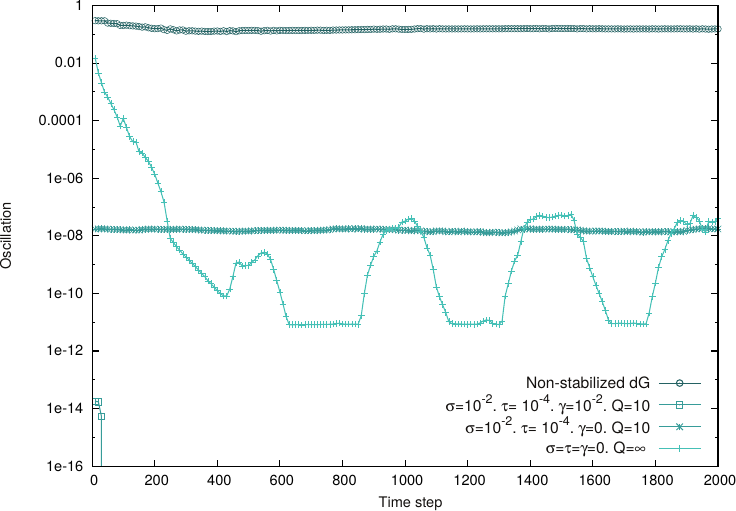}
	\caption{Oscillation values at each time step for the three body rotation test using different choices for the stabilization parameters. A $200\times 200$ mesh and Crank-Nicolson time integration with $\ntsteps=2000$ have been used. }\label{fig-threebody4_osc}
\end{figure}

We can observe in \Fig{threebody4_osc} how, if the method is not stabilized, the oscillations appear from the first iterations and do not decrease in time, being of order $10^{-1}$. On the other hand, one can observe that the stabilized version of the method gives oscillatory results on the first iterations if the method is not smoothed. This is due to the fact that the nonlinear solver is not able to converge in the first time steps and what is plotted is the result after 50 nonlinear iterations. Instead, when smoothing the stabilization, the nonlinear solver converges and the violation of the DMP is of the order of the tolerance for $\gamma=0$. When using $\gamma=10^{-2}$, the method is even more dissipative (as expected since the greater the value of $\gamma$ the more the shock detector is activated) and the DMP is only violated up to machine precision. We want to point out that the appearance of jumps on the smooth part of the figures, such as the cone, is due to the lumping of the mass matrix. When avoiding the mass lumping, those jumps disappear, but the DMP is then violated.

\section{Conclusions}\label{sec-conclusions}

In this work we have designed a method that fulfills the DMP property
for the steady multidimensional transport problem when a dG space
discretization is used. In the transient case, together with implicit time
stepping, the method enjoys the LED property. The original scheme is
stabilized with an artificial diffusion graph-Laplacian operator. The
edge diffusion is only activated on troublesome regions based on a shock
detector that relies on the jumps of the gradient of the solution
around the nodes of the mesh, in order to minimize the smearing of the
solution and improve the solution in smooth regions. We provide a set
of conditions to be satisfied by the stabilized formulation to enjoy
the DMP (and LED) properties, and designed an artificial edge
viscosity to fulfill these conditions. The results hold for arbitrary
meshes and space dimensions. The resulting method is proved to be
DMP-preserving, LED, linearity-preserving, and Lipschitz
continuous. We have also proved the existence of solutions. However,
the method is stil highly nonlinear and it is hard to attain nonlinear
solver convergence. Thus, we propose a smooth version of the scheme
that is twice differentiable and still enjoys DMP and LED
properties. We provide a set of numerical experiments to check the
features proved in the theoretical analysis, and to show the
improvement in terms of computational cost due to the combination of
the smooth version of the scheme with a Newton nonlinear solver with
line search.

Some interesting future work might be to extend these features to
approximations of much more involved equations such as Euler or
compressible Navier-Stokes problems.

\section*{Acknowledgments}

SB gratefully acknowledges the support received from the Catalan Government through the ICREA Acad\`emia Research Program. JB gratefully acknowledges the support received from ''la Caixa'' Foundation through its PhD scholarship program. AH gratefully acknowledges the support received from the Catalan Government through a FI fellowship. We acknowledge the financial support to CIMNE via the CERCA Programme  / Generalitat de Catalunya.

\bibliographystyle{siam}
\bibliography{dG_DMP_references}

\end{document}